\documentclass{amsart}	
\usepackage{amsmath,amssymb,amsthm,amscd}
\usepackage{graphicx,subfigure}
\usepackage[dvipsnames]{xcolor}
\usepackage{pb-diagram} 
\usepackage{tabularx}
\usepackage{multirow}
\usepackage{longtable}
\usepackage{hyperref}
\usepackage[foot]{amsaddr}
\usepackage{enumitem}
\usepackage{comment}
\usepackage{geometry}
\usepackage{microtype}
\newcommand{\negphantom}[1]{\settowidth{\dimen0}{#1}\hspace*{-\dimen0}}
\newcommand{\mybox}[1]{\raisebox{0pt}[2.3ex][0ex]{$#1$}}
\newcommand\Tspace[1]{\rule{0pt}{#1 ex}}  
\newcommand\Bspace[1]{\rule[-#1 ex]{0pt}{#1ex}}

\newcommand{\blue}[1]{\bgroup\color{blue}{#1}\egroup}
\newcommand{\red}[1]{\bgroup\color{red}{#1}\egroup}
\newcommand{\green}[1]{\bgroup\color{OliveGreen}{#1}\egroup}
\newcommand{\plum}[1]{\bgroup\color{Plum}{#1}\egroup}
\newcommand{\cel}[1]{\bgroup\color{RoyalBlue}{#1}\egroup}

\newtheorem{theorem}{Theorem}[section]

\newtheorem{lemma}[theorem]{Lemma}
\newtheorem{definition}[theorem]{Definition}
\theoremstyle{remark} 
\newtheorem{remark}[theorem]{Remark}

\numberwithin{equation}{section}

\newcommand\size{0.65}
\newcommand\sizeC{0.75}

\newcommand{\al}{\alpha}

\newcommand{\de}{\delta}

\renewcommand\epsilon{\varepsilon}
\newcommand{\ga}{\gamma}
\newcommand{\ka}{\kappa}
\newcommand{\vka}{\varkappa}
\newcommand{\la}{\lambda}
\newcommand{\om}{\omega}
\newcommand{\si}{\sigma}
\newcommand{\te}{\theta}
\newcommand{\vp}{\varphi}

\newcommand{\De}{\Delta}

\newcommand{\La}{\Lambda}

\newcommand{\Om}{\Omega}

\newcommand{\hnu}{\hat\nu}

\newcommand{\cA}{{\mathcal A}}

\newcommand{\cE}{{\mathcal E}}

\newcommand{\cK}{{\mathcal{K}}}

\newcommand{\cO}{{\mathcal O}}

\newcommand{\cR}{{\mathcal R}}
\newcommand{\cS}{{\mathcal S}}

\newcommand{\cX}{{\mathcal X}}

\newcommand{\cW}{{\mathcal W}}

\newcommand{\RR}{{\mathbb R}}
\newcommand{\CC}{{\mathbb C}}

\newcommand{\TT}{{\mathbb T}}
\newcommand{\ZZ}{{\mathbb Z}}

\newcommand\fracC{\mathfrak{C}}
\newcommand\fracCDe{\mathfrak{C}_\De}

\newcommand\fracCDeK{\mathfrak{C}_\DeK}
\newcommand\fracCDteK{\mathfrak{C}_{\scalebox{\size}{$\De\DteK$}}}
\newcommand\fracCDvpK{\mathfrak{C}_{\scalebox{\size}{$\De\DvpK$}}}
\newcommand\fracCDteKT{\mathfrak{C}_{\scalebox{\size}{$\De(\DteK)^\ttop$}}}
\newcommand\fracCDvpKT{\mathfrak{C}_{\scalebox{\size}{$\De(\DvpK)^\ttop$}}}

\newcommand\fracCW{\mathfrak{C}_{\scalebox{\size}{$\DeW$}}}
\newcommand\fracCWT{\mathfrak{C}_{\scalebox{\size}{$\De\WT$}}}

\newcommand\fracCtU{\mathfrak{C}_{\scalebox{\size}{$\tU$}}}
\newcommand\fracCNO{\mathfrak{C}_{\scalebox{\size}{$\De\NO$}}}
\newcommand\fracCNOT{\mathfrak{C}_{\scalebox{\size}{$\De\NOT$}}}

\newcommand\fracCDeB{\mathfrak{C}_{\DeB}}
\newcommand\fracCDeinvaverS{\mathfrak{C}_{\scalebox{\size}{$\De\aver{\S}^{\!-1}$}}}
\newcommand\fracCDela{\mathfrak{C}_{\Dela}}
\newcommand\fracCDeinvla{\mathfrak{C}_{\Deinvla}}
\newcommand\bi{\boldsymbol{i}}
\newcommand\bJ{\boldsymbol{J}}
\newcommand\bg{\boldsymbol{g}}
\newcommand{\abs}[1]{|{#1}|}
\newcommand{\Abs}[1]{\left|{#1}\right|}
\newcommand{\norm}[1]{\|{#1}\|}
\newcommand{\Norm}[1]{\left\|{#1}\right\|}
\newcommand{\aver}[1]{{\langle{#1}\rangle}}
\newcommand{\Max}[1]{{\max{\left\{#1\right\}}}}

\newcommand\hatal{\hat\al}
\newcommand\hatom{\hat\om}

\newcommand{\tphi}{\tilde\phi}
\newcommand{\phiK}{\phi_T\comp(\K,\id)}
\newcommand{\phiT}{\phi_T}
\newcommand{\DzphiT}{{\Dif_z}\phi_T}
\newcommand{\DzphiTT}{{(\Dif_z}\phi_T)^\ttop}
\newcommand{\DDzphiT}{{\Dif^2_z}\phi_T}

\newcommand{\DzphiK}{\Dif_z\phi_T\comp(\K,\id)}
\newcommand{\DzphiKT}{\big(\Dif_z\phi_T\comp(\K,\id)\big)^\ttop}

\newcommand{\DeDephiT}{{\De^2\phi_T}}

\newcommand{\DeDzphiT}{{\De\DzphiT}}
\newcommand\angles{{\te,\vp}}
\newcommand\barangles{{\bar\te,\bar\vp}}
\newcommand\invla{\la^{\!-1}}

\newcommand\id{\mathrm{id}}

\newcommand\dif{{\rm d}}
\newcommand\Dif{{\rm D}}

\newcommand\pd{ \partial}

\newcommand{\Lie}[1]{\mathfrak{L}_{#1}}
\newcommand{\R}[1]{\mathfrak{R}_{#1}}

\newcommand{\CR}{c_{\mathfrak{R}}}

\newcommand\Lop{{\mathfrak{L}}}
\newcommand\Rop{{\mathfrak{R}}}

\newcommand\Romal{{\cR}_{\om\al}\mspace{1.5mu}}
\newcommand\Ral{{\cR}_\al\mspace{1.5mu}}
\newcommand{\comp}{{\!\:\circ\!\:}}

\newcommand\ttop{{\!\top\!}}

\newcommand\K{K}
\newcommand\KO{{\K_0}}
\newcommand\W{{W}}
\newcommand\WT{{W^\ttop}}

\newcommand\DK{{\Dif\! \K}}
\newcommand\DteK{\Dif_\te\K}
\newcommand\DvpK{\Dif_\vp\K}
\newcommand\DteKT{({\Dif_\te\K})^{\ttop}}
\newcommand\DvpKT{({\Dif_\vp \K})^{\ttop}}

\newcommand\cXT{{\cX^\ttop}}
\renewcommand\L{{L}}
\newcommand\LT{{\L}^\ttop}
\newcommand\N{{N}}
\newcommand\NT{{\N}^\ttop}
\newcommand\NO{\hat N}
\newcommand\NOT{\NO^\ttop}
\newcommand\B{{B}}
\newcommand\BT{{B^\ttop}}

\newcommand\hS{\hat S}
\newcommand\tS{S}
\newcommand\hSone{\hat S_1}

\renewcommand\S{\cS}
\newcommand\A{{A}}
\newcommand\AT{{A^\ttop}}
\newcommand\Aone{A_1}
\newcommand\Atwo{A_2}
\newcommand\Athree{A_3}
\newcommand\Afour{A_4}

\renewcommand\P{{P}}
\newcommand\PT{{P^\ttop}}
\newcommand\hP{{\hat \P}}
\newcommand\hPT{{\hP^\ttop}}

\newcommand\LaT{{\La^\ttop}}
\newcommand\hLa{{\La}_{\scalebox{\size}{$\hP$}}}
\newcommand\tLa{{\La_{\scalebox{\size}{$\P$}}}}

\newcommand\G{G}
\newcommand\DG{{\rm D}\G}
\renewcommand\DJ{{\rm D}\J}
\newcommand\DJT{{(\rm D}\J)^\ttop}
\newcommand\GK{\G\comp\K}

\newcommand\GL{G_\L}

\newcommand\barK{\bar\K}
\newcommand\DtebarK{\Dif_\te\barK}
\newcommand\DtebarKT{(\Dif_\te\barK)^\ttop}
\newcommand\DvpbarK{\Dif_\vp\barK}
\newcommand\DvpbarKT{(\Dif_\vp\barK)^\ttop}
\newcommand\barW{\bar\W}
\newcommand\barWT{\bar\W^\ttop}
\newcommand\barL{\bar\L}

\newcommand\GbarK{\G\comp\barK}
\newcommand\GbarL{G_{\bar\L}}
\newcommand\barNO{\hat{\bar{N}}}
\newcommand\barNOT{\hat{\bar{N}}^\ttop}
\newcommand\barB{\bar \B}
\newcommand\barhS{\hat{\bar{S}}}

\newcommand\barla{\bar\la}
\newcommand\barinvla{{\barla}^{-1}}
\newcommand\DeK{{\De \K}}
\newcommand\Dela{{\De \la}}
\newcommand\Deinvla{{\De \invla}}
\newcommand\DeW{{\De\W}}
\newcommand\DeOm{{\De\Om\hspace{0.2ex}}}

\newcommand\DeL{\De\L}
\newcommand\DeLT{\De\L^\ttop}
\newcommand\DeG{\De \G}
\newcommand\DeJ{\De \J}

\newcommand\DeGL{\De \GL}
\newcommand\DeB{\De \B}

\newcommand\DeNO{\De\NO}
\newcommand\DeNOT{\De\NOT}

\newcommand\DehS{\De\hS}
\newcommand\etaK{\eta_{\scalebox{\size}{$\K$}}}
\newcommand\etaaK{\eta_{\scalebox{\size}{$\K$}}^1}
\newcommand\etabK{\eta_{\scalebox{\size}{$\K$}}^2}
\newcommand\etacK{\eta_{\scalebox{\size}{$\K$}}^3}
\newcommand\etadK{\eta_{\scalebox{\size}{$\K$}}^4}
\newcommand\etacaK{\eta_{\scalebox{\size}{$\K$}}^{31}}
\newcommand\etacbK{\eta_{\scalebox{\size}{$\K$}}^{32}}

\newcommand\xiK{\xi_{\scalebox{\size}{$\K$}}}
\newcommand\xiaK{\xi_{\scalebox{\size}{$\K$}}^1}
\newcommand\xibK{\xi_{\scalebox{\size}{$\K$}}^2}
\newcommand\xicK{\xi_{\scalebox{\size}{$\K$}}^3}
\newcommand\xidK{\xi_{\scalebox{\size}{$\K$}}^4}

\newcommand\etaW{\eta_{\scalebox{\size}{$\W$}}}
\newcommand\etaaW{\eta_{\scalebox{\size}{$\W$}}^1}
\newcommand\etabW{\eta_{\scalebox{\size}{$\W$}}^2}
\newcommand\etacW{\eta_{\scalebox{\size}{$\W$}}^3}
\newcommand\etadW{\eta_{\scalebox{\size}{$\W$}}^4}

\newcommand\xiW{\xi_{\scalebox{\size}{$\W$}}}
\newcommand\xiaW{\xi_{\scalebox{\size}{$\W$}}^1}
\newcommand\xibW{\xi_{\scalebox{\size}{$\W$}}^2}
\newcommand\xicW{\xi_{\scalebox{\size}{$\W$}}^3}
\newcommand\xidW{\xi_{\scalebox{\size}{$\W$}}^4}

\newcommand\bOm{\boldsymbol{\om}}
\newcommand\bal{\boldsymbol{\al}}
\newcommand\J{{J}}

\newcommand\OmO{{\Omega_{0}}}

\newcommand\OmL{{\Omega_{L}}}
\newcommand\OmDK{{\Omega_{\scalebox{\size}{$\DK$}}}}
\newcommand\OmDKcX{{\Omega_{\scalebox{\size}{$\DK\!\cX$}}}}
\newcommand\OmDKW{{\Omega_{\scalebox{\size}{$\DK\!\W$}}}}
\newcommand\OmcXDK{{\Omega_{\scalebox{\size}{$\cX\DK$}}}}
\newcommand\OmcX{{\Omega_{\scalebox{\size}{$\cX$}}}}
\newcommand\OmcXW{{\Omega_{\scalebox{\size}{$\cX\W$}}}}
\newcommand\OmWDK{{\Omega_{\scalebox{\size}{$\W\DK$}}}}
\newcommand\OmWcX{{\Omega_{\scalebox{\size}{$\W\!\cX$}}}}
\newcommand\OmW{{\Omega_{\scalebox{\size}{$\W$}}}}
\newcommand\OmK{{\Omega\comp\K}\mspace{1.5mu}}
\newcommand\OmphiK{{\Omega\comp\phiK}\mspace{1.5mu}}
\newcommand\OmbarK{{\Omega\comp\K\comp\Romal}\mspace{1.5mu}}

\newcommand\Da{{\Dif a}}
\newcommand\DDa{{\Dif^2\! a}}
\newcommand\DaT{{(\Dif a)^{\!\top}}}

\newcommand\DOm{{\Dif \Om}}

\newcommand\U{{\mathcal{U}}}

\newcommand\tU{{\tilde\U}}

\newcommand\E{\cE}
\newcommand\Esym{E_{\scalebox{\size}{\rm symp$\P$}}} 
\newcommand\Ered{E_{\scalebox{\size}{\rm red$\P$}}} 
\newcommand\ElinK{E^\K_{\scalebox{\size}{\rm lin}}} 
\newcommand\ElinW{E^\W_{\scalebox{\size}{\rm lin}}} 
\newcommand\EL{E_\L}
\newcommand\EK{{E_{\scalebox{\size}{$\K$}}}}
\newcommand\EKnew{{E_{\scalebox{\size}{$\barK$}}}}
\newcommand\EW{{E_{\scalebox{\size}{$\W$}}}}
\newcommand\EWnew{{E_{\scalebox{\size}{$\barW$}}}}
\newcommand\tildeEW{{\widetilde{E_{\scalebox{\size}{$\W$}}}}}
\newcommand\EcX{E_\cX}
\newcommand\EinvhP{E_{\scalebox{\size}{{\rm inv}$\scriptstyle\hP$}}}
\newcommand\EsymhLa{E_{\scalebox{\size}{{\rm symp}$\hLa$}}}
\newcommand\EsymA{E_{\scalebox{\size}{{\rm sym}$\A$}}}
\newcommand\EsyminvLahS{{E_{\scalebox{\size}{{\rm
sym}$\La^{\!-1}\!\hS$}}}}

\newcommand\hEsym{E_{\scalebox{\size}{\rm symp$\hP$}}}
\newcommand\hEred{E_{\scalebox{\size}{\rm red$\hP$}}}

\newcommand\cteOm{c_{\scalebox{\size}{$\Om$}}}
\newcommand\cteDOm{c_{\scalebox{\size}{$\DOm$}}}

\newcommand\cteDa{c_{\scalebox{\size}{$\Da$}}}
\newcommand\cteDDa{c_{\scalebox{\size}{$\DDa$}}}
\newcommand\cteDaT{c_{\scalebox{\size}{$(\Da)^\ttop$}}}
\newcommand\cteJ{c_{\scalebox{\size}{$\J$}}}
\newcommand\cteJT{c_{\scalebox{\size}{$\J^\ttop$}}}
\newcommand\cteDzphiT{c_{\scalebox{\size}{{$\DzphiT$}}}}
\newcommand\cteDzphiTT{c_{\scalebox{\size}{{$\DzphiTT$}}}}
\newcommand\cteDDzphiT{c_{\scalebox{\size}{{$\DDzphiT$}}}}

\newcommand\cteX{c_{\scalebox{\size}{$X$}}}
\newcommand\cteXT{c_{\scalebox{\size}{$J^\ttop$}}}
\newcommand\cteDzX{c_{\scalebox{\size}{$\Dif_z X$}}}
\newcommand\cteDzXT{c_{\scalebox{\size}{$(\Dif_z X)^\ttop$}}}
\newcommand\cteDDzH{{c_{\scalebox{\size}{$\Dif^2_z H$}}}}

\newcommand\cteG{{c_{\scalebox{\size}{$\G$}}}}
\newcommand\cteDG{c_{\scalebox{\size}{$\DG$}}}
\newcommand\cteDJ{c_{\scalebox{\size}{$\DJ$}}}
\newcommand\cteDJT{c_{\scalebox{\size}{$\DJT$}}}

\newcommand\sigmaDteK{\sigma_{\scalebox{\size}{{${\rm D}_\te K$}}}}
\newcommand\sigmaDteKT{\sigma_{\scalebox{\size}{{$\left({\rm D}_\te K\right)^\ttop$}}}}
\newcommand\sigmaDvpK{\sigma_{\scalebox{\size}{{${\rm D}_\vp K$}}}}
\newcommand\sigmaDvpKT{\sigma_{\scalebox{\size}{{$\left({\rm D}_\vp K\right)^\ttop$}}}}

\newcommand\sigmaW{\sigma_{\scalebox{\size}{{$\W$}}}}
\newcommand\sigmaWT{\sigma_{\scalebox{\size}{{$\W^\ttop$}}}}

\newcommand\sigmala{\sigma_{\scalebox{\size}{{$\la$}}}}
\newcommand\sigmainvla{\sigma_{\scalebox{\size}{{$\la^{\scalebox{0.7}{$-1$}}$}}}}

\newcommand\sigmainvaverS{\sigma_{\scalebox{\size}{$\aver{\S}^{\!-1}$}}}

\newcommand\sigmaNO{\sigma_{\scalebox{\size}{$\NO$}}}
\newcommand\sigmaNOT{\sigma_{\scalebox{\size}{$\NOT$}}}

\newcommand\sigmaB{{\sigma_{\scalebox{\size}{$B$}}}}

\newcommand\Cla{C_\la}
\newcommand\CcX{{C_{\cX}}}
\newcommand\CcXT{{C_{\cXT}}}

\newcommand\CL{{C_{\L}}}
\newcommand\CLT{{C_{\LT}}}

\newcommand\CP{C_\P}
\newcommand\CPT{C_{\P^\ttop}}
\newcommand\ChP{C_\hP}
\newcommand\ChPT{C_{\hP^\ttop}}

\newcommand\ChS{{C_{\hS}}}
\newcommand\ChST{{C_{\hS^\ttop}}}

\newcommand\CinvLahSK{{C_{\scalebox{\sizeC}{$\La^{\!-1}\!\hS$}}^\K}}
\newcommand\CinvLahSW{{C_{\scalebox{\sizeC}{$\La^{\!-1}\!\hS$}}^\W}}

\newcommand\CA{{C_{\A}}}
\newcommand\CAT{{C_{\AT}}}
\newcommand\CEsymAK{{C_{\scalebox{\sizeC}{$\EsymA$}}^\K}}
\newcommand\CEsymAW{{C_{\scalebox{\sizeC}{$\EsymA$}}^\W}}

\newcommand\CN{{C_{\N}}}
\newcommand\CNT{{C_{\NT}}}

\newcommand\CELK{{C_{\scalebox{\sizeC}{$\EL$}}^\K}}
\newcommand\CELW{{C_{\scalebox{\sizeC}{$\EL$}}^\W}}

\newcommand\CELTK{{C_{\scalebox{\sizeC}{$\EL^\ttop$}}^\K}}
\newcommand\CELTW{{C_{\scalebox{\sizeC}{$\EL^\ttop$}}^\W}}

\newcommand\CEsym{{C_{\scalebox{\sizeC}{$\Esym$}}}}
\newcommand\CEsymK{{C_{\scalebox{\sizeC}{$\Esym$}}^{\K}}}
\newcommand\CEsymW{{C_{\scalebox{\sizeC}{$\Esym$}}^{\W}}}

\newcommand\ChEsym{{C_{\scalebox{\sizeC}{$\hEsym$}}}}
\newcommand\ChEsymK{{C_{\scalebox{\sizeC}{$\hEsym$}}^\K}}
\newcommand\ChEsymW{{C_{\scalebox{\sizeC}{$\hEsym$}}^{\W}}}

\newcommand\CEredK{{C_{\scalebox{\sizeC}{$\Ered$}}^{\K}}}
\newcommand\CEredW{{C_{\scalebox{\sizeC}{$\Ered$}}^{\W}}}

\newcommand\ChEredK{{C_{\scalebox{\sizeC}{$\hEred$}}^{\K}}}
\newcommand\ChEredW{{C_{\scalebox{\sizeC}{$\hEred$}}^{\W}}}

\newcommand\CElinKK{{C_{\scalebox{\sizeC}{$\ElinK$}}^{\K}}}
\newcommand\CElinKKW{{C_{\scalebox{\sizeC}{$\ElinK$}}^{\K\W}}}

\newcommand\CElinWKK{{C_{\scalebox{\sizeC}{$\ElinW$}}^{\K\K}}}
\newcommand\CElinWWW{{C_{\scalebox{\sizeC}{$\ElinW$}}^{\W\W}}}
\newcommand\CElinWKW{{C_{\scalebox{\sizeC}{$\ElinW$}}^{\K\W}}}

\newcommand\CEredbbK{{C_{\scalebox{\sizeC}{$\Ered^{22}$}}^{\K}}}

\newcommand\CEredbbW{{C_{\scalebox{\sizeC}{$\Ered^{22}$}}^{\W}}}

\newcommand\ChEredaaK{{C_{\scalebox{\sizeC}{$\hEred^{11}$}}^{\K}}}
\newcommand\ChEredbbK{{C_{\scalebox{\sizeC}{$\hEred^{22}$}}^{\K}}}
\newcommand\ChEredaaW{{C_{\scalebox{\sizeC}{$\hEred^{11}$}}^{\W}}}
\newcommand\ChEredbbW{{C_{\scalebox{\sizeC}{$\hEred^{22}$}}^{\W}}}
\newcommand\ChEredbaK{{C_{\scalebox{\sizeC}{$\hEred^{21}$}}^{\K}}}
\newcommand\ChEredbaW{{C_{\scalebox{\sizeC}{$\hEred^{21}$}}^{\W}}}

\newcommand\CtildeEK{{C_{{\scalebox{\sizeC}{$\tildeEW$}}}^{\K}}}
\newcommand\CtildeEW{{C_{{\scalebox{\sizeC}{$\tildeEW$}}}^{\W}}}

\newcommand\CEK{{C_{\scalebox{\sizeC}{$E_{\bar\K}$}}}}
\newcommand\CEKKK{{C_{\scalebox{\sizeC}{$E_{\bar\K}$}}^{{\K\K}}}}
\newcommand\CEKKW{{C_{\scalebox{\sizeC}{$E_{\bar\K}$}}^{{\K\W}}}}

\newcommand\CEW{{C_{\scalebox{\sizeC}{$E_{\bar\W}$}}}}
\newcommand\CEWKK{{C_{\scalebox{\sizeC}{$E_{\bar\W}$}}^{\K\K}}}
\newcommand\CEWWW{{C_{\scalebox{\sizeC}{$E_{\bar\W}$}}^{\W\W}}}
\newcommand\CEWKW{{C_{\scalebox{\sizeC}{$E_{\bar\W}$}}^{\K\W}}}

\newcommand\COmDK{C_{\scalebox{\sizeC}{$\OmDK$}}}
\newcommand\CLieOmDK{{C_{\scalebox{\sizeC}{$\Lop \OmDK$}}}}

\newcommand\COmLK{{C_{\scalebox{\sizeC}{$\OmL$}}^{\K}}}
\newcommand\COmLW{{C_{\scalebox{\sizeC}{$\OmL$}}^{\W}}}

\newcommand\CEcX{C_{\scalebox{\sizeC}{$\EcX$}}}
\newcommand\CEcXT{C_{\scalebox{\sizeC}{$\EcX^\ttop$}}}

\newcommand\CLieaone{{C_{\scalebox{\sizeC}{$\Lop\OmDKcX$}}}}
\newcommand\CLieatwoK{{C^\K_{\scalebox{\sizeC}{$\Lop^{1\la} \OmDKW$}}}}
\newcommand\CLieatwoW{{C^\W_{\scalebox{\sizeC}{$\Lop^{1\la} \OmDKW$}}}}
\newcommand\CLieathree{{C_{\scalebox{\sizeC}{$\Lop \OmcXDK$}}}}
\newcommand\CLieafiveK{{C^\K_{\scalebox{\sizeC}{$\Lop^{1\la} \OmcXW$}}}}
\newcommand\CLieafiveW{{C^\W_{\scalebox{\sizeC}{$\Lop^{1\la} \OmcXW$}}}}
\newcommand\CLieasixK{{C^\K_{\scalebox{\sizeC}{$\Lop^{1\la} \OmWDK$}}}}
\newcommand\CLieasixW{{C^\W_{\scalebox{\sizeC}{$\Lop^{1\la} \OmWDK$}}}}
\newcommand\CLieasevenK{{C^\K_{\scalebox{\sizeC}{$\Lop^{1\la} \OmWcX$}}}}
\newcommand\CLieasevenW{{C^\W_{\scalebox{\sizeC}{$\Lop^{1\la} \OmWcX$}}}}
\newcommand\Caone{{C_{\scalebox{\sizeC}{$\OmDKcX$}}}}
\newcommand\CatwoK{{C_{\scalebox{\sizeC}{$\OmDKW$}}^\K}}
\newcommand\CatwoW{{C_{\scalebox{\sizeC}{$\OmDKW$}}^\W}}
\newcommand\Cathree{{C_{\scalebox{\sizeC}{$\OmcXDK$}}}}
\newcommand\CafiveK{{C_{\scalebox{\sizeC}{$\OmcXW$}}^{\K}}}
\newcommand\CafiveW{{C_{\scalebox{\sizeC}{$\OmcXW$}}^{\W}}}
\newcommand\CasixK{{C_{\scalebox{\sizeC}{$\OmWDK$}}^\K}}
\newcommand\CasixW{{C_{\scalebox{\sizeC}{$\OmWDK$}}^\W}}
\newcommand\CasevenK{{C_{\scalebox{\sizeC}{$\OmWcX$}}^\K}}
\newcommand\CasevenW{{C_{\scalebox{\sizeC}{$\OmWcX$}}^\W}}

\newcommand\CDe{{C_\De}}
\newcommand\CE{{C_\E}}
\newcommand\CEj{{C_{\E_j}}}
\newcommand\CEjmo{{C_{\E_{j-1}}}}
\newcommand\CEinvhPK{{C_{\scalebox{\sizeC}{$\EinvhP$}}^\K}}
\newcommand\CEinvhPW{{C_{\scalebox{\sizeC}{$\EinvhP$}}^\W}}

\newcommand\CetaoneK{{C_{\scalebox{\sizeC}{$\eta_\K^1$}}}}
\newcommand\CetatwoK{{C_{\scalebox{\sizeC}{$\eta_\K^2$}}}}
\newcommand\CetathreeK{{C_{\scalebox{\sizeC}{$\eta_\K^3$}}}}
\newcommand\CetafourK{{C_{\scalebox{\sizeC}{$\eta_\K^4$}}}}

\newcommand\CaveretathreeK{{C_{\scalebox{\sizeC}{$\aver{\etaK^{3}}$}}}}
\newcommand\CaveretathreeoneK{{C_{\scalebox{\sizeC}{$\aver{\etaK^{31}}$}}}}
\newcommand\CaveretathreetwoK{{C_{\scalebox{\sizeC}{$\aver{\etaK^{32}}$}}}}

\newcommand\CetaoneWK{{C^\K_{\scalebox{\sizeC}{$\eta_\W^1$}}}}
\newcommand\CetatwoWK{{C^\K_{\scalebox{\sizeC}{$\eta_\W^2$}}}}
\newcommand\CetathreeWK{{C^\K_{\scalebox{\sizeC}{$\eta_\W^3$}}}}
\newcommand\CetafourWK{{C^\K_{\scalebox{\sizeC}{$\eta_\W^4$}}}}

\newcommand\CetaoneWW{{C^\W_{\scalebox{\sizeC}{$\eta_\W^1$}}}}
\newcommand\CetatwoWW{{C^\W_{\scalebox{\sizeC}{$\eta_\W^2$}}}}
\newcommand\CetathreeWW{{C^\W_{\scalebox{\sizeC}{$\eta_\W^3$}}}}
\newcommand\CetafourWW{{C^\W_{\scalebox{\sizeC}{$\eta_\W^4$}}}}

\newcommand\CaverxithreeK{{C_{\scalebox{\sizeC}{$\aver{\xi_\K^3}$}}}}
\newcommand\CxiK{{C_{\scalebox{\sizeC}{$\xi_\K$}}}}
\newcommand\CxioneK{{C_{\scalebox{\sizeC}{$\xi_\K^1$}}}}
\newcommand\CxitwoK{{C_{\scalebox{\sizeC}{$\xi_\K^2$}}}}
\newcommand\CxithreeK{{C_{\scalebox{\sizeC}{$\xi_\K^3$}}}}
\newcommand\CxifourK{{C_{\scalebox{\sizeC}{$\xi_\K^4$}}}}

\newcommand\CxiWK{{C^\K_{\scalebox{\sizeC}{$\xiW$}}}}
\newcommand\CxioneWK{{C^\K_{\scalebox{\sizeC}{$\xi_W^1$}}}}
\newcommand\CxitwoWK{{C^\K_{\scalebox{\sizeC}{$\xi_W^2$}}}}
\newcommand\CxithreeWK{{C^\K_{\scalebox{\sizeC}{$\xi_W^3$}}}}
\newcommand\CxifourWK{{C^\K_{\scalebox{\sizeC}{$\xi_W^4$}}}}

\newcommand\CxiWW{{C^W_{\scalebox{\sizeC}{$\xiW$}}}}
\newcommand\CxioneWW{{C^W_{\scalebox{\sizeC}{$\xi_W^1$}}}}
\newcommand\CxitwoWW{{C^W_{\scalebox{\sizeC}{$\xi_W^2$}}}}
\newcommand\CxithreeWW{{C^W_{\scalebox{\sizeC}{$\xi_W^3$}}}}
\newcommand\CxifourWW{{C^W_{\scalebox{\sizeC}{$\xi_W^4$}}}}

\newcommand\CDeK{C_{\scalebox{\sizeC}{$\DeK$}}}
\newcommand\CDeW{C_{\scalebox{\sizeC}{$\DeW$}}}
\newcommand\CDeWK{C^\K_{\scalebox{\sizeC}{$\DeW$}}}
\newcommand\CDeWW{C^\W_{\scalebox{\sizeC}{$\DeW$}}}

\newcommand\CDeDzphiTKK{{C^{\K\K}_{\scalebox{\sizeC}{$\DeDzphiT[\DeW]$}}}}
\newcommand\CDeDzphiTKW{{C^{\K\W}_{\scalebox{\sizeC}{$\DeDzphiT[\DeW]$}}}}

\newcommand\CDeWDelaKK{C^{\K\K}_{\scalebox{\sizeC}{$\DeW\Dela$}}}
\newcommand\CDeWDelaWW{C^{\W\W}_{\scalebox{\sizeC}{$\DeW\Dela$}}}
\newcommand\CDeWDelaKW{C^{\K\W}_{\scalebox{\sizeC}{$\DeW\Dela$}}}

\newcommand\CDeL{C_{\scalebox{\sizeC}{$\DeL$}}}
\newcommand\CDeLK{C^{\K}_{\scalebox{\sizeC}{$\DeL$}}}
\newcommand\CDeLW{C^{\W}_{\scalebox{\sizeC}{$\DeL$}}}

\newcommand\CDeLT{C_{\scalebox{\sizeC}{$\DeL^\ttop$}}}
\newcommand\CDeLTK{C^{\K}_{\scalebox{\sizeC}{$\DeL^\ttop$}}}
\newcommand\CDeLTW{C^{\W}_{\scalebox{\sizeC}{$\DeL^\ttop$}}}

\newcommand\CDeGL{C_{\scalebox{\sizeC}{$\DeGL$}}}
\newcommand\CDeGLK{C^{\K}_{\scalebox{\sizeC}{$\DeGL$}}}
\newcommand\CDeGLW{C^{\W}_{\scalebox{\sizeC}{$\DeGL$}}}

\newcommand\CDeB{C_{\scalebox{\sizeC}{$\De \B$}}}
\newcommand\CDeBK{C^{\K}_{\scalebox{\sizeC}{$\De \B$}}}
\newcommand\CDeBW{C^{\W}_{\scalebox{\sizeC}{$\De \B$}}}

\newcommand\CDeNO{C_{\scalebox{\sizeC}{$\De \NO$}}}
\newcommand\CDeNOK{C^{\K}_{\scalebox{\sizeC}{$\De \NO$}}}
\newcommand\CDeNOW{C^{\W}_{\scalebox{\sizeC}{$\De \NO$}}}

\newcommand\CDeNOT{C_{\scalebox{\sizeC}{$\De \NOT$}}}
\newcommand\CDeNOTK{C^{\K}_{\scalebox{\sizeC}{$\De \NOT$}}}
\newcommand\CDeNOTW{C^{\W}_{\scalebox{\sizeC}{$\De \NOT$}}}

\newcommand\CDela{C_{\scalebox{\sizeC}{$\Dela$}}}
\newcommand\CDelaK{C^{\K}_{\scalebox{\sizeC}{$\Dela$}}}
\newcommand\CDelaW{C^{\W}_{\scalebox{\sizeC}{$\Dela$}}}

\newcommand\CDeinvla{C_{\scalebox{\sizeC}{$\Deinvla$}}}

\newcommand\CDehS{C_{\scalebox{\sizeC}{$\DehS$}}}
\newcommand\CDehSK{C^{\K}_{\scalebox{\sizeC}{$\DehS$}}}
\newcommand\CDehSW{C^{\W}_{\scalebox{\sizeC}{$\DehS$}}}

\newcommand\CDeinvaverS{C_{\scalebox{\sizeC}{$\De \aver{\S}^{\!-1}$}}}
\newcommand\CDeinvaverSK{C^{\K}_{\scalebox{\sizeC}{$\De \aver{\S}^{\!-1}$}}}
\newcommand\CDeinvaverSW{C^{\W}_{\scalebox{\sizeC}{$\De \aver{\S}^{\!-1}$}}}

\newcommand\rhoj{{\rho_j}}

\newcommand\EKj{{E_{\K_j}}}
\newcommand\EWj{{E_{\W_j}}}

\allowdisplaybreaks

\begin{document}
\title[On the convergence of flow map parameterization methods]{
   On the convergence of flow map parameterization methods for
   whiskered tori in quasi-periodic Hamiltonian systems
}
\date{\today}

\author{{\'A}lvaro Fern{\'a}ndez-Mora$^{\mbox{1}}$}
\address[1]{Departament de Matem\`atiques i Inform\`atica,
Universitat de Barcelona, Gran Via 585, 08007 Barcelona, Spain.}
\email{alvaro@maia.ub.es (corresponding author)}

\author{{\`A}lex Haro$^{\mbox{1,2}}$}
\address[2]{Centre de Recerca Matem\`atica, Edifici C, Campus Bellaterra, 08193 Bellaterra, Spain}
\email{alex@maia.ub.es}

\author{J. M. Mondelo$^{\mbox{3}}$}
\address[3]{Departament de Matem\`atiques, Universitat Aut\`onoma
de Barcelona, Av.~de l'Eix Central, Edifici C, 08193
Bellaterra (Barcelona), Spain.}
\email{josemaria.mondelo@uab.cat}

\begin{abstract}
In this work, we obtain an {\em a-posteriori} theorem for
the existence of partly hyperbolic invariant tori in analytic Hamiltonian
systems: autonomous, periodic, and quasi-periodic. The method of
proof is based on the convergence of a KAM iterative scheme to
solve the invariance equations of tori and their invariant bundles under
the framework of the parameterization method.
Starting from parameterizations analytic in a complex strip and
satisfying their invariance equations approximatly, we derive
conditions for the existence of analytic parameterizations in a
smaller strip satisfying the invariance equations exactly.
The proof relies on the careful treatment of the analyticity loss
with each iterative step and on the control of geometric properties
of symplectic flavour. We also provide all the necessary explicit
constants to perform computer assisted proofs.

\end{abstract}

\maketitle


\section{Introduction}
Suppose we have a dynamical system for which we know an invariant
torus exists. A very natural question is whether the invariant
torus persists under perturbations. An answer to this
question came from the pioneering works of Kolmogorov, Arnold,
and Moser \cite{Kolmogorov54, Arnold63a, Moser62} on the 
persistence of invariant tori under {\em small enough}
perturbations of integrable systems---what is today known as KAM
theory. A key ingredient in the proofs on persistence is the
assumption that the perturbation is {\em small enough}.
Nonetheless, numerical experiments strongly suggest that persistence
occurs beyond the perturbative regime.

The {\em small enough} assumption on the perturbation was
dropped in what was originally called KAM theory without
action-angle variables in \cite{GonzalezJLV05}, where the authors obtain
theorems for the existence of invariant Lagrangian tori in
Hamiltonian systems. Existence is obtained from the convergence of an iterative
procedure, in a scale of Banach spaces, for the embedding of the
torus satisfying a functional equation. Instead of a small enough
perturbation, what is necessary is an approximate embedding
satisfying a functional equation with an error that is {\em small
enough}. 
The method of proof is based on the ideas of
the parameterization method introduced in \cite{CabreFL03a,
CabreFL03b, CabreFL05} for invariant manifolds of fixed points.
Similar ideas where used to prove the
existence of, and to compute, normally hyperbolic invariant tori
for quasi-periodic maps in the trilogy \cite{HaroL06b, HaroL06a,
HaroL06c}. For the parameterization method in a variety of different
contexts see e.g.,
\cite{FontichLS09,HuguetLS12,LuqueV11,CallejaCL13a,CanadellH17a,HM21,gonzalez2022}
and for a survey see \cite{HaroCFLM16}.
This kind of theorems are known as {\em a-posteriori} and
require some initial data from where a condition for existence
can be checked in a finite number of computations. In
combination with rigorous numerical techniques, it is possible to
perform computer assisted proofs. This approach led in
\cite{FiguerasHL17} to the formulation of a general methodology
for such kind of proofs---using the parameterization method---for
Lagrangian invariant tori. A modified approach that leads to
sharp estimates for the condition for existence was introduced in
\cite{Villanueva17} and similar ideas for systems with extra first
integrals were developed in \cite{FH24}.  

In this work, we focus on tori invariant under vector fields, for
which an extended approach is to look for codimension one tori
invariant under flow maps. One option is to solve numerically the
invariance equation with collocation methods \cite{GomezM01}.
However, this approach suffers from the large matrix problem,
which restricts the dimension of the invariant tori by the
computational cost of iterative steps. Another option is to use
Floquet theory to compute reducible invariant tori and avoid the
large matrix problem \cite{gimeno22}. Under the parameterization
method, efficient algorithms were introduced in \cite{HM21} to
compute invariant tori for autonomous systems in problems of
celestial mechanics. The method was later generalized to the
quasi-periodic setting in \cite{FMHM24}. 

We build on \cite{HM21,FMHM24} and obtain an 
{\em a-posteriori} theorem for partly hyperbolic invariant tori
and their invariant bundles in quasi-periodic Hamiltonian
systems. We use flow maps related to the
natural frequencies of tori to reduce the dimension of the tori
by one and do not use an autonomous reformulation by introducing
fictitious {\em action} coordinates---this point of view allows
the theorem to be applied to both autonomous and quasi-periodic
Hamiltonian systems and constitutes the proof on the convergence
of the flow map methods of \cite{HM21, FMHM24}. Additionally, we
obtain both the stable and unstable bundles simultaneously which
provides a clear geometrical picture of the tangent space on the
torus. The proof relies on geometric properties of symplectic flavour
that hold approximately when we have parameterizations that
satisfy their invariance equations approximately. We obtain
approximate geometric lemmas for such properties that, together
with the lemma on quadratically small averages obtained in
\cite{FMHM24}, control the error in the KAM iterative procedure.
The new error in the invariance equations is controlled by 
explicit constants which require the
careful treatment of the analyticity loss with each iterative
step. The proof concludes by obtaining conditions for convergence
of the KAM iterative procedure. 

The structure of this work is the following. In Section
\ref{sec:setting} we introduce the symplectic and analytic
setting as well as some preliminary notions. 
Section \ref{sec:invariant tori} describes partly hyperbolic
invariant tori, the iterative procedure for their computation,
and our main result: the KAM theorem. In Section \ref{sec:some lemmas}
we state and prove the approximate geometric lemmas that we use
to control the new errors after one step of the iterative
procedure; for which we obtain estimates in Section \ref{sec:tori
bundle correction}. Lastly, in Section \ref{sec:proof KAM} we
conclude the proof of the KAM theorem and in Appendix
\ref{ap:constants} we provide tables with all the explicit
constants obtained in our estimates.

\section{Setting}\label{sec:setting}

\subsection{Basic notation}\label{sec:basic notation}
We consider the vector spaces $\RR^m$ and
$\CC^m$ equipped with
the norm
\[
        |v|=\max_{i=1,\dots,m}|v_i|,
\]
which induces the following norm in the spaces of $n_1\times n_2$ 
matrices with components in $\RR$ or $\CC$
\[
        |M|=\max_{i=1,\dots,n_1}\sum_{j=1,\dots,n_2}|M_{ij}|.
\]
We denote by $M^\ttop$ the transpose of the matrix $M$. Also, 
the $n\times n$ identity and zero matrices are denoted by $I_n$
and $\cO_n$, respectively. Meanwhile, the $n-$dimensional zero
vector is denoted by $0_n$. When it is clear from the context, 
empty blocks denote zero matrices of suitable dimensions.
Let $\TT^\ell=\RR^\ell/\ZZ^\ell$ be the standard $\ell-$dimensional
torus.
For $\al\in\RR^\ell$, we define the
rotation map $\Ral:\TT^{\ell}\to\TT^{\ell}$ as
\[
   \Ral(\vp)=(\vp+\al)\ \mod 1.
\]
Similarly, for $(\om,\al)\in \RR^{d+\ell}$, we define the
rotation map $\Romal:\TT^{d+\ell}\to\TT^{d+\ell}$ as
\begin{equation}\label{eq:romal}
   \Romal(\te,\vp)=(\te+\om,\vp+\al)\ \mod 1.
\end{equation}
Throughout the exposition, $(\om,\al)$ will be fixed and,
occasionally, we will use the notation $\bar\te:=\te+\om$, and $\bar\vp:=\vp+\al$. 


\subsection{Symplectic setting} \label{sec:symplectic setting}
We assume we have an exact symplectic form $\bOm$, defined on an
open set $U\subset\RR^{2n}$, with matrix representation
$\Om:U\to\RR^{2n\times 2n}$. We also assume we have an analytic
Hamiltonian $H:U\times\TT^\ell\to\RR$ that depends on time
quasi-periodically as $H(z,\vp) = H(z,\hat\al t)$ for frequencies
$\hatal\in\RR^\ell$.  We can then construct
$X:U\times\TT^\ell\to\RR^{2n}$ as 
\begin{equation*}
   X(z,\vp):=\Om(z)^{-1}\Dif_z H(z,\vp)^\ttop,
\end{equation*}
and on $U\times\TT^\ell$, we can define
the vector field $\tilde
X:U\times\TT^\ell\to\RR^{2n}\times\RR^\ell$ as follows
\begin{equation}\label{eq:tVF}
   \begin{pmatrix}
    \dot z \\ \dot \vp
   \end{pmatrix}=\tilde X(z,\vp):=
   \begin{pmatrix}
   X(z,\vp) \\ \hatal
   \end{pmatrix}.
\end {equation}
We can think of $U\times\TT^\ell$ as a bundle with base
$\TT^\ell$ and of $\tilde X$ as a fiberwise Hamiltonian vector
field---in each fiber, we have a symplectic vector
structure given by $\Om$ and a Hamiltonian vector field
$X(\cdot,\vp):U\to\RR^{2n}$. Our goal is to find partly
hyperbolic invariant tori for the vector field $\tilde X$.

Let us denote the flow associated to $\tilde X$ by
$\tphi:D\subset \RR\times U\times \TT^\ell\to
U\times\TT^\ell$, where $D$ is the open set domain of definition
of the flow. That is, for every $(z,\vp)$ let
$I_{z,\vp}\subset\RR$ be the maximal interval of existance of 
$\tphi(\cdot,z,\vp)$, then 
\[
   D=\{(t,z,\vp)\in\RR\times U \times\TT^\ell ~|~ t\in I_{z,\vp}
\}.
\]
Notice that $\tphi$ adopts the form 
\[
   \tphi(t,z,\vp) = \begin{pmatrix}
   \phi(t,z,\vp)\\ \vp+\hatal t
   \end{pmatrix},
\]
where the evolution operator $\phi$ satisfies
\begin{equation*}
   \frac{\pd}{\pd t}\phi(t,z,\vp) = X\big(\phi(t,z,\vp),\vp+\hatal
   t\big), \quad \phi(0,z,\vp) = z.
\end{equation*}
From this point onwards, we will adopt the standard notations
$\tphi_t(z,\vp)=\tphi(t,z,\vp)$ and $\phi_t(z,\vp) =
\phi(t,z,\vp)$.

Since the symplectic form is exact, $\bOm=\dif\bal$, where $\bal$ is the
action 1-form. Then, the matrix representation of $\bOm$ is given
by
\[
\Om(z)=\Dif a(z)^\ttop-\Dif a(z),
\]
where $a:U\to\RR^{2n}$ and $a(z)^\ttop$ is the matrix
representation of $\bal$ at $z$. Also, because $\bOm$ is exact
symplectic, $\phi_t$ is a fiberwise exact symplectomorphism.
That is, for each $\vp\in\TT^\ell$ and fixed $t$, $\phi_t$ satisfies symplecticity
\[
\Dif_z\phi_t(z,\vp)^\ttop\Om\big(\phi_t(z,\vp)\big)
\Dif_z\phi_t(z,\vp) = \Om(z)
\]
and exactness
\[
a\big(\phi_t(z,\vp)\big)^\ttop\Dif_z\phi_t(z,\vp)-a(z)^\ttop=\Dif_z
p_t(z, \vp),
\]
for some primitive function $p_t:U\times\TT^\ell\to\RR$, see
\cite{FMHM24} for details.

Let us also assume we have an almost-complex structure $\bJ$ on
$U$, compatible with the symplectic form, with matrix representation 
$J:U\to\RR^{2n\times 2n}$. That is, $\bJ$ is
anti-involutive and symplectic---in coordinates, these properties
read
\[
   J(z)^2=-I_{2n},\quad
J(z)^\ttop\Om(z)J(z)=\Om(z).
\]
The almost complex structure and the symplectic form induce a
Riemannian metric $\bg$ with matrix representation
$G:U\to\RR^{2n\times 2n}$ constructed as $G(z)=-\Om(z)J(z)$.

A particular example of a compatible tripple $(\bOm,\bJ,\bg)$ is
when $\bOm$ is the standard symplectic form, i.e., 
\[
   \Om_0(z)=
   \begin{pmatrix}
      \cO_n & -I_n\\
      I_n & \cO_n
   \end{pmatrix},
\] $\bg$ is the Euclidean metric, and $\bJ$ is given by the
standard linear complex structure on $\RR^{2n}$ coming from the
complex structure on $\CC^n$. That is,
\[
   G_0(z) = \begin{pmatrix}
      I_n & \cO_n\\
      \cO_n & I_n
      \end{pmatrix}, \quad
      J_0(z) = 
   \begin{pmatrix}
      \cO_n & -I_n\\
      I_n & \cO_n
   \end{pmatrix}.
\]

\subsection{Analytic setting}\label{sec:analytic setting}
We assume we have a real analytic Hamiltonian. Therefore, 
we will work with real analytic functions defined in
complex neighborhoods of real domains. Let us construct the
complex neighborhood of $\TT^\ell$ as the following complex strip of
width $\rho$
\[
   \TT_\rho^\ell=\{\te\in\CC^\ell/\ZZ^\ell : \Abs{{\rm Im}\ \te}<\rho\}.
\]
Furthermore, assume we have a holomorphic extension
$\KO:\TT^d_r\times\TT^\ell_r\to\CC^{2n}$ of a map
$\KO:\TT^d\times\TT^\ell\to\RR^{2n}$, for some $d$ and $r>0$,
where we abuse notation and denote both the map and its
holomorphic extension by $\KO$.
Let us then construct the sets $\U\subset\CC^{2n}$,
$\tU\subset\CC^{2n}\times\TT^\ell_r$, and $\tU_\vp\subset\CC^{2n}$ as 
\begin{align}\label{eq:defU}
   \U&:=\{z\in\CC^{2n} : \exists
   (\te,\vp)\in\TT^d_r\times\TT^\ell_r :
\Abs{z-\KO(\te,\vp)}<R\},\\\nonumber
   \tU:&=\{(z,\vp)\in\CC^{2n}\times\TT^\ell_r : \exists
   \te\in\TT^d_r : \Abs{z-\KO(\te,\vp)}<R\},\\\nonumber
   \tU_\vp:&=\{z\in\CC^{2n} : \exists
   \te\in\TT^d_r : \Abs{z-\KO(\te,\vp)}<R\},
\end{align}
for some $R>0$. For analytic functions $f:\U\to\CC$, we consider the
norm
\begin{equation}\label{eq:normuU}
        \norm{f}_\U=\sup_{z\in\U}|f(z)|,
\end{equation}
that extends naturally 
for matrix-valued maps $M:\U\to\CC^{n_1\times n_2}$, with
$M_{ij}$ analytic, as
\begin{equation}\label{eq:normMU}
        \norm{M}_\U=\max_{i=1,\dots,n_1}\sum_{j=1,\dots,n_2}
        \norm{M_{ij}}_\U.
\end{equation}
For their derivatives, we define the norms
\begin{equation}\label{eq:normDU}
\norm{\Dif^k f}_\U=\sum_{l_1,\dots,l_k}\Norm{\frac{\partial^k
f}{\partial z_{l_1}\dots\partial z_{l_k}}}_\U,\quad
\norm{\Dif^k M}_\U=\max_{i=1,\dots,n_1}
\sum_{j=1,\dots,n_2}\norm{\Dif^k M_{ij}}_\U,
\end{equation}
where each index $l_1,\dots,l_k$ goes from 1 to $2n$.
Furthermore, the action of the
$k-$order derivative on $v_1,\dots,v_k\in\CC^{n_2}$
satisfies the bounds 
\[
   \norm{\Dif^k M(z) [v_1,\dots,v_k]}_{\U}\leq\norm{\Dif^k M}_{\U}
   \Abs{v_1}\dots\Abs{v_k},
\]
where
\begin{align*}\nonumber
   \left(\Dif^k M(z) [v_1,\dots,v_k]\right)_{ij}&=
   \Dif^k M_{ij}(z)
   [v_1,\dots,v_k]\\
&=\sum_{l_1,\dots,l_k}\frac{\partial^k
M_{ij}}{\partial_{z_{l_1}},\dots,\partial_{z_{l_k}}}(z)
\big(v_{1}\big)_{l_1},\dots,\big(v_k\big)_{l_k}.
\end{align*}
Lastly, for analytic functions $f:\tU\to\CC$ we consider the norm
\[
   \norm{f}_{\tU} =
   \sup_{\vp\in\TT^\ell_r}\norm{f(\cdot,\vp)}_{\tU_\vp}
\]
and the analogous norm to \eqref{eq:normMU} for matrix-valued maps.

\subsubsection{Spaces of quasi-periodic real-analytic functions}
Let us denote by $\cA(\TT^{d+\ell}_\rho)$ the Banach space of
real-analytic functions $f:\TT^{d+\ell}_\rho\to\CC$---such that
$f(\TT^{d+\ell})\subset\RR$---that can be
continuously extended to the closure $\bar\TT^{d+\ell}_\rho$
and endowed with the norm
\[
   \norm{f}_\rho=\sup_{(\te,\vp)\in\TT^{d+\ell}_\rho}
   \abs{f(\te,\vp)}.
\]
We will also use the notation $\cA(\TT^{d+\ell}_\rho)^{m}$ for
the analogous space for functions $f:\TT^{d+\ell}_\rho\to\CC^m$,
and the natural extension of the norm for functions
$M:\TT^{d+\ell}_\rho\to\CC^{n_1\times n_2}$
\[
        \norm{M}_\rho=\max_{i=1,\dots,n_1}\sum_{j=1,\dots,n_2}
        \norm{M_{ij}}_\rho.
\]
For the derivatives of quasi-periodic real-analytic functions, we
use Cauchy estimates. For any $\de\in(0,\rho)$, the derivative
operator
$\partial_i:\cA(\TT^{d+\ell}_\rho)\to\cA(\TT^{d+\ell}_{\rho-\de})$
satisfies the bounds
\begin{equation*}
\norm{\partial_i f}_{\rho-\de}\leq\frac{1}{\de}\norm{f}_\rho,\quad
\norm{\Dif_\te f}_{\rho-\de}\leq\frac{d}{\de}\norm{f}_\rho,\quad
\norm{\Dif_\vp f}_{\rho-\de}\leq\frac{\ell}{\de}\norm{f}_\rho.
\end{equation*}
Additionally, the Fourier representation of $f$ is given by 
\[
   f(\angles)=\sum_{(k,j)\in\ZZ^{d+\ell}}\hat
   f_{jk}e^{\bi2\pi(k\te+j\vp)},
\]
where $\hat f_{jk}\in\CC$ and 
\[
   \hat f_{00}=\int_{\TT^{d+\ell}} f(\angles)d\te d\vp
\]
is the average of $u$ that we denote as $\aver{f}$. 
Lastly, for functions $f:\tU\to\CC$ and 
$M:\tU\to\CC^{n_1\times n_2}$, we define the analogous norms to 
\eqref{eq:normuU}, \eqref{eq:normMU}, and \eqref{eq:normDU}.

\subsection{Cohomological equations in $\TT^{d+\ell}$}\label{sec:coho}
We dedicate this section to two types of cohomological equations that
naturally arise in KAM theory. Albeit the theory is well known,
see e.g. \cite{Russmann75}, we tailor the exposition to be well
suited for our purposes.

\subsubsection{Non-small divisors cohomological equations}\label{sec:cohoNsmall}
For $a,b\in\RR$ such that $\abs{a}\neq\abs{b},~\om\in\RR^d$,
and $\al\in\RR^\ell$, we define the cohomological operator
$\Lie{\om\al}^{ab}:\cA(\TT^{d+\ell}_\rho)\to\cA(\TT^{d+\ell}_\rho)$
acting on functions $\xi:\TT^{d+\ell}_\rho\to\CC$ as 
\begin{equation*}
\Lie{\om\al}^{ab}
\xi:=a\xi-b\xi\comp\Romal,
\end{equation*}
where $\Romal$ is defined as in \eqref{eq:romal}. We also define the Rüssmann operator
$\R{\om\al}^{ab}:\cA(\TT^{d+\ell}_\rho)\to\cA(\TT^{d+\ell}_\rho)$
acting on functions $\eta:\TT^{d+\ell}_\rho\to \CC$
as the solution of the cohomological equation
\begin{equation}\label{eq:cohogeneral}
   \Lie{\om\al}^{ab} \xi = \eta.
\end{equation}
That is, $\xi=\R{\om\al}^{ab}\eta$ solves \eqref{eq:cohogeneral}. 
Expanding $\xi$ and $\eta$ in their Fourier series, we obtain the
converging expansion 
\[
   \left(\R{\om\al}^{ab}\eta\right)(\angles) =
   \sum_{(k,j)\in\ZZ^{d+\ell}}\hat\xi_{jk}e^{\bi2\pi(j\te+k\vp)},\quad
   \hat\xi_{jk}=\frac{\hat\eta_{jk}}{a-b e^{\bi2\pi(j\om+k\al)}}.
\]
Additionally, we have the estimate
\begin{equation}\label{eq:estR2}
   \norm{\xi}_{\rho}\leq\frac{1}{\big|\abs
   {a}-\abs{b}\big|}\norm{\eta}_{\rho}.
\end{equation}
We will also use the notation $\Lie{\om\al}^{ab}=\Lop^{ab}$ and
$\R{\om\al}^{ab}=\Rop^{ab}$, when it is clear from the context.

\subsubsection{Small divisors cohomological
equations}\label{sec:cohosmall}
We also define the cohomological operator 
$\Lie{\om\al}:\cA(\TT^{d+\ell}_\rho)\to\cA(\TT^{d+\ell}_\rho)$
as 
\[
   \Lie{\om\al}\xi:=\xi-\xi\comp\Romal,
\]
and for any $\de\in(0,\rho]$, we define the Rüssmann
operator
$\R{\om\al}:\cA(\TT^{d+\ell}_\rho)\to\cA(\TT^{d+\ell}_{\rho-\de})$ 
as the zero-average solution of 
\begin{equation}\label{eq:cohosmall}
   \Lie{\om\al}\xi=\eta-\aver{\eta}.
\end{equation}
That is $\xi=\R{\om\al}\eta$ solves \eqref{eq:cohosmall} with
$\aver{\xi}=0.$ Again, expanding $\xi$ and $\eta$ in their Fourier
series, we obtain the formal expansion
\[
   \left(\R{\om\al}\eta\right)(\angles) =
   \sum_{(k,j)\in\ZZ^{d+\ell}}\hat\xi_{jk}e^{\bi2\pi(j\te+k\vp)},\quad
   \hat\xi_{jk}=\frac{\hat\eta_{jk}}{1-e^{\bi2\pi(j\om+k\al)}},
\]
where the series is uniformly convergent in a narrower strip if
$(\om,\alpha)$ is Diophantine. 
\begin{definition}\label{def:diophantine}
The rotation vector $(\om,\al)\in\TT^{d+\ell}$ is Diophantine if 
there exists $\gamma>0$ and $\tau\geq d + \ell$
such that for all $n\in\ZZ$ and $(j,k)\in\ZZ^{d}\times\ZZ^\ell
\setminus \{ 0\}$
\[
   |j\cdot\omega+k\cdot\alpha-n|\geq
   \frac{\gamma}{\left(|j|_1+|k|_1\right)^\tau}, 
\]
where $|\cdot|_1$ is the $\ell^1$-norm. 
\end{definition}

Then, we have R\"usmann estimate
\begin{equation}\label{eq:estR1}
   \norm{\xi}_{\rho-\de}\leq
   \frac{\CR}{\ga\de^\tau}\norm{\eta}_\rho,
\end{equation}
where $\CR$ is the R\"ussmann constant. See
\cite{Russmann75} for details. We will also use the notations
$\Lie{\om\al}=\Lop$ and $\R{\om\al}=\Rop$.

\section{Partially hyperbolic invariant tori}\label{sec:invariant tori}
We are interested in $(d+\ell+1)-$dimensional tori with
frequencies $\hatom\in\RR^{d+1}$ and $\hatal\in\RR^\ell$---
invariant for the vector field \eqref{eq:tVF}---and in their
invariant bundles.  Equivalently, for
$\hatom=:\frac{1}{T}(\om,1)$ and $\al:=\hatal T$, where
$\om\in\TT^d$ and $\al\in\TT^\ell$, we can look for
$(d+\ell)-$dimensional tori $\cK$, invariant for the flow map
$\phi_T$, and in bundles $\cW$ with base $\cK$, invariant under
the differential of $\phi_T$, see
\cite{FMHM24,HM21}. In the present work, we focus on invariant tori with
rank 1 bundles---which corresponds to $d=n-2$. This choice is
motivated by applications in celestial mechanics such as
quasi-periodic perturbations of restricted three body problems.
In the spirit of the parameterization method, we are 
interested in fiberwise embeddings $\K:\TT^{d+\ell}\to U$
and $\W:\TT^{d+\ell}\to\RR^{2n}$, with rate $\la\in\RR\backslash\{1,-1\}$,
satisfying the following invariance equations
\begin{gather}\label{eq:invK}
   \phiK - \K\comp\Romal = 0,\\\label{eq:invW}
   \DzphiK \W-\W\comp\Romal\lambda=0,
\end{gather}
where $(\K,\id)$ stands for the following bundle map over
the identity
\begin{align*}
   (\K,\id):\quad \TT^{d+\ell} &\longrightarrow
   U\times\TT^\ell \\
(\angles)&\longmapsto \big(\K(\angles),\vp\big).
\end{align*}
Without loss of generality, we can assume $0<\abs{\la}<1$.
We will refer to $K$ and $W$ as the parameterizations of $\cK$
and $\cW$, respectively. It can be shown that due to certain
geometric properties, see \cite{FMHM24}, if $(\K,\W,\la)$
satisfy \eqref{eq:invK} and \eqref{eq:invW}, then there exists a
set of coordinates given by a frame $\P$ such that the
differential of the flow map reduces to an upper triangular
matrix-valued map---constant along the diagonal. That is, the
linearized dynamics in the coordinates given by $\P$ reduce to
$\tLa:\TT^{d+\ell}\to\RR^{2n}$ as 
\begin{equation}\label{eq:reducibility}
    (\P\comp\Romal)^{-1}\DzphiK\P = \tLa,
 \end{equation} 
where
\begin{equation}\label{eq:defLas}
    \tLa:=\left(
         \begin{array}{c|c}
            \La & \tS \\
            \hline
                   & \mybox{\La^{-\ttop}}
         \end{array}
      \right),\quad
      \La:=\left(
         \begin{array}{c|c}
            I_{n-1}&  \\
            \hline
                   & \la 
         \end{array}
      \right),
\end{equation}
and $\tS:\TT^{d+\ell}\to\RR^{n\times n}$. In Section
\ref{sec:some lemmas} we provide expressions and bounds for the
geometric properties when \eqref{eq:invK} and \eqref{eq:invW}
only hold approximately.

\subsection{Adapted frames}\label{sec:frame}
In this section we briefly cover the construction of adapted
frames that reduce the linearization of $\phi_T$ on $\cK$ to an
upper triangular matrix-valued map, constant along the diagonal. 
In what follows, assume $\cK$ is an invariant torus under
$\phi_T$, $\cW$ is the invariant bundle with base $\cK$, and that
the dynamics in the hyperbolic linear subspace are reduced to a constant
contraction with rate $\abs{\la}$. For a detailed
description of the construction of adapted frames and their
geometric properties see \cite{FMHM24}.

We construct a frame $\P:\TT^{d+\ell}\to\RR^{2n\times2n}$ as
the juxtaposition of two subframes
$\L,\N:\TT^{d+\ell}\to\RR^{2n\times n}$, i.e., 
\[
   \P = \big(\L\quad \N \big).
\]
The subframe $\L$ is tangent to the torus whereas the
normal subframe $\N$ complements $\L$ to a full basis of
$T_\cK U$, i.e., the tangent bundle of $U$ restricted to $\cK$.
We proceed by first constructing the subframe 
$\L$ as follows 
\begin{equation*}
 \L:=\Big({\rm D}_{\te}\K
 \quad \mathcal{X} \quad W \Big),
\end{equation*}
where
\begin{equation*}
  \mathcal{X}:=X\comp(\K,\id) - {\rm
  D}_\vp  \K\hatal,
\end{equation*}
and the subframe $\N$ as
\begin{equation*}
\N = \NO + \L\A,
\end{equation*}
for certain $\A:\TT^{d+\ell}\to\RR^{n\times n}$, where
\begin{align}\label{eq:defN0}
   \NO &=
\J\comp\K\L\B,\\\nonumber
   \B&=\GL^{-1},\\\nonumber
   \GL&=\L^\ttop
   \G\comp\K\L.
\end{align}
The role of $\A$ is to obtain simultaneously the stable and
unstable bundles of the invariant torus but its use is not
necessary, i.e., we could use the frame $\hP=(\L\ \NO)$. We will
come back to this topic but first let us emphasize some 
properties of these matrix-valued maps. The frame $\P$ is
symplectic with respect to the standard symplectic form, i.e.,
$\P$ satisfies
\begin{equation*}
\P^\ttop\OmK \P = \OmO,
\end{equation*}
from where we obtain 
\begin{equation}\label{eq:invP}
   \P^{-1} = -\OmO\P^\top\OmK.
\end{equation}
Moreover, $\A$ needs to be symmetric to preserve the
symplecticity of $\P$ and $\GL$ is invertible since $\K$ is a
fiberwise embedding. That is, for each $\vp\in\TT^\ell$,
$\K(\cdot,\vp):\TT^d\to U$ is an embedding.  Lastly, both $\GL$
and $\B$ are symmetric.

To see how we obtain both the stable and unstable bundles with
the right choice of $\A$, let us denote the torsion of the normal bundles
parameterized by $\NO$ and $\N$, when transported by
$\DzphiT\comp(\K,\id)$, by the matrix-valued maps
$\hS,\tS:\TT^{d+\ell}\to\RR^{n\times n}$ 
\begin{align}\label{eq:defshat}
   \hS &:=
   \NO^\ttop\comp\Romal\OmbarK
   \DzphiK\NO,\\\nonumber
   \tS &:=
   \N^\ttop\comp\Romal\OmbarK
   \DzphiK\N,
\end{align}
where a computation reveals
\begin{equation}\label{eq:tS-hS}
   \tS = \hat S + \La \A -
    \A^\ttop\comp\Romal\La^{-\ttop} .
\end{equation}
For the details see \cite{HM21,FMHM24}.
From \eqref{eq:reducibility} and \eqref{eq:defLas}, observe that
if we choose $\A$ such that $\tS$ adopts the form
\begin{equation}\label{eq:reducedS}
   \tS=
\left(
   \begin{array}{c|c}
      \S & \\
      \hline
        & 
\end{array}
\right)
,
\end{equation}
where $\S:\TT^{d+\ell}\to\RR^{(n-1)\times(n-1)}$, 
then the last column of $\P$, that we denote by
$\widetilde\W$, satisfies
\begin{equation}\label{eq:unstablebundle}
   \DzphiK\widetilde\W -\widetilde\W\comp\Romal\la^{-1}=0. 
\end{equation}
Therefore, $\W$ and $\widetilde\W$ parameterize both 
the stable and unstable bundles of $\cK$, respectively.
\begin{remark}\label{rem:hSsym}
When $\K$ parameterizes an invariant torus and the bundle is
invariant, the torsion $\hS$
satisfies the symmetry condition
\begin{equation}\label{eq:hSsym}
   \La^{-1}\hS - \left(\La^{-1}\hS\right)^\ttop=0.
\end{equation}
Then, for $\tS$ as in \eqref{eq:reducedS} and a suitable
symmetric $\A$, \eqref{eq:tS-hS} holds.
However, the symmetry condition for $\hS$
only holds approximately when $(\K,\W,\la)$ satisfy \eqref{eq:invK}
and \eqref{eq:invW} approximately,  see Section \ref{sec:ap_symmetry}.
\end{remark}
Let us define splittings for $\hS$ and $\A$ into blocks of sizes
$(n-1)\times(n-1), (n-1)\times 1, 1\times(n-1), 1\times 1$ as
\begin{equation}\label{eq:splithS}
\hS = \begin{pmatrix} \hS_1 & \hS_2 \\
   \hS_3 & \hS_4 \end{pmatrix}, \quad 
   \A=\begin{pmatrix}
      \Aone & \Atwo\\
      \Athree & \Afour
   \end{pmatrix}.
\end{equation}
Then, condition \eqref{eq:reducedS} reads 
\begin{align}\label{eq:cohoA1}
   \S&=\hS_1 +\phantom{\la}\Aone
   - \Aone^\ttop\comp\Romal,\\
   0&=\hS_2+ \phantom{\la}\Atwo - \Athree^\ttop\comp\Romal\lambda^{-1},
   \label{eq:cohoA2}\\
   0&=\hS_3+
   \la\Athree - \Atwo^\ttop\comp\Romal,
   \label{eq:cohoA3}\\
   0&=\hS_4+
   \la\Afour - \Afour\comp\Romal\lambda^{-1},
   \label{eq:cohoA4}
\end{align}
where we can take $\A_1=\cO_{n-1}$---consequently $\S=\hS_1$---and
solve \eqref{eq:cohoA4} as described in Section
\ref{sec:cohoNsmall}. Notice that \eqref{eq:cohoA2} and
\eqref{eq:cohoA3} are coupled non-small cohomological equations
that can be rewritten as
\begin{align}\label{eq:cohoA2coupled}
   0&=\left(\hS_2\la + \hS_3^\ttop\comp\Romal\la^{-1}\right)
   + \Atwo\la- \Atwo\comp\cR_{\om\al}^{\comp 2}\la^{-1},\\\label{eq:cohoA3coupled}
   0&=\left(\hS_3+\hS_2^\ttop\comp\Romal\right) + \la\Athree -
   \la^{-1}\Athree\comp\cR_{\om\al}^{\comp 2},
\end{align}
where $\cR_{\om\al}^{\comp 2}$ stands for $\Romal\comp\Romal$. Then, we can solve
\eqref{eq:cohoA2coupled} and \eqref{eq:cohoA3coupled} as 
non-small divisors cohomological equations. We will also refer to $\S$
as the reduced torsion or the {\it shear}.
\begin{remark}
From \eqref{eq:hSsym}, \eqref{eq:cohoA2coupled} and
\eqref{eq:cohoA3coupled}, it follows that $\A_2= \A_3^\ttop$.
Consequently, $\A$ is symmetric.
\end{remark}

\begin{remark}
From \eqref{eq:cohoA1}, it is possible to further reduce
$\S$ to a constant matrix $\S=\aver{\hS_1}$ by solving  
\[
   \Lop\Aone = \aver{\hS_1}-\hSone.
\]
\end{remark}

\subsection{Iterative procedure}\label{sec:quasi-Newton}
Let us define the error in the torus and bundle invariance
equations $\EK,\EW:\TT^{d+\ell}\to\RR^{2n}$ as
\begin{gather}\label{eq:defEK}
   \EK = \phiK - \K\comp\Romal,\\\label{eq:defEW}
   \EW = \DzphiK \W-\W\comp\Romal \lambda.
\end{gather}
The iterative scheme designed in \cite{FMHM24}
constructs sequences for $\K, \W,$ and $\la$ such that for the
limiting objects $\K_\infty, \W_\infty$, and $\la_\infty$ the
errors satisfy $\norm{E_{\K_\infty}}=0$ and
$\norm{E_{\W_\infty}}=0$ in some suitable norm. The sequences
are obtained recursively by choosing corrections
$\DeK,\DeW:\TT^{d+\ell}\to\RR^{2n}$, and $\Dela\in\RR$ such that,
for $\barK=\K+\DeK$, $\barW=\W+\DeW$, and $\barla=\la+\Dela$,
\eqref{eq:defEK} and \eqref{eq:defEW} vanish at first order. 
In this section, we describe non-rigorously one step of the
iterative scheme---the rigorous description is postponed to
Section \ref{sec:iterative lemma} and the convergence of the
scheme to Section \ref{sec:convergence}.

\subsubsection{One step on the torus}\label{sec:NewtonK}
Let us begin with \eqref{eq:defEK}. We compute the correction of
the torus parameterization in the coordinates given by
the frame $\P$. That is, we choose $\DeK=\P\xiK$, where
$\xiK:\TT^{d+\ell}\to\RR^{2n}$ satisfies  
\begin{equation}\label{eq:lineqDeK}
\DzphiK P\xiK - P\comp\Romal\xiK\comp\Romal = -\EK
\end{equation}
up to first order in $\EK$. 
Let us left-multiply the linearized equation \eqref{eq:lineqDeK} by the
approximate inverse of $\P\comp\Romal$ given by
$-\OmO\P^\ttop\comp\Romal\OmbarK$, see \eqref{eq:invP}. Then, we
obtain after neglecting higher order terms
\begin{equation}
   \label{eq:lineq2}
   \left( 
      \begin{array}{c|c}
         \La & \tS \\
         \hline
             & \mybox{\La^{-\ttop}}
      \end{array}
   \right)
   \xiK- \xiK\comp\Romal=\etaK,
\end{equation}
where 
\begin{equation*}
 \etaK=\OmO  \P^\ttop\comp\Romal\OmbarK\EK.
\end{equation*}
 We can now split $\xiK$ and $\etaK$ into
 $(n-1)\times1\times(n-1)\times1$ components as
\begin{equation*}
   \xiK=
      \begin{pmatrix}
         \xiaK\\ \xibK\\ \xicK\\ \xidK
      \end{pmatrix}
   ,\quad
   \etaK=
      \begin{pmatrix}
         \etaaK\\ \etabK\\ \etacK\\ \etadK
      \end{pmatrix},
\end{equation*}
and rewrite system \eqref{eq:lineq2} as follows
\begin{align}
   \xiaK+\S\xicK
         -\xiaK\comp\Romal &= \etaaK,\label{eq:corrtr1} \\
   \lambda\xibK-\xibK\comp\Romal 
      &= \etabK,\label{eq:corrtr2} \\
   \xicK-\xicK\comp\Romal
      &= \etacK-\langle\etacK\rangle,\label{eq:corrtr3} \\
   \lambda^{-1}\xidK-\xidK\comp\Romal
      &= \etadK.\label{eq:corrtr4}
\end{align}
Observe that we have subtracted $\aver{\etacK}$ in \eqref{eq:corrtr3}
in order to overcome the obstruction in the small divisors
cohomological equation---in \cite{FMHM24}, we proved that
$\aver{\etacK}$ is quadratically small so its contribution in
\eqref{eq:corrtr3} is of second order. Then,
\begin{align}\label{eq:xiK2}
        \xibK &= \Rop^{\la1}\etabK,\\\label{eq:xiK3}
        \xicK &= \Rop\etacK + \aver{\xicK},\\\label{eq:xiK4}
   \xidK &= \Rop^{\la^{\!\scalebox{0.5}{$-1$}}1}\etadK, 
\end{align}
where $\aver{\xicK}$ is free. We can then require
\begin{equation}\label{eq:averxiK3}
   \aver{\xicK} = \aver{\S}^{\!-1}\left<\etaaK -\S\Rop\etacK \right>  
\end{equation}
given the non-degeneracy condition $\det \aver{\S}\neq 0$.
Then
\begin{equation}\label{eq:xiK1}
   \xiaK = \Rop\bigg(\etaaK - \S\Big(\Rop\etacK +
   \aver{\xicK}\Big)\bigg) + \aver{\xiaK},
\end{equation}
where $\aver{\xiaK}$ is free and we can simply take as zero.
This freedom has to do with the freedom to choose the phase of
the generating torus $\cK$.

\subsection{One step on the bundle}
For the parameterization of the invariant bundle, 
we compute as well the correction in the coordinates given by the
frame $\P$. That is, we take $\DeW=\P\xiW$ where 
$\xiW:\TT^{d+\ell}\to\RR^{2n}$ satisfies
\begin{equation}\label{eq:linDeW}
   \DzphiK \P\xiW - \P\comp\Romal\xiW\comp\Romal\la-\W\comp\Romal\De\la
   =-\tildeEW,
\end{equation}
and
\begin{align}\label{eq:tildeEW}
        \tildeEW:&=\DzphiT\comp(\K+\DeK,\id)\W -
        \W\comp\Romal\la\\\nonumber
         &=
        \EW + \int_0^1\DDzphiT\comp\big(\K+s\DeK,\id\big)[\DeK,\W]
   ds.
\end{align}
We emphasize that albeit we compute $\DeK$ in the previous
section, we linearize around $\K$ and not $\barK = \K + \DeK$.
This allows us to use the same frame $\P$ to correct the torus
and the bundle. It is possible to consider only
terms up to first order in $\DeK$ in \eqref{eq:tildeEW}, i.e., 
\begin{equation}\label{eq:tildeEW2}
   \tildeEW=\EW + \DDzphiT\comp(\K,\id)[\DeK,\W]
\end{equation}
which can be easier to compute than \eqref{eq:tildeEW} in
practical applications. However, for the proof on convergence,
\eqref{eq:tildeEW2} requires upper bounds on the norm of
$\Dif^3_z\phi_T$, whereas \eqref{eq:tildeEW} requires bounds only
of $\Dif^2_z\phi_T$.
It is also possible to compute first $\barK$ and consider
\eqref{eq:invW} around $\barK$. Then, we can obtain $\xiW$ and $\Dela$ from
\begin{equation*}
\DzphiT\comp(\barK,\id)\bar\P\xiW -
\bar\P\comp\Romal\xiW\comp\Romal\la - \W\comp\Romal\Dela =
-E_{\barK\W},
\end{equation*}
where $\bar\P$ is the frame for the corrected parameterization
$\barK$ and $E_{\barK\W}$ corresponds to \eqref{eq:defEW} but on
$\barK$ instead of $\K$.
Notice that then we do not have to compute $\DDzphiT$ but we
have to compute two frames: for $\K$ and for $\barK$.
Computationally, this is an interesting option and it is in fact the
approach followed in \cite{FMHM24}. However, for the KAM theorem,
it would require to control both frames $\P$ and $\bar\P$.
In what follows, we consider only \eqref{eq:tildeEW}.

Let us left-multiply the linearized equation
\eqref{eq:linDeW} again by the approximate inverse of
$\P\comp\Romal$, i.e., by $-\OmO\P^\ttop\comp\Romal\OmbarK$, to obtain, after neglecting higher
order terms,
\begin{equation}\label{eq:coho-bundle}
   \left(
      \begin{array}{c|c}
      \La & \tS \\
      \hline
          & \mybox{\La^{-\ttop}}
      \end{array}
   \right)
   \xiW -\xiW\comp\Romal\la = \etaW +
   e_n\De\la,
\end{equation}
where
\begin{equation*}
\etaW:=\OmO  \P^\ttop\comp\Romal\OmbarK\tildeEW,
\end{equation*}
\[
   e_n:=\begin{pmatrix} 0_{n-1} \\ 1 \\ 0_{n-1}\\ 0 \end{pmatrix}.
\]
Using analogous splittings for $\etaW$ and $\xiW$ as in Section
\ref{sec:NewtonK}, system \eqref{eq:coho-bundle} reads
\begin{align*}
   \xiaW+\S\xicW
         -\xiaW\comp\Romal\la
      &= \etaaW, \\
   \la\xibW-\xibW\comp\Romal\la
         &= \etabW+\Dela,\\
   \xicW-\xicW\comp\Romal\la
      &= \etacW,\\
   \invla\xidW-\xidW\comp\Romal\la
      &= \etadW.
\end{align*}
We can then take $\Dela=-\aver{\etabW}$ 
and
\begin{align}\label{eq:xiW1}
   \xiaW&=
   \Rop^{1\la}
   \big(\etaaW-\S\Rop^{1\la}\etacW\big),\\
   \xibW&=\frac{1}{\la}\Rop\etabW
   +\aver{\xibW},\\
   \xicW&=\Rop^{1\la}\etacW,\\\label{eq:xiW4}
   \xidW&=\Rop^{\invla\la}\etadW,
\end{align}
where $\aver{\xibW}$ is free and we can take as zero.
This freedom has to do with the liberty to choose the size of the
frame generating the stable bundle of $\cK$.

\subsection{The KAM theorem}\label{sec:KAM theorem}
The convergence of the iterative process described in Section
\ref{sec:quasi-Newton} leads to an {\it a-posteriori} KAM theorem
on existence of partially hyperbolic invariant tori and its
invariant bundles. In what follows, recall the definitions for
the sets $\U$ and $\tU$ from Section \ref{sec:analytic setting}.

\begin{theorem}[KAM theorem for partially hyperbolic tori]\label{thm:KAM}
\begin{itemize}[leftmargin=*]
Let $\bOm=\dif \bal$ be an exact symplectic form defined in an open set
$U\subset\RR^{2n}$ and let $\bJ$ be an almost complex structure
on $U$ compatible with $\bOm$. Let $\bg$ be the induced metric
and $H:U\times\TT^\ell\to\RR$ be a real-analytic quasi-periodic Hamiltonian
with frequencies $\hatal\in \RR^{\ell}$. For $d=n-2$, let
$\om\in\RR^{d},~T>0,$ and $\al=\hatal T$ be such that
$(\om,\al)$ is Diophantine with constants $\ga>0,$ $\tau\geq
d+\ell$ and suppose we have  $\K:\TT^{d+\ell}\to U,~
\W:\TT^{d+\ell}\to\RR^{2n}$, and $\abs{\la}<1$. Assume the
following hypothesis.

\item [$H_1$]  There exists some $r,R>0$
such that the geometric objects $\bOm,~ \bal,~\bJ,~$ and $\bg$
can be analytically extended to the complex set
$\U\subset\CC^{2n}$ containing $U$, see \eqref{eq:defU}.  Moreover, for all
$t\in[0,T]$ and all $(z,\vp)\in\tU$, $\tphi_t(z,\vp)$ exists.
Then, there exist constants such that:
\begin{itemize}
        \item the matrix representations $\Om,~ a,~ \J,$ and $\G$
                satisfy
\begin{align*}
   \norm{\Om}_{\U} &\leq \cteOm, 
                          & \norm{\Dif \Om}_{\U}&\leq \cteDOm,
                          &\norm{\Da}_{\U} &\leq \cteDa, 
                          & \norm{\DaT}_{\U} &\leq \cteDaT,
\\
   \norm{\DDa}_{\U} &\leq \cteDDa,  
                               & \norm{\J}_{\U} &\leq \cteJ, 
                               & \norm{\J^\ttop}_{\U} &\leq \cteJT,
                               & \norm{\DJ}_\U & \leq \cteDJ,\\
   \norm{\DJT}_{\U} &\leq \cteDJT, 
                    &\norm{G}_{\U} &\leq \cteG, 
                    &\norm{\DG}_\U&\leq \cteDG.
\end{align*}
\item  The Hamiltonian $H$, the vector field $X$, and the flow
        map $\phi_T$ satisfy
\begin{align*}
 \norm{X}_{\tU} &\leq \cteX, 
& \norm{X^\ttop}_{\tU} &\leq \cteXT,
& \norm{\Dif_z X}_{\tU} &\leq \cteDzX,
& \norm{(\Dif_z X)^\ttop}_{\tU} &\leq \cteDzXT,\\
 \norm{\Dif^2_z H}_{\tU} &\leq \cteDDzH, 
& \norm{\DzphiT}_{\tU} &\leq \cteDzphiT,
& \norm{\DzphiTT}_{\tU} &\leq \cteDzphiTT,
& \norm{\DDzphiT}_{\tU} &\leq \cteDDzphiT.
\end{align*}
\end{itemize}

\item [$H_2$]\label{H:2} There exists $\rho\in(0,r)$ such that
   $\K\in\cA(\TT^{d+\ell}_\rho)^{2n}$, $\Dif
   \K\in\cA(\TT^{d+\ell}_\rho)^{2n\times(d+\ell)}$, and
$\W\in\cA(\TT^{d+\ell}_\rho)^{2n}$.
\begin{itemize}
   \item Moreover, there exist constants such that
\begin{align*}
   \norm{\DteK}_{\rho} &< \sigmaDteK,  
   &\norm{\DteKT}_{\rho} &< \sigmaDteKT,
   &\norm{\DvpK}_{\rho} &< \sigmaDvpK, 
   &\norm{\DvpKT}_{\rho} &< \sigmaDvpKT,\\
   \norm{\W}_{\rho} &< \sigmaW, 
   & \norm{\WT}_{\rho} &< \sigmaWT,
   & \norm{\NO}_{\rho} &< \sigmaNO,
   & \norm{\NOT}_{\rho} &< \sigmaNOT,\\
   \norm{\B}_\rho &< \sigmaB.
\end{align*}
\item Also, there exist $\sigmala<1$ and
   $\sigmainvla>1$  satisfying
\begin{equation*}
   \abs{\la} < \sigmala,\ \abs{\invla}<\sigmainvla.
\end{equation*}
\end{itemize}
\item[$H_3$]\label{H:3}
There exists a constant $\sigmainvaverS$ such that the shear
$\S$ in \eqref{eq:reducedS}, which coincides with the
block $\hat S_1$ in
\eqref{eq:splithS}, satisfies the non-degeneracy condition
$\big|\aver{\S}^{-1}\big|<\sigmainvaverS$.
\end{itemize}
Then, for every $\rho_{\infty}\in(0,\rho)$ and
$\de\in(0,\frac{\rho-\rho_\infty}{3})$, there exist a constant
$\fracC$ such that if the invariance errors 
\begin{gather*}
   \EK=\phiK-\K\comp\Romal,\\
   \EW=\DzphiK\W-\W\comp\Romal\la,
\end{gather*}
satisfy
\begin{equation}\label{eq:HKAM}
   \frac{\fracC}{\ga^2\de^{2\tau}}\E<1,
\end{equation} 
where 
\begin{equation}\label{eq:defE}
   \E:=\Max{\frac{\norm{\EK}_\rho}{\ga^2\de^{2\tau}},\
   \norm{\EW}_\rho},
\end{equation}
then there exist $\la_\infty$, an invariant torus
$\cK_\infty=\K_\infty(\TT^{d+\ell})$ 
and an invariant bundle $\cW_\infty=\W_\infty(\TT^{d+\ell})$
satisfying:
\begin{itemize}
   \item[$i)$] $\K_\infty,\W_\infty\in\cA(\TT^{d+\ell}_{\rho_\infty})^{2n}$ and
$\K_{\infty}(\TT^{d+\ell}_{\rho_\infty})\subset\U$.
\item [$ii)$] $\K_\infty, \W_\infty$, and the associated objects
      $\NO_\infty, \B_\infty,$ and $\S_\infty$ satisfy $H_2$ and
      $H_3$ in the complex strip $\TT^{d+\ell}_{\rho_\infty}$.
   \item [$iii)$] The dynamics on $\cK_\infty$ are conjugated to a rigid
      rotation with rotation vector $(\om,\al)$ and on
      $\cW_\infty$ to a constant contraction with rate
      $\abs{\la_\infty}$.
   \item [$iv)$] The limiting objects are close to the original ones in
      the following sense
      \begin{align} \label{eq:KAMK}
      \norm{\K_\infty-\K}_{\rho_\infty}&\leq\fracCDeK\E,\\\label{eq:KAMDteK}
      \norm{\W_\infty-\W}_{\rho_\infty}&\leq \frac{\fracCW}{\ga\de^\tau}\E,\\
      \abs{\la_\infty-\la}&\leq \fracCDela\E,\label{eq:KAMla}\\
      \abs{\invla_\infty-\invla}&\leq
      \fracCDeinvla\E\label{eq:KAMinvla},
      \end{align}
      for some explicit constants $\fracCDeK, \fracCW,
      \fracCDela$, and $\fracCDeinvla$.
\end{itemize}
\end{theorem}

\begin{remark}
In hypothesis $H_2$, it is not strictly necessary to assume
bounds for $\NO$ since we can control them
automatically from the bounds on $\J,\L,$ and $\B$.
Nonetheless, we include it in the hypothesis for the
sake of flexibility in order apply the theorem to particular
problems.  
\end{remark}
\begin{remark}
Note that, since we use the matrix-valued map $\A$ to reduce the
torsion to the form \eqref{eq:reducedS}, we obtain a
parameterization $\widetilde\W_{\infty}$ of another invariant bundle
$\widetilde\cW_{\infty}$ satisfying \eqref{eq:unstablebundle}.
\end{remark}
\begin{remark}
The theorem applies to autonomous systems with minor
modifications---it only requires setting to zero some quantities. 
\end{remark}

\section{Error estimates of approximate geometric
properties}\label{sec:some lemmas}
For the proof of Theorem \ref{thm:KAM}, we first need to control 
certain geometric properties that only hold
approximately when the parameterizations of tori and bundles satisfy
their invariance equations approximately. In this section we
obtain upper bounds for these properties in terms of the error in
the invariance equations. 

\subsection{Adapted frames control}\label{sec:frames control}
Let us begin by controlling the adapted frame
we use throughout the iterative scheme. Let 
\begin{align}\label{eq:defCcX}
   \CcX&:=\cteX + \sigmaDvpK\abs{\hatal}, &
   \CcXT&:=\cteXT + \abs{\hatal^\ttop}\sigmaDvpKT,\\
   \label{eq:defCL}
   \CL&:=\sigmaDteK +   \CcX + \sigmaW, &
   \CLT&:=\Max{\sigmaDteKT,\CcXT,\sigmaWT}.
\end{align}
Then, 
\begin{align*}
   \norm{\cX}_\rho&<\CcX,  & \norm{\cXT}_\rho &<\CcXT,\\
   \norm{\L}_\rho &<\CL, & \norm{\LT}_\rho &<\CLT.
\end{align*}
If we define $\ChP:=\CL+\sigmaNO$ and $\ChPT:=\Max{\CLT,\ \sigmaNOT}$, we can control the frame $\hP$ as
\begin{equation}\label{eq:esthP}
   \norm{\hP}_{\rho}\leq\ChP,\qquad 
   \norm{\hP^\ttop}_\rho\leq\ChPT.
\end{equation}
Similarly, let $\ChS:=\sigmaNOT\cteOm\cteDzphiT\sigmaNO$
and $\ChST:=\sigmaNOT\cteOm\cteDzphiTT\sigmaNO$. Hence we have
the following estimates for $\hS$ and $\hS^\ttop$ 
\begin{equation}\label{eq:esthatS}
   \norm{\hS}_\rho<\ChS,\qquad
   \norm{\hS^\ttop}_\rho<\ChST.
\end{equation}
In order to control the frame $\P$, 
let us first obtain estimates for $\norm{\A}_\rho$  and
$\norm{\AT}_\rho$ as follows
\begin{align*}
   \norm{A}_\rho &\leq\Max{\norm{\Atwo}_\rho,\
      \norm{\Athree}_\rho+  \norm{\Afour}_\rho},\\
      \norm{\AT}_\rho&\leq\Max{\norm{\A_3^\ttop}_\rho,\
      \norm{\A_2^\ttop}_\rho + \norm{\Afour}_\rho}.
\end{align*}
Note that $\Atwo$ and $\Athree$ satisfy \eqref{eq:cohoA2coupled}
and \eqref{eq:cohoA3coupled}, respectively. Therefore, we can use
estimate \eqref{eq:estR2}  and obtain
\begin{align*}
   \norm{\Atwo}_\rho &\leq\frac{1}{\abs{\invla}-\abs{\la}}\left(
   \norm{\hS_2}_\rho\abs{\la} +
   \norm{\hS_3^\ttop}_\rho\abs{\invla}\right), & 
   \norm{\Athree}_\rho &\leq \frac{1}{\abs{\invla}-\abs{\la}}\left(
   \norm{\hS_3}_\rho + \norm{\hS_2^\ttop}_\rho\right),\\
   \norm{\Atwo^\ttop}_\rho &\leq\frac{1}{\abs{\invla}-\abs{\la}}\left(
   \norm{\hS_2^\ttop}_\rho\abs{\la} +
   \norm{\hS_3}_\rho\abs{\invla}\right), & 
   \norm{\Athree^\ttop}_\rho &\leq \frac{1}{\abs{\invla}-\abs{\la}}\left(
      \norm{\hS_3^\ttop}_\rho + \norm{\hS_2}_\rho\right).
\end{align*}
Also, $\Afour$ satisfies \eqref{eq:cohoA4} so
\begin{equation*}
   \norm{\Afour}_\rho\leq\frac{1}{\abs{\invla}-\abs{\la}}
   \norm{\hS_4}_\rho.
\end{equation*}
Then,
\begin{align*}
   \norm{\A}_\rho&\leq\frac{1}{\abs{\invla}-\abs{\la}}
   \Max{\norm{\hS_2}_\rho\abs{\la} +
   \norm{\hS_3^\ttop}_\rho\abs{\invla},\ 
\norm{\hS_3}_\rho + \norm{\hS_2^\ttop}_\rho + \norm{\hS_4}_\rho 
}\\
&\leq\frac{1}{\abs{\invla}-\abs{\la}}
\Max{\norm{\hS}_\rho + \norm{\hS_3^\ttop}_\rho\abs{\invla},\ \norm{\hS}_\rho
+ \norm{\hS_2^\ttop}_\rho\abs{\invla}}\\
&\leq\frac{1}{\abs{\invla}-\abs{\la}}\left(\norm{\hS}_\rho
   + \norm{\hS^\ttop}_\rho\abs{\invla}\right)\\
&\leq\frac{1}{1-\abs{\la}^2}\left(\norm{\hS}_\rho\abs{\la} +
\norm{\hS^\ttop}_\rho\right).
\end{align*}
Similarly, 
\begin{align*}
   \norm{\A^\ttop}_\rho&\leq\frac{1}{\abs{\invla}-\abs{\la}}
   \Max{\norm{\hS_3^\ttop}_\rho +  \norm{\hS_2}_\rho,\ 
      \norm{\hS_2^\ttop}_\rho\abs{\la} +
      \norm{\hS_3}_\rho\abs{\invla} + \norm{\hS_4}_\rho 
}\\
&\leq\frac{1}{\abs{\invla}-\abs{\la}}
\Max{\norm{\hS_3^\ttop}_\rho + \norm{\hS}_\rho\abs{\invla},\ 
   \norm{\hS_2^\ttop}_\rho + \norm{\hS}_\rho
\abs{\invla}}\\
&\leq\frac{1}{\abs{\invla}-\abs{\la}}\left(\norm{\hS^\ttop}_\rho
+ \norm{\hS}_\rho\abs{\invla}\right)\\
&\leq\frac{1}{1-\abs{\la}^2}\left(\norm{\hS^\ttop}_\rho\abs{\la}
   + \norm{\hS}_\rho\right).
\end{align*}
Finally we obtain
\begin{align}\label{eq:estA}
   \norm{A}_\rho &<\frac{1}{1-(\sigmala)^2}\left(\ChS\sigmala
   + \ChST\right)=:\CA,\\\label{eq:estAT}
      \norm{\AT}_\rho &<\frac{1}{1-(\sigmala)^2}
      \left(\ChST\sigmala +\ChS \right) =:\CAT.
\end{align}
We can now control $\N$ and $\NT$ as 
\begin{align}\label{eq:estN}
   \norm{\N}_\rho &\leq \norm{\NO}_\rho + \norm{\L}_\rho \norm{\A}_\rho
   <\sigmaNO +\CL\CA=:\CN,\\\label{eq:estNT}
   \norm{\NT}_\rho &\leq \norm{\NOT}_\rho + 
   \norm{\AT}_\rho\norm{L^\ttop}_\rho
   <\sigmaNOT + \CAT\CLT=:\CNT,
\end{align}
and the frame $\P$ as
\begin{align}\label{eq:estP}
   \norm{\P}_\rho&
   <\CL+\CN=:\CP,\\\label{eq:estPT}
   \norm{\PT}_\rho&
   <\Max{\CLT,\ \CNT}=:\CPT.
\end{align}

\subsection{Approximate invariance of the subframe $\L$}\label{sec:invL}
Let $\EL:\TT^{d+\ell}\to\RR^{2n\times n}$ denote the error in the
invariance of the subframe $\L$. That is, 
\begin{equation}\label{eq:defEL}
   \EL:=\DzphiK \L - \L\comp\Romal\La.
\end{equation}
Also, let us define
\begin{align*}
\De X:=&X\comp\big(\phiT\comp(\K,\id),\Ral\big)
-X\comp\left(\K\comp\Romal,\Ral\right)
\\
=& \int_0^1\Dif_z X\comp\big(\K\comp\Romal +
s\EK,\Ral\big)\EK ds,
\end{align*}
where the last equality only holds in $\tU$
for sufficiently small $\EK$. More specifically, suppose we want to obtain upper
bounds for $\norm{\De X}_{\rho-\de}$. Then, we need that
for all $(\angles)\in\TT^d_{\rho-\de}\times\TT^\ell_{\rho-\de}$
\[
\left(K(\barangles) + s\EK(\angles),\bar \vp\right)\in\tU.
\]
Equivalently, for all $(\angles)\in\TT^d_{\rho-\de}\times\TT^\ell_{\rho-\de}$
and all $s\in[0,1]$, we need some
$\te^*\in\TT^d_{\rho-\de}$ such that 
\[
   \Abs{K(\barangles) + s\EK(\angles) -
   \KO(\te^*,\bar\vp)}<R.
\]
Observe that if 
\begin{equation}\label{eq:ass_EK}
   \frac{\norm{\EK}_\rho}{R-\norm{\K-\KO}_\rho}<1
\end{equation}
holds, we can take $\te^*=\bar\te$ and,
consequently, $\De X$ is well-defined in
$\TT^d_{\rho-\de}\times\TT^\ell_{\rho-\de}$.  
We will encounter the need of hypothesis \eqref{eq:ass_EK} in
order to guarantee that analogous norms for various objects are well
defined. Therefore, in what follows, we assume \eqref{eq:ass_EK}
holds.

\begin{lemma}\label{lemma:invL}
The error in the invariance of the subframe $\L$ 
is given by
\begin{equation}\label{eq:EL}
   \EL = \left(\Dif_\te \EK\ \big|\  \EcX
   \ \big|\  \EW \right),
\end{equation}
where 
\begin{equation*}
   \EcX:=\De X
      -\Dif_\vp \EK \hatal.
\end{equation*}
Additionally, for any $\de\in(0,\rho)$
\begin{align}\label{eq:estEL}
   \norm{\EL}_{\rho-\de}&\leq \frac{\CELK}{\de}\norm{\EK}_{\rho} +
   \CELW \norm{\EW}_{\rho},\\
   \norm{\EL^\ttop}_{\rho-\de}&\leq \frac{\CELTK}{\de}\norm{\EK}_{\rho} +
   \CELTW \norm{\EW}_{\rho}.\label{eq:estELT}
\end{align}

\end{lemma}

\begin{proof}
Let us first differentiate \eqref{eq:defEK} with respect to both $\te$ and $\vp$ to
obtain 
\begin{align}\label{eq:dthEK}
   \Dif_\te \EK&=\DzphiK \Dif_\te \K - \Dif_\te
   \K\comp\Romal,\\
   \Dif_\vp \EK&=\DzphiK\Dif_\vp \K + \Dif_\vp\phiK -
   \Dif_\vp \K\comp\Romal\label{eq:dvpEK}.
\end{align}
The first term in \eqref{eq:EL} is immediate from \eqref{eq:dthEK}.
Also, the last term in \eqref{eq:EL} is precisely definition
\eqref{eq:defEW}.
For the second term, we need to show that 
\[
\DzphiK\cX-\cX\comp\Romal=\EcX.
\]
First observe that at $t=0$
\begin{equation*}
   \frac{\pd}{\pd t}\left(\tilde\phi_T\comp\tilde\phi_t(z,\vp)\right) =
   \begin{pmatrix}
   \Dif_z\phi_T(z,\vp) & \Dif_\vp\phi_T(z,\vp)\\
   \cO_{\ell\times2n} & I_{\ell}
   \end{pmatrix} 
   \begin{pmatrix}  X(x,\vp) 
      \\ \hatal \end{pmatrix} = 
      \begin{pmatrix} X\big(\phi_T(z,\vp),\bar\vp \big) \\ \hatal
      \end{pmatrix}.
\end{equation*}
From the first block row, at $z=K(\angles)$, we
obtain
\[
\DzphiT\big(\K(\angles),\vp\big) X\big(\K(\angles),\vp\big)
  + \Dif_\vp\phiT\big(\K(\angles),\vp\big)\hatal =
X\Big(\phiT\big(\K(\angles),\vp\big),\bar\vp\Big).
\]
Then, after using \eqref{eq:dvpEK} and some manipulation, we
finally obtain
\[
   \DzphiK \cX - \cX\comp\Romal = \De X - \Dif_\vp
   \EK\hatal.
\]
For the estimate \eqref{eq:estEL}, note that for any $\delta\in(0,\rho)$
\begin{equation*}
   \norm{\EcX}_{\rho-\de}\leq \norm{\De X}_{\rho-\de} +
   \norm{\Dif_\vp\EK}_{\rho-\de}\abs{\hatal}
\end{equation*}
and from assumption \eqref{eq:ass_EK}, $\De X$ is well-defined.
Then, we can obtain the estimates
\begin{align}\label{eq:estEcX}
   \norm{\EcX}_{\rho-\de}&\leq \left(\cteDzX \de +
   \ell\abs{\hatal}\right)\frac{\norm{\EK}_\rho}{\de}\leq\CEcX
   \frac{\norm{\EK}_\rho}{\de},\\\label{eq:estEcXT}
     \norm{\EcX^\ttop}_{\rho-\de}&\leq \norm{\left(\De X\right)^\ttop}_{\rho-\de} +
     \Abs{\hatal^\ttop}\norm{\left(\Dif_\vp\EK\right)^\ttop}_{\rho-\de}\\
     \nonumber
    &\leq2n \left(\cteDzXT \de +\Abs{\hatal^\ttop}\right)
    \frac{\norm{\EK}_\rho}{\de}\leq\CEcXT
    \frac{\norm{\EK}_\rho}{\de}.
    \end{align}
Lastly, 
\begin{align*}
   \norm{\EL}_{\rho-\de}&\leq\norm{\Dif_\te \EK}_{\rho-\de} +
   \norm{\EcX}_{\rho-\de} +\norm{\EW}_{\rho-\de}\\
                        &\leq \frac{d}{\de} \norm{\EK}_\rho+
                        \frac{\CEcX}{\de}\norm{\EK}_{\rho}
                         + \norm{\EW}_{\rho}\\
                        &\leq \frac{\left(d+\CEcX
                        \right)}{\de}\norm{\EK}_\rho +
                        \norm{\EW}_{\rho}\\
                        &\leq\frac{\CELK}{\de}\norm{\EK}_\rho
                        + \CELW\norm{\EW}_{\rho},\\
   \norm{\EL^\ttop}_{\rho-\de}&\leq\Max{\norm{\left(\Dif_\te
      \EK\right)^\ttop}_{\rho-\de},\ \norm{\EcX^\ttop}_{\rho-\de},\
   \norm{\EW^\ttop}_{\rho-\de}}\\ 
                              &\leq
                              \Max{\frac{2n}{\de}\norm{\EK}_\rho,\
                              \frac{\CEcXT}{\de}\norm{\EK}_\rho,\
                           2n\norm{\EW}_\rho}\\   
                              &\leq \Max{2n,\
                              \CEcXT}\frac{\norm{\EK}_\rho}{\de} +
                              2n\norm{\EW}_\rho\\
                              &\leq\frac{\CELTK}{\de}\norm{\EK}_\rho
                              + \CELTW\norm{\EW}_\rho.
\end{align*}
\end{proof}

\subsection{Approximate isotropy of $\cK$}
For each $\vp\in\TT^d$, let us denote the coordinate
representation of the pullback of
the symplectic form $\bOm$ by $\K$, as
$\OmDK:\TT^{d+\ell}\to\RR^{2n\times2n}$ given by
\[
   \OmDK:=(\Dif_\theta \K)^\ttop
   \OmK\Dif_\theta \K.
\]
Let us also define 
\begin{align}\label{eq:defDeOm}
   \DeOm:=& \Om\comp\phiT\comp(\K,\id) -
\Om\comp\K\comp\Romal\\\nonumber
=& \int_0^1 \Dif \Om\comp(\K\comp\Romal + s\EK)
\EK ds,
\end{align}
where, similarly as in Section \ref{sec:invL}, the last equality
only holds in $\U$ for sufficiently small $\EK$.
\begin{lemma}\label{lemma:ap_iso}
The approximate isotropy of $\cK$ is given by
$\OmDK = \Rop \left(\Lop\OmDK\right)$,
where 
\begin{align}\label{eq:LieOmK}
   \Lop\OmDK &= \OmDK - \OmDK\comp\Romal\\ \nonumber
   &= (\Dif_\theta\K\comp\Romal)^\ttop\DeOm\Dif_\theta
   \K\comp\Romal + (\Dif_\theta \K\comp\Romal)^\ttop
   \OmphiK \Dif_\te
   \EK\\
   &\phantom{=} +  (\Dif_\te \EK)^\ttop \OmphiK
   \DzphiK \Dif_\te \K\nonumber.
\end{align}
Additionally, for any $\de\in(0,\rho)$
\begin{equation}\label{eq:estOmDK}
   \norm{\OmDK}_{\rho-\de}\leq \frac{\COmDK}{\ga\de^{\tau+1}}
   \norm{\EK}_{\rho}.
\end{equation}
\end{lemma}

\begin{proof}
From the symplecticity of $\phi_T$ and \eqref{eq:dthEK}
\begin{align*}
   \OmDK=&(\Dif_\theta \K)^\ttop \big(\DzphiK\big)^\ttop
   \OmphiK\DzphiK\Dif_\te \K\\
   =& \Big(\Dif_\theta \K\comp\Romal + \Dif_\te \EK \Big)^\ttop
   \OmphiK \Big(\Dif_\theta \K\comp\Romal + \Dif_\te \EK
   \Big) \\
   =&(\Dif_\te \K\comp\Romal)^\ttop \OmphiK \Dif_\theta
   \K\comp\Romal
   + (\Dif_\te \K\comp\Romal)^\ttop \OmphiK
   \Dif_\te \EK\\
   \phantom{=}& + (\Dif_\te \EK)^\ttop \OmphiK
   \DzphiK \Dif_\te \K.
\end{align*}
Then, using the previous expression, the definition of
$\Lop$, and $\eqref{eq:defDeOm}$ we obtain \eqref{eq:LieOmK}.
We also need to show that $\aver{\OmDK}=0$. Observe that 
since $\Om=(\Dif a)^\ttop-\Dif a$
\begin{equation*}
\OmDK=\big(\Dif_\te(a\comp\K)\big)^\ttop \Dif_\te \K -
\big(\Dif_\te\K\big)^\ttop
\Dif_\te (a\comp\K)
\end{equation*}
and each entry $ij$ of $\OmDK$ reads
\begin{align*}
   (\OmDK)_{ij}=&\sum_{k}\partial_{\te^i}\big(a_k\comp\K)\big)
   \partial_{\te^j} \K^k - \partial_{\te^j}\big(a_k\comp\K\big)
   \partial_{\te^i} \K^k\\
   =&\sum_{k} \partial_{\te^i}\big(a_k\comp\K 
   \partial_{\te^j} \K^k\big) - a_k\comp\K
   \partial^2_{\te^i\te^j}\K^k \\
\phantom{=}& - \partial_{\te^j}\big(a_k\comp\K \partial_{\te^i}
   \K^k\big) + a_k\comp\K \partial^2_{\te^i\te^j}\K^k\\
=& \sum_{k} \partial_{\te^i} \big(a_k\comp\K\partial_{\te^j}
   \K^k\big) - \partial_{\te^j}\big(a_k\comp\K \partial_{\te^i}
      \K^k\big).
\end{align*}
Then, $\aver{\OmDK}=0$ since we are taking averages of
derivatives with respect to $\te$. Therefore, $\OmDK =\Rop\left(\Lop\OmDK\right)$.
For the estimate \eqref{eq:estOmDK}, note first that---similarly
as in Lemma \ref{lemma:invL}---$\DeOm$ is
well-defined by assumption \eqref{eq:ass_EK}. Then, we can obtain the
following estimate from \eqref{eq:LieOmK}
\begin{align}\label{eq:estLieOmK}
   \norm{\Lop\OmDK}_{\rho-\de}&\leq\Big(\sigmaDteKT\cteDOm\sigmaDteK\de
   + \sigmaDteKT\cteOm d + 2 n
\cteOm\cteDzphiT\sigmaDteK\Big)\frac{\norm{\EK}_{\rho}}{\de}\\\nonumber
                                       &\leq \CLieOmDK
                                       \frac{\norm{\EK}_{\rho}}{\de},
\end{align}
and, using the Rüssmann estimate \eqref{eq:estR1}, we obtain
\begin{equation*}
   \norm{\OmDK}_{\rho-2\de}\leq
\frac{\CR\CLieOmDK}{\ga\de^{\tau+1}}\norm{\EK}_{\rho}\leq
\frac{\COmDK}{\ga\de^{\tau+1}}\norm{\EK}_{\rho}.
\end{equation*}

\end{proof}

\subsection{Approximate Lagrangianity of the subframe $L$}
The Lagrangianity of the subframe $L$ is given by
$\OmL:\TT^{d+\ell}\to\RR^{n\times n}$ constructed as
\begin{equation*}
   \OmL:=L^\ttop \OmK L
                  =\begin{pmatrix}
                   \OmDK & \OmDKcX & \OmDKW \\
                   \OmcXDK & \OmcX & \OmcXW \\
                    \OmWDK & \OmWcX & \OmW
                    \end{pmatrix},
\end{equation*}
where
\begin{align*}
   \OmDK&\phantom{:}=(\Dif_\theta \K)^\ttop
   \OmK\Dif_\theta \K, & 
   \OmDKcX&:= (\Dif_\te \K)^\ttop\OmK\cX, & 
   \OmDKW&:= (\Dif_\te \K)^\ttop\OmK \W, \\ 
   \OmcXDK&:= \cX^\ttop \OmK \Dif_\te \K, & 
   \OmcX&:= \cX^\ttop \OmK \cX, & 
   \OmcXW&:= \cX^\ttop \OmK \W, \\
   \OmWDK&:= \W^\ttop \OmK \Dif_\te \K, &
   \OmWcX&:= \W^\ttop \OmK \cX, & 
   \OmW&:= \W^\ttop \OmK \W. 
\end{align*}
Observe that because $\Om$ is skew-symmetric, we have 
$\OmcXDK=-\OmDKcX^\ttop$, $\OmWDK=-\OmDKW^\ttop$, and
$\OmWcX=-\OmcXW^\ttop$. Furthermore, 
that using the symplecticity of $\phi_T$, \eqref{eq:defEL}, and
definition \eqref{eq:defDeOm}, we obtain that $\OmL$ is a solution of
the following equation
\begin{align}\label{eq:eq_OmL}
   \OmL- \LaT\OmL\comp\Romal\La =&
\LaT\L^\ttop\comp\Romal \DeOm\L\comp\Romal\La\\\nonumber
&+ \LaT\L^\ttop\comp\Romal\OmphiK\EL\\\nonumber
&+\EL^\ttop\OmphiK\DzphiK\L.
\end{align}

\begin{lemma}\label{lemma:ap_Lag}
The approximate Lagrangianity of the frame $L$,
is given by
\begin{align*}
   \OmDK &=\Rop\big(\Lop \OmDK\big), & 
   \OmDKcX &= \Rop\big(\Lop \OmDKcX\big), & 
   \OmDKW &= \Rop^{1\la}\big( \Lop^{1\la}
   \OmDKW\big),\\
   \OmcXDK &=\Rop\big(\Lop\OmcXDK\big), &
   \OmcX&=0, & 
   \OmcXW&=\Rop^{1\la}\big(\Lop^{1\la}
   \OmcXW\big),\\
   \OmWDK&=\Rop^{1\la} \big(\Lop^{1\la}
   \OmWDK\big), & 
   \OmWcX&=\Rop^{1\la} \big(\Lop^{1\la}
   \OmWcX\big), & 
   \OmW&=0,
\end{align*}
where
\begin{align*}
   \Lop \OmDKcX &=(\Dif_\te \K\comp\Romal)^\ttop
   \DeOm\cX\comp\Romal + (\Dif_\te \K\comp\Romal)^\ttop\OmphiK \EcX\\
   &\phantom{=} + \Dif_\te\EK^\ttop\OmphiK \DzphiK\cX,\\
   \Lop^{1\la}\OmDKW &=  (\Dif_\te \K\comp\Romal)^\ttop
   \DeOm \W\comp\Romal\la + (\Dif_\te\K\comp\Romal)^\ttop\OmphiK \EW \\
   &\phantom{=} + \Dif_\te \EK^\ttop\OmphiK \DzphiK \W,\\
   \Lop\OmcXDK    &=\cX^\ttop\comp\Romal\DeOm\Dif_\te
   \K\comp\Romal + \cX^\ttop\comp\Romal\OmphiK \Dif_\te \EK \\
   &\phantom{=} + \EcX^\ttop\OmphiK \DzphiK \Dif_\te \K,\\
   \Lop^{1\la}\OmcXW &=\cX^\ttop\comp\Romal
   \DeOm \W\comp\Romal\la + \cX^\ttop\comp\Romal\OmphiK \EW \\
   &\phantom{=} + \EcX^\ttop\OmphiK \DzphiK \W,\\
   \Lop^{1\la}\OmWDK &=
   \la\W^\ttop\comp\Romal\DeOm\Dif_\te \K\comp\Romal +
   \la\W^\ttop\comp\Romal \OmphiK \Dif_\te \EK \\
   &\phantom{=} + \EW^\ttop\OmphiK  \DzphiK \Dif_\te \K,\\
   \Lop^{1\la}\OmWcX &=
   \la\W^\ttop\comp\Romal\DeOm\cX\comp\Romal +
   \la\W^\ttop\comp\Romal \OmphiK  \EcX\\
   &\phantom{=}+ \EW^\ttop\OmphiK \DzphiK\cX.
\end{align*}

Additionally, for any $\de\in(0,\rho/2)$
\begin{equation}\label{eq:estOmL}
   \norm{\OmL}_{\rho-2\de}\leq
\frac{\COmLK}{\ga\de^{\tau+1}}\norm{\EK}_{\rho} + \COmLW
\norm{\EW}_{\rho}.
\end{equation} 
\end{lemma}

\begin{proof}
Note that $\OmDK$ is already controlled by Lemma
\ref{lemma:ap_iso}. Also, the expressions for the cohomological operator applied to
each block of $\OmL$ follow directly from \eqref{eq:eq_OmL}.
Additionally, $\OmcX$ is trivially zero and, in
the case of $1-$dimensional bundles, $\OmW$ is also trivially zero.
We also need to show that $\OmDKcX$ has zero average and,
consequently, so does $\OmcXDK$. Observe that 
\[
\OmDKcX=(\Dif_\te\K)^\ttop \OmK\big(X\comp(\K,\id) -
   \Dif_\vp\K\hatal\big).
\]
For the first term, if we use that $\Om(z) X(z,\vp) = \Dif_z H(z,\vp)^\ttop$,
we obtain
\[
\left \langle (\Dif_\te\K)^\top \big(\Dif_z
   H\comp(\K,\id)\big)^\top\right\rangle
 = \Big\langle
 \Big(\Dif_\te\big(H\comp(\K,\id)\big)\Big)^\ttop\Big\rangle=0,
\]
since we are taking averages of derivatives with respect to $\te$. 
The second term, using that $\bOm$ is exact, reads 
\begin{equation*}
\Big\langle (\Dif_\te\K)^\ttop \Dif_\vp
(a\comp \K) -\big(\Dif_\te
(a\comp\K)\big)^\ttop \Dif_\vp \K\Big\rangle \hatal.
\end{equation*}
Then, each entry $i$ of $\aver{\OmDKcX}$ is obtained from
\begin{align*}
   \aver{\OmDKcX}_i&=\sum_j\Big\langle \sum_{k}\partial_{\te^i}\K^k
   \partial_{\vp^j}(a_k\comp\K)-
   \partial_{\te^i}(a_k\comp\K)\partial_{\vp^j}\K^k\Big\rangle\hatal^j\\
&=\sum_j\Big\langle \sum_{k} \partial_{\vp^j}\big(a_k\comp \K
\partial_{\te^i} \K^k \big) - a_k\comp\K\partial^2_{\te^i\vp^j}\K^k\\ 
&\phantom{=} -\partial_{\te^i}\big(a_k\comp\K \partial_{\vp^j}
   \K^k \big) + a_k\comp\K\partial^2_{\te^i\vp^j}
\K^k\Big\rangle\hatal^j\\
 & = 0
\end{align*}
since we are taking averages of derivatives with respect to
$\theta$ and $\varphi$.
For estimate \eqref{eq:estOmL}, note first that 
\begin{align}\label{eq:normOmL}
   \norm{\Om_L}_{\rho-2\de}\leq\max\Big\{&\norm{\OmDK}_{\rho-2\de}
      +\norm{\OmDKcX}_{\rho-2\de}
      +\norm{\OmDKW}_{\rho-\de},\\\nonumber
     & \norm{\OmcXDK}_{\rho-2\de} + \norm{\OmcX}_{\rho-\de} +
   \norm{\OmcXW}_{\rho-\de},\\\nonumber
 &\norm{\OmWDK}_{\rho-\de} + \norm{\OmWcX}_{\rho-\de} +
\norm{\OmW}_{\rho-\de}\Big\}
\end{align}
and that $\DeOm$ is well-defined by assumption \eqref{eq:ass_EK}.
We proceed then by obtaining estimates for each block. Let us
first estimate
\begin{align}\label{eq:estLieOmDKcX}
   \norm{\Lop\OmDKcX}_{\rho-\de}&\leq\left(\sigmaDteKT\cteDOm\CcX\de
    + \sigmaDteKT\cteOm\CEcX + 2n\cteOm\cteDzphiT\CcX\right)
    \frac{\norm{\EK}_{\rho}}{\de}\\\nonumber
    &\leq \frac{\CLieaone}{\de}\norm{\EK}_\rho,\\\label{eq:estLieOmDKW}
   \norm{\Lop^{1\la}\OmDKW}_{\rho-\de}&\leq
   \left(\sigmaDteKT\cteDOm\sigmaW\sigmala
      \de +2n\cteOm\cteDzphiT\sigmaW\right)\frac{\norm{\EK}_{\rho}}{\de} +
      \left(\sigmaDteKT \cteOm
      \right)\norm{\EW}_{\rho}\\\nonumber
    &\leq\frac{\CLieatwoK}{\de}\norm{\EK}_{\rho} +
    \CLieatwoW\norm{\EW}_{\rho},\\\label{eq:estLieOmcXDK}
    \norm{\Lop\OmcXDK}_{\rho-\de}&\leq\left(\CcXT\cteDOm\sigmaDteK\de
    + \CcXT\cteOm d + \CEcXT\cteOm\cteDzphiT\sigmaDteK\right)
    \frac{\norm{\EK}_{\rho}}{\de}\\\nonumber
     &\leq\frac{\CLieathree}{\de}\norm{\EK}_\rho,\\\label{eq:estLieOmcXW}
    \norm{\Lop^{1\la}\OmcXW}_{\rho-\de}&\leq\left(\CcXT\cteDOm\sigmaW\sigmala\de
    + \CEcXT\cteOm\cteDzphiT\sigmaW\right)\frac{\norm{\EK}_\rho}{\de} +
    \left(\CcXT\cteOm\right)\norm{\EW}_\rho\\\nonumber
  &\leq\frac{\CLieafiveK}{\de}\norm{\EK}_{\rho}
  + \CLieafiveW\norm{\EW}_{\rho},\\\label{eq:estLieOmWDK}
     \norm{\Lop^{1\la}\OmWDK}_{\rho-\de}&\leq
     \left(\sigmala\sigmaWT\cteDOm\sigmaDteK\de
     + \sigmala\sigmaWT\cteOm d\right) \frac{\norm{\EK}_\rho}{\de} +
     \left(2n\cteOm\cteDzphiT\sigmaDteK\right)\norm{\EW}_{\rho}\\\nonumber
   &\leq\frac{\CLieasixK}{\de}\norm{\EK}_{\rho}
   + \CLieasixW\norm{\EW}_{\rho}\\     \label{eq:estLieOmWcX}
     \norm{\Lop^{1\la}\OmWcX}_{\rho-\de}&\leq\left(\sigmala\sigmaWT\cteDOm\CcX\de
     + \sigmala\sigmaWT\cteOm\CEcX\right)\frac{\norm{\EK}_\rho}{\de} +
     \left(2n\cteOm\cteDzphiT\CcX\right)\norm{\EW}_{\rho}\\\nonumber
     &\leq\frac{\CLieasevenK}{\de}\norm{\EK}
      + \CLieasevenW\norm{\EW}_{\rho},
\end{align}
where we used Lemma \ref{lemma:invL}. 

We are now in the position to use Rüssmann estimates \eqref{eq:estR1} and
\eqref{eq:estR2} and obtain
\begin{align}\label{eq:estOmDKcX}
   \norm{\OmDKcX}_{\rho-2\de}&\leq \frac{\CR
   \CLieaone}{\ga\de^{\tau+1}}\norm{\EK}_\rho
   \leq\frac{\Caone}{\ga\de^{\tau+1}}\norm{\EK}_{\rho},\\
   \label{eq:estOmDKW}
      \norm{\OmDKW}_{\rho-\de}&\leq
      \frac{1}{1-\sigmala}\left(\frac{\CLieatwoK}{\de}\norm{\EK}_{\rho}
      + \CLieatwoW\norm{\EW}_{\rho}\right)\\\nonumber
&\leq\frac{\CatwoK}{\de}\norm{\EK}_{\rho}
+ \CatwoW\norm{\EW}_{\rho},\\\label{eq:estOmcXDK}
\norm{\OmcXDK}_{\rho-2\de}&\leq  \frac{\CR
      \CLieathree}{\ga\de^{\tau+1}}\norm{\EK}_\rho
      \leq\frac{\Cathree}{\ga\de^{\tau+1}}\norm{\EK}_{\rho},\\\label{eq:estOmcXW}
   \norm{\OmcXW}_{\rho-\de}&\leq \frac{1}{1-\sigmala}\left(\frac{\CLieafiveK}{\de}
\norm{\EK}_{\rho} + \CLieafiveW\norm{\EW}_{\rho}\right)\\\nonumber
  &\leq\frac{\CafiveK}{\de} \norm{\EK}_{\rho}+ 
  \CafiveW\norm{\EW}_{\rho}, \\\label{eq:estOmWDK}
      \norm{\OmWDK}_{\rho-\de}&\leq \frac{1}{1-\sigmala}\left(\frac{\CLieasixK}{\de}
\norm{\EK}_{\rho} + \CLieasixW\norm{\EW}_{\rho}\right)\\\nonumber
  &\leq\frac{\CasixK}{\de}
  \norm{\EK}_{\rho}+
  \CasixW\norm{\EW}_{\rho},\\\label{eq:estOmWcX}
         \norm{\OmWcX}_{\rho-\de}&\leq \frac{1}{1-\sigmala}\left(
            \frac{\CLieasevenK}{\de}\norm{\EK}_{\rho}
  +  \CLieasevenW\norm{\EW}_{\rho}\right)\\\nonumber
&\leq\frac{\CasevenK}{\de} \norm{\EK}_{\rho}+
  \CasevenW\norm{\EW}_{\rho}.
\end{align}
Lastly, from \eqref{eq:normOmL}, we obtain
\begin{align*}
   \norm{\OmL}_{\rho-2\de}&\leq\max\Big\{\left(\COmDK + \Caone +
   \CatwoK\ga\de^\tau\right)\frac{\norm{\EK}_\rho}{\ga\de^{\tau+1}} +
\CatwoW\norm{\EW}_\rho,   \\
&\phantom{\leq\max\Big\{} \left(\Cathree+\CafiveK\ga\de^\tau\right)
   \frac{\norm{\EK}_\rho}{\ga\de^{\tau+1}} +
    \CafiveW\norm{\EW}_\rho, \\
&\phantom{\leq\max\Big\{} \left(\CasixK + \CasevenK\right)
   \frac{\norm{\EK}_\rho}{\de} +
       \left(\CasixW +  \CasevenW\right)\norm{\EW}_\rho\Big\}\\
&\leq \max\Big\{\COmDK + \Caone + \CatwoK\ga\de^\tau,\
   \Cathree+\CafiveK\ga\de^\tau,\ \\
&\phantom{\leq\max\Big\{}
\Big(\CasixK+ \CasevenK\Big)\ga\de^\tau \Big\} \frac{\norm{\EK}_\rho}{\ga\de^{\tau+1}}\\ 
&\phantom{\leq}+ \max\Big\{\CatwoW,\ \CafiveW,\
   \CasixW+\CasevenW\Big\}\norm{\EW}_\rho\\
&\leq \frac{\COmLK}{\ga\de^{\tau+1}}\norm{\EK}_{\rho}
 + \COmLW \norm{\EW}_{\rho}.
\end{align*}

\end{proof}

\subsection{Approximate symplecticity of the frame $\hP$}
Let $\hEsym:\TT^{d+\ell}\to\RR^{2n\times2 n}$ denote the error in the
symplecticity of $\hP$. That is, 
\begin{equation}\label{eq:defhEsym}
\hEsym:=\hP^\ttop\OmK \hP - \OmO  
 =\begin{pmatrix} \L^\ttop\OmK \L &
\L^\ttop\OmK \NO + I_n \\
\NO^\ttop\OmK \L -I_n &
\NO^\ttop\OmK \NO \end{pmatrix}.
\end{equation}

\begin{lemma}\label{lemma:ap_sym_hP}
The symplecticity error $\hEsym$ 
is given by
\begin{equation}\label{eq:exphEsym}
   \hEsym =
   \left(\begin{array}{c|c} \OmL & \\ 
   \hline & \mybox{\B^\ttop\OmL\B}\end{array}\right).
\end{equation}
Additionally, for any $\de\in(0,\rho/2)$ 
\begin{equation}\label{eq:esthEsym}
   \norm{\hEsym}_{\rho-2\de}\leq
   \frac{\ChEsymK}{\ga\de^{\tau+1}}\norm{\EK}_\rho
   + \ChEsymW\norm{\EW}_\rho.
\end{equation}
\end{lemma}

\begin{proof}
Let us compute
\begin{align*}
   \L^\ttop\OmK \NO&=\L^\ttop\OmK\Big(\NO
   + \L\A\Big) = -I_n,\\
   \NO^\ttop\OmK \L&=I_n,\\
   \NO^\ttop\OmK \NO &= \B^\ttop\OmL\B,
\end{align*}
where we used \eqref{eq:defN0}. From these expressions and the definition of
$\hEsym$, \eqref{eq:exphEsym} follows directly. In order to obtain
estimate \eqref{eq:esthEsym} notice that
\begin{equation*}
   \norm{\hEsym}_{\rho-2\de}\leq\norm{\OmL}_{\rho-2\de}
   \max\Big\{1,\ \norm{B}_\rho^2\Big\}.
\end{equation*}
We can now use Lemma \ref{lemma:ap_Lag} to obtain
\begin{align*}
   \norm{\hEsym}_{\rho-2\de}&\le
   \norm{\OmL}_{\rho-2\de}\max\Big\{1,\ (\sigmaB)^2 \Big\} \\
   &\leq\left(\frac{\COmLK}{\ga\de^{\tau+1}}\norm{\EK}_{\rho}
   + \COmLW \norm{\EW}_{\rho}\right)\max\Big\{1,\ (\sigmaB)^2  \Big\}\\
   &\leq \frac{\ChEsymK}{\ga\de^{\tau+1}}\norm{\EK}_\rho
   + \ChEsymW\norm{\EW}_\rho.
\end{align*}
\end{proof}

\subsection{Approximate reducibility of the frame
$\hP$}\label{sec:ap_red_hP}
Let us define the error in the reducibility of the linearized
dynamics in the coordinates given by the frame $\hP$
as $\Ered:\TT^{d+\ell}\to\RR^{2n\times 2n}$ given by
\begin{equation}\label{eq:defhEred}
\hEred :=- \OmO  \hP^\ttop\comp\Romal\OmbarK \DzphiK
\hP - \left(\begin{array}{c|c}
   \La & \hS\\
   \hline
       & \mybox{\La^{-\ttop}}
\end{array}\right).
\end{equation}
First, let 
\begin{equation*}
   \renewcommand*{\arraystretch}{1.4}
   \hEred=\begin{pmatrix} \hEred^{11} & \hEred^{12}\\ \hEred^{21}
      & \hEred^{22}\end{pmatrix}
\end{equation*}
where
\begin{align*}
   \hEred^{11} &=  \NO^\ttop\comp\Romal\OmbarK \DzphiK
   \L-\La,\\
   \hEred^{12} &=  \NO^\ttop\comp\Romal\OmbarK \DzphiK
   \NO-\hS,\\
   \hEred^{21} &= -\L^\ttop\comp\Romal\OmbarK \DzphiK
   \L, \\
   \hEred^{22} &= -\L^\ttop\comp\Romal\OmbarK \DzphiK
   \NO - \La^{-\ttop}.
\end{align*}

\begin{lemma}\label{lemma:ap_red_hP}
The reducibility error $\hEred$,
is given by
\begin{align*}
   \hEred^{11} &= \NO^\ttop\comp\Romal\OmbarK  \EL,\\
   \hEred^{12} &=  \cO_n\\
   \hEred^{21} &= -\OmL\comp\Romal\La -
   \L^\ttop\comp\Romal\OmbarK \EL,\\
   \hEred^{22} &= \L^\ttop\comp\Romal\DeOm\DzphiK \NO\\
               &\phantom{=}+ \La^{-\ttop} \EL^\ttop\OmphiK \DzphiK \NO.
\end{align*}
Additionally, for any $\de\in(0,\rho/2)$ 
\begin{equation}\label{eq:esthEred}
   \norm{\hEred}_{\rho-2\de}\leq
   \frac{\ChEredK}{\ga\de^{\tau+1}}\norm{\EK}_\rho +
   \ChEredW \norm{\EW}_\rho.
\end{equation}

\end{lemma}

\begin{proof}
First, recall the definition \eqref{eq:defshat} for $\hS$. Then,
$\hEred^{12}$ is identically zero. For $\hEred^{11}$,
let us use \eqref{eq:defEL} to obtain
\begin{align*}
   \hEred^{11} &=
   \NO^\ttop\comp\Romal \OmbarK \left(\L\comp\Romal\La +
   \EL\right)-\La\\
   &= \NO^\ttop\comp\Romal \OmbarK  \EL
\end{align*}
Similarly, for $\hEred^{21}$ we have
\begin{align*}
   \hEred^{21} &=
   -\L^\ttop\comp\Romal\OmbarK  \DzphiK \L\\
   &=-\L^\ttop\comp\Romal\OmbarK  \left(\L\comp\Romal\La +
   \EL\right)\\
   &=-\OmL\comp\Romal\La - \L^\ttop\comp\Romal \OmbarK
   \EL.
\end{align*}
Lastly, using the symplecticity of $\phi_T$,
and \eqref{eq:defEL} we have
\begin{align*}
   \hEred^{22} &= -\La^{-\ttop}L^{\ttop}\DzphiKT
   \OmbarK\DzphiK \NO \\
      &\phantom{=}+  \La^{-\ttop}\EL^{\ttop}\OmbarK\DzphiK
      \NO -\La^{-\ttop}\\
      &=\La^{-\ttop}\L^{\ttop}\DzphiKT
      \DeOm\DzphiK \NO \\
      &\phantom{=} +\La^{-\ttop}\EL^{\ttop}\OmbarK\DzphiK
      \NO \\
      &=\L^\ttop\comp\Romal\DeOm\DzphiK \NO \\
      &\phantom{=} + \La^{-\ttop}\EL^\ttop\OmphiK \DzphiK \NO.
\end{align*}
For estimate \eqref{eq:esthEred}, first note that
$\DeOm$ is well-defined by assumption \eqref{eq:ass_EK}, that
\begin{align*}
   \norm{\La}&=\Max{1,\ \abs{\la}} \leq\Max{1,\ \sigmala}=1,\\
   \norm{\La^{-1}}&=\Max{1,\ \abs{\invla}}\leq\Max{1,\
   \sigmainvla}=\sigmainvla,
\end{align*}
and
\[
   \norm{\hEred}_{\rho-2\de}\leq\Max{\norm{\hEred^{11}}_{\rho-2\de} 
         ,\  \norm{\hEred^{21}}_{\rho-2\de} +
               \norm{\hEred^{22}}_{\rho-2\de}}.
\]
We proceed then by computing estimates for each block
$\hEred^{ij}$ as 
\begin{align}\label{eq:esthEred11}
   \norm{\hEred^{11}}_{\rho-\de} & \leq \sigmaNOT\cteOm\CELK
   \frac{\norm{\EK}_{\rho}}{\de} +
    \sigmaNOT\cteOm\CELW\norm{\EW}_\rho
   \\\nonumber
   &\leq\frac{\ChEredaaK}{\de}\norm{\EK}_\rho
   + \ChEredaaW\norm{\EW}_\rho, \\\label{eq:esthEred21}
   \norm{\hEred^{21}}_{\rho-2\de} & \leq
   \Big(\COmLK + \CLT\cteOm\CELK\ga\de^\tau\Big)
   \frac{\norm{\EK}_\rho}{\ga\de^{\tau+1}} +
   \Big(\COmLW + \CLT\cteOm\CELW\Big)\norm{\EW}_\rho \\\nonumber
      &\leq\frac{\ChEredbaK}{\ga\de^{\tau+1}}\norm{\EK}_\rho
      + \ChEredbaW\norm{\EW}_\rho,\\\label{eq:esthEred22}
   \norm{\hEred^{22}}_{\rho-\de} & \leq
   \Big(\CLT\cteDOm\cteDzphiT\sigmaNO\de +
   \sigmainvla\CELTK\cteOm\cteDzphiT\sigmaNO\Big)
   \frac{\norm{\EK}_\rho}{\de}
\\\nonumber
   &\phantom{<}+ \sigmainvla\CELTW\cteOm\cteDzphiT\sigmaNO\norm{\EW}_\rho \\\nonumber
   &\leq\frac{\ChEredbbK}{\de}\norm{\EK}_\rho
   + \ChEredbbW\norm{\EW}_\rho,
\end{align}
where we used Lemmas \ref{lemma:invL} and \ref{lemma:ap_Lag}.
Finally,
\begin{align*}
   \norm{\hEred}_{\rho-2\de}&\leq\max\Big\{\ChEredaaK\frac{\norm{\EK}_\rho}{\de} +
\ChEredaaW\norm{\EW}_\rho,\ \Big(\ChEredbaK +
\ChEredbbK\ga\de^\tau\Big)\frac{\norm{\EK}_\rho}{\ga\de^{\tau+1}}\\
 & \phantom{\leq}+ \Big(\ChEredbaW + \ChEredbbW\Big)\norm{\EW}_\rho\Big\}\\
&\leq\max\Big\{\ChEredaaK\ga\de^\tau,\ \ChEredbaK + \ChEredbbK\ga
\de^\tau\Big\}\frac{\norm{\EK}_\rho}{\ga\de^{\tau+1}} \\
&\phantom{\leq}+\max\Big\{\ChEredaaW,\ \ChEredbaW+\ChEredbbW\Big\}\norm{\EW}_\rho\\
&\leq\frac{\ChEredK}{\ga\de^{\tau+1}}\norm{\EK}_\rho +
                           \ChEredW\norm{\EW}_\rho.
\end{align*}
\end{proof}

\subsection{Approximate invertibility of $\hP$}\label{sec:invhP}
Let $\EinvhP:\TT^{d+\ell}\to\RR^{2n\times2n}$
be the error in the invertibility of the frame
$\hP$ under the assumption that $\hP$ is symplectic. That is,
\begin{equation}\label{eq:defEinvhP}
   \EinvhP:=\hP\left(-\OmO\hP^\ttop\OmK\right) -I_{2n}.
\end{equation}
In what follows, we will need certain matrix to be sufficiently
close to the identity. In particular, we will need $I_{2n}-\OmO\hEsym$
to be invertible in $\TT^{d+\ell}_{\rho-2\de}$. This can be
ensured if
\begin{equation}\label{eq:ass_hEsym}
   \norm{\hEsym}_{\rho-2\de}\leq\hnu<1.
\end{equation}
For instance, if \eqref{eq:ass_hEsym} holds, then $\hP^{-1}$ is
invertible with inverse
\[
   \hP^{-1} =-\left(I_{2n}-\OmO\hEsym\right)^{-1}\OmO\hPT\OmK.
\]
In the remaining of this section, we will assume
\eqref{eq:ass_hEsym} holds.

\begin{lemma}\label{lemma:ap_inv_hP}
The error in the invertibility of $\hP$ is given by
\begin{equation}\label{eq:EinvhP}
\EinvhP=
\hP\OmO\hEsym\left(I-\OmO\hEsym\right)^{-1}\OmO\hP^\ttop\OmK.
\end{equation}
Additionally, for any $\de\in(0,\rho/2)$
\begin{equation}\label{eq:estEinvhP}
   \norm{\EinvhP}_{\rho-2\de}
  \leq\frac{\CEinvhPK}{\ga\de^{\tau+1}}\norm{\EK}_\rho
  + \CEinvhPW\norm{\EW}_\rho.
\end{equation}
\end{lemma}
\begin{proof}
Note that 
\[
-\OmO\hP^\ttop\OmK\hP = I - \OmO\hEsym.
\]
Therefore, we obtain
\begin{align*}
   \hP\left(-\OmO\hP^\ttop\OmK\right) &=
\hP\big(I-\OmO\hEsym\big)\hP^{-1} \\
&= I + \hP\OmO\hEsym\left(I-\OmO\hEsym\right)^{-1}\OmO\hP^\ttop\OmK
\end{align*}
and 
\begin{equation*}
\EinvhP=
\hP\OmO\hEsym\left(I-\OmO\hEsym\right)^{-1}\OmO\hP^\ttop\OmK.
\end{equation*}
For estimate \eqref{eq:estEinvhP}, since
$\norm{\hEsym}_{\rho-2\de}\leq\hnu<1$ by assumption,
\begin{equation*}
   \Norm{\left(I_{2n}-\OmO\hEsym\right)^{-1}}_{\rho-2\de}\leq
\sum_{k=0}^\infty \Norm{\hEsym}_{\rho-2\de}^k\leq
\frac{1}{1-\norm{\hEsym}_{\rho-2\de}}\leq\frac{1}{1-\hnu}.
\end{equation*}
Then, from \eqref{eq:EinvhP} and Lemma \ref{lemma:ap_sym_hP}, we
obtain
\begin{align*}
   \norm{\EinvhP}_{\rho-2\de}&\leq\frac{1}{1-\hnu}
   \CP\CPT\cteOm\norm{\hEsym}_{\rho-2\de}\\
  &\leq\frac{\CEinvhPK}{\ga\de^{\tau+1}}\norm{\EK}_\rho
  + \CEinvhPW\norm{\EW}_\rho.
\end{align*}

\end{proof}
\subsection{Approximate symmetries of $\La^{-1}\hS$ and
$\A$}\label{sec:ap_symmetry}
The use of the matrix-valued map $\A$ to reduce the torsion is
equivalent to applying the following transformation to the frame $\hP$ 
\begin{equation}\label{eq:transformhP}
   \P=\hP\left(\begin{array}{c|c}
   I_n & \A\\\hline
    & I_n
   \end{array}\right).
\end{equation}
This transformation is symplectic---and consequently so is
$\P$---if $\A$ is symmetric. Recall
that according to remark \ref{rem:hSsym}, $\A$ is symmetric if 
the symmetry condition \eqref{eq:hSsym} holds, which only
holds approximately when $\K$ and $\W$ satisfy \eqref{eq:invK} and
\eqref{eq:invW} approximately. In order to control the
approximate geometric properties of the frame $\P$, we need to
control the error in the symmetry of $\A$ which in turn is
controlled by the error in
the symmetry condition \eqref{eq:hSsym}.
Let us define the following errors in the symmetry of
$\La^{-1}\hS$ and $\A$
$\EsyminvLahS, \EsymA:\TT^{d+\ell}\to\RR^{n\times n}$ be defined as
\begin{align*}
   \EsyminvLahS &= \La^{-1}\hS-\big(\La^{-1}\hS\big)^{\ttop},\\
   \EsymA & = \A-\A^\ttop.
\end{align*}

\begin{lemma}\label{lemma:esthSsym}
For any $\de\in(0,\rho/2)$ the torsion $\hS$ satisfies
\begin{equation}\label{eq:esthSsym}
\norm{\EsyminvLahS}_{\rho-2\de} \leq
  \frac{\CinvLahSK}{\ga\de^{\tau+1}}\norm{\EK}_\rho +
 \CinvLahSW\norm{\EW}_\rho.
\end{equation}
\end{lemma}
\begin{proof}
Let us inspect the symplecticity of the reduced dynamics
under the frame $\hP$ with respect to
$\OmO$, i.e., the symplecticity of the following matrix-valued
map
\begin{equation*}
\hLa:= \left(\begin{array}{c|c}\La & \hS\\\hline &
\mybox{\La^{-\ttop}}\end{array}\right).
\end{equation*}
A straightforward computation reveals
\begin{align}\label{eq:EsymhLa}
   \EsymhLa:&=\hLa^\ttop\OmO\hLa -\OmO\\\nonumber
   & =\left(\begin{array}{c|c}
    & \\\hline
   \rule{0pt}{1.2em} & \mybox{\La^{-1}\hS - \big(\La^{-1}\hS\big)^{\ttop}}
   \end{array}\right).
\end{align}
Then, from \eqref{eq:defhEred} and  \eqref{eq:defEinvhP}
we obtain  
\begin{align*}
   \EsymhLa =& -\OmO- 
   \hPT\DzphiKT\OmbarK\hP\comp\Romal\OmO\hPT\comp\Romal\OmbarK\DzphiK\hP\\
   & -\hPT\DzphiKT\OmbarK\hP\comp\Romal\hEred -
    \hEred^\ttop\OmO\hLa\\
   =&-\OmO +
   \hPT\DzphiKT\big(\OmphiK-\DeOm\big)\big(I+\EinvhP\comp\Romal\big)\DzphiK\hP\\
    & -\hPT\DzphiKT\OmbarK\hP\comp\Romal\hEred -
   \hEred^\ttop\OmO\hLa,
\end{align*}
and using the symplecticity of $\phi_T$ and \eqref{eq:defhEsym}, we
have
\begin{align*}
   \EsymhLa =& \hEsym -
\hEred^\ttop\OmO\hLa-\hPT\DzphiKT\DeOm\DzphiK\hP\\
 &+\hPT\DzphiKT\OmbarK\EinvhP\comp\Romal\DzphiK\hP\\
 &-\hPT\DzphiKT\OmbarK\hP\comp\Romal\hEred.
\end{align*}
Lastly, from \eqref{eq:EsymhLa} and lemmas
\ref{lemma:ap_sym_hP}, \ref{lemma:ap_red_hP} and
\ref{lemma:ap_inv_hP}, we obtain
\begin{align}\label{eq:symhS}
   \EsyminvLahS=&\BT\OmL\B -
   \hEred^{22^\ttop}\hS -\NOT\DzphiKT\DeOm\DzphiK\NO\\\nonumber
  &+\NOT\DzphiKT\OmbarK\EinvhP\comp\Romal\DzphiK\NO\\\nonumber
  &-\NOT\DzphiKT\OmbarK\hP\Ered^{22}
\end{align}
from where estimate \eqref{eq:esthSsym} follows.
\end{proof}

\begin{lemma}\label{lemma:estA-AT}
Let $\A:\TT^{d+\ell}\to\RR^{n\times n}$
satisfy \eqref{eq:cohoA2}-\eqref{eq:cohoA4}, where
\begin{equation*}
   \A =
   \begin{pmatrix}
      \cO_{n-1} & \Atwo\\
      \Athree & \Afour
   \end{pmatrix}.
\end{equation*}
Then, for any $\de\in(0,\rho/2)$, $\A$ satisfies
\begin{equation}\label{eq:estA-AT}
   \norm{\EsymA}_{\rho-2\de}\leq
   \frac{\CEsymAK}{\ga\de^{\tau+1}}\norm{\EK}_\rho +
   \CEsymAW\norm{\EW}_\rho.
\end{equation}
\end{lemma}
\begin{proof}
First, notice that 
\[
\A-\A^\ttop = \left( 
\begin{array}{c|c}
    & \Atwo - \Athree^\ttop\\
   \hline
   \mybox{\Athree - \Atwo^\ttop} & 
\end{array}\right).
\]
If we multiply \eqref{eq:cohoA2} by $\la$  and subtract 
the transpose of $\eqref{eq:cohoA3}$, we obtain
\begin{equation*}
   \Lop^{-\la1}\left(\Atwo-\Athree^\ttop\right)=\hS_2\la-\hS_3^\ttop.
\end{equation*}
Similarly, we can multiply the transpose of \eqref{eq:cohoA2} by
$\la$ and subtract it to \eqref{eq:cohoA3} to obtain
\begin{equation*}
   \Lop^{-\la 1}\left(\Athree-\Atwo^\ttop\right)
   =\hS_3-\la\hS_2^\ttop.
\end{equation*}
Therefore
\begin{align*}
   \norm{\Atwo-\Athree^\ttop}_{\rho-2\de}&\leq
   \frac{1}{1-\abs{\la}}\norm{\hS_2\la-\hS_3^\ttop}_{\rho-2\de}\leq
   \frac{1}{1-\abs{\la}}\norm{\hS\La^\ttop-\La\hS^\ttop}_{\rho-2\de},\\
   \norm{\Athree-\Atwo^\ttop}_{\rho-2\de}&\leq
   \frac{1}{1-\abs{\la}}\norm{\hS_3-\la\hS_2^\ttop}_{\rho-2\de}\leq
   \frac{1}{1-\abs{\la}}\norm{\hS\La^\ttop-\La\hS^\ttop}_{\rho-2\de},
\end{align*}
from where we obtain the estimate 
\begin{equation*}
   \norm{\EsymA}_{\rho-2\de}\leq
   \frac{1}{1-\abs{\la}}\norm{\hS\La^\ttop-\La\hS^\ttop}_{\rho-2\de}
   \leq\frac{1}{1-\sigmala}\norm{\EsyminvLahS}_{\rho-2\de},
\end{equation*}
where we used that $\norm{\La}$=1.
We can now use Lemma \ref{lemma:esthSsym} and obtain
\eqref{eq:estA-AT}.

\end{proof}

\subsection{Approximate symplecticity of the frame
$\P$}\label{sec:ap_symp}
Let $\Esym:\TT^{d+\ell}\to\RR^{2n\times2 n}$ denote the error in the
symplecticity of $\P$. That is, 
\begin{equation}\label{eq:defEsym}
\Esym:=\P^\ttop\OmK \P - \OmO.  
\end{equation}
\begin{lemma}\label{lemma:ap_sym}
For any $\de\in(0,\rho/2)$ the symplecticity error $\Esym$ 
satisfies
\begin{equation}\label{eq:estEsym}
   \norm{\Esym}_{\rho-2\de}\leq
   \frac{\CEsymK}{\ga\de^{\tau+1}}\norm{\EK}_\rho
   + \CEsymW\norm{\EW}_\rho.
\end{equation}
\end{lemma}

\begin{proof}
From definition \eqref{eq:defEsym} and transformation
\eqref{eq:transformhP} we obtain
\begin{equation*}
\Esym= \left(\begin{array}{c|c} I_n & \\\hline \mybox{\AT} &
   I_n\end{array}\right) \hEsym
\left(\begin{array}{c|c} I_n & \A\\\hline &
I_n\end{array}\right)+
\left(\begin{array}{c|c}  & \\\hline & \mybox{\A-\AT}\end{array}\right).
\end{equation*}
Consequently, 
\begin{align*}
   \norm{\Esym}_{\rho-2\de}&\leq
   \left(1+\norm{\AT}_\rho\right)\left(1+\norm{\A}_\rho\right)
   \norm{\hEsym}_{\rho-2\de} + \norm{\EsymA}_{\rho-2\de}\\
   &\leq \left(1+\CAT\right)\left(1+\CA\right)
    \norm{\hEsym}_{\rho-2\de} + \norm{\EsymA}_{\rho-2\de}.
\end{align*}
Then, using Lemmas \ref{lemma:ap_sym_hP} and \ref{lemma:estA-AT},
result \eqref{eq:estEsym} is a direct computation.
\end{proof}

\subsection{Approximate reducibility of the frame $\P$}\label{sec:ap_red}
Let us now define the error in the reducibility of the linearized
dynamics in the coordinates given by the frame $\P$ as
$\Ered:\TT^{d+\ell}\to\RR^{2n\times 2n}$ given by 
\begin{equation}\label{eq:defEred}
\Ered:= - \OmO  \P^\ttop\comp\Romal\OmbarK \DzphiK
\P - \left(
\begin{array}{c|c}
   \La & \tS\\
   \hline
       & \mybox{\La^{-\ttop}}
\end{array}
\right).
\end{equation}
\begin{lemma}\label{lemma:ap_red}
For any $\de\in(0,\rho/2)$ the reducibility error $\Ered$
satisfies
\begin{equation}\label{eq:estEred}
   \norm{\Ered}_{\rho-2\de}\leq
   \frac{\CEredK}{\ga\de^{\tau+1}}\norm{\EK}_\rho +
   \CEredW \norm{\EW}_\rho.
\end{equation}
\end{lemma}

\begin{proof}
As in the proof of Lemma \ref{lemma:ap_sym}, we use definition
\eqref{eq:defEred} and transformation \eqref{eq:transformhP} to
obtain
\begin{equation*}
\Ered=
\left(\begin{array}{c|c} I_n & -\AT\comp\Romal\\\hline  &
   I_n\end{array}\right)
   \hEred
\left(\begin{array}{c|c} I_n & \A\\\hline  &
   I_n\end{array}\right),
\end{equation*}
where we used that $\A$ solves \eqref{eq:tS-hS}.
Then, 
\begin{align*}
   \norm{\Ered}_{\rho-2\de}&\leq  \left(1+\norm{\AT}_\rho\right)
   \left(1+\norm{\A}_{\rho-2\de}\right)\norm{\hEred}_{\rho-2\de}\\
   &\left(1+\CAT\right)\left(1+\CA\right)\norm{\hEred}_{\rho-2\de},
\end{align*}
from where we obtain \eqref{eq:estEred}.
\end{proof}

\section{Torus and bundle corrections}\label{sec:tori bundle
correction}
In this section, we provide two lemmas to control the error in
the torus and bundle parameterizations after one step of the
iterative scheme. The lemmas rely on the bounds for the
approximate geometric properties obtained in Section
\ref{sec:some lemmas}.
\subsection{The linearized equation for the torus}\label{sec:lin-eqK}
Let us now inspect the new error in the invariance equation after the
iterative step of Section \ref{sec:quasi-Newton}, i.e., after choosing $\xiK$ that solves
\eqref{eq:corrtr1}-\eqref{eq:corrtr4}. A straightforward calculation reveals
\begin{align}
   \EKnew:=&\phi_T\comp(\K + \DeK,\id) -
   \K\comp\Romal- \DeK\comp\Romal\label{eq:Enew1}\\
   =& \EK+ \DzphiK \DeK - \De
   \K\comp\Romal + \DeDephiT\nonumber,
\end{align}
where
\begin{equation*}
   \DeDephiT:=\int_0^1(1-s)\Dif^2_z\phi_T\comp(\K+s\DeK,\id)
   \left[\DeK,\DeK\right]ds
\end{equation*}
is only well-defined in $\tU$ for sufficiently small $\DeK$.
Similarly as in Section \ref{sec:invL}, we can guarantee so if
\begin{equation}\label{eq:ass_DeK}
   \frac{\norm{\DeK}_{\rho-2\de}}{R-\norm{\K-\K_0}_\rho}<1.
\end{equation}
Also, we will need $I_{2n}-\OmO\Esym$ to be invertible in
$\TT^{d+\ell}_{\rho-2\de}$. This condition is granted if 
\begin{equation}\label{eq:ass_Esym}
   \norm{\Esym}_{\rho-2\de}\leq\nu<1
\end{equation}
holds---recall the analogous assumption \eqref{eq:ass_hEsym} in
Section \ref{sec:invhP}. For the remaining of this section, we
incorporate \eqref{eq:ass_DeK} and \eqref{eq:ass_Esym} to our
assumptions.
\begin{lemma}\label{lemma:EKnew}
The new error in the torus invariance equation after correcting
$K$ by solving
\eqref{eq:lineq2} is given by 
\begin{equation}\label{eq:Enew2}
   \EKnew=\P\comp\Romal\Big(I_{2n}-\OmO \Esym\comp\Romal\Big)^{-1}\ElinK
   + \DeDephiT,
\end{equation}
where 
\begin{equation*}
   \ElinK:=\Ered\xiK +
   \OmO \Esym\comp\Romal\xiK\comp\Romal -\begin{pmatrix} 0_{n\times1} \\
   \aver{\etacK}\\ 0_{1\times1} \end{pmatrix}. 
\end{equation*}
Additionally, for any $\de\in(0,\rho/2)$
\begin{equation}\label{eq:estEKnew}
   \norm{\EKnew}_{\rho-2\de}\leq \frac{\CEKKK}{\ga^4\de^{4\tau}}\norm{\EK}_\rho^2 +
    \frac{\CEKKW}{\ga^2\de^{2\tau}}\norm{\EK}_\rho\norm{\EW}_\rho.
\end{equation}

\end{lemma}

\begin{proof}
Let us begin by inspecting the sum
\[
\EK + \DzphiK\DeK - \DeK\comp\Romal.
\]
If we left-multiply it by $-\OmO\P^\ttop\comp\Romal\OmbarK$ and by
its inverse, we obtain 
\begin{align*}
&\EK + \DzphiK\DeK - \DeK\comp\Romal\\
&=\left(-\OmO\P^\ttop\comp\Romal\OmbarK\right)^{-1}
\Big(-\etaK + \tLa\xiK + \Ered\xiK
   -\big(I_{2n}-\OmO \Esym\comp\Romal\big) \xiK\comp\Romal \Big)\\
&=\left(-\OmO\P^\ttop\comp\Romal\OmbarK\right)^{-1}
\left(\Ered\xiK + \OmO \Esym\comp\Romal\xiK\comp\Romal -\begin{pmatrix} 0_{n\times1} \\
   \aver{\etacK}\\ 0_{1\times1} \end{pmatrix}
   \right)\\
&=\left(-\OmO\P^\ttop\comp\Romal\OmbarK\right)^{-1}
\ElinK,
\end{align*}
where we used definitions \eqref{eq:defEsym}, \eqref{eq:defEred},
and that $\xiK$ satisfies \eqref{eq:corrtr1}-\eqref{eq:corrtr4}.
Recall that by assumption $\norm{\Esym}_{\rho-2\de}\leq\nu<1$ so
$I_{2n}-\OmO\Esym\comp\Romal$ is invertible.
Then, from the previous expression, that 
\begin{equation}\label{eq:inv-property}
   \Big(-\OmO  P^\ttop\comp\Romal\OmbarK \Big)^{-1} =
   P\comp\Romal\Big(I_{2n}-\OmO \Esym\comp\Romal\Big)^{-1},
\end{equation}
and \eqref{eq:Enew1}, we obtain \eqref{eq:Enew2}. 
For estimate \eqref{eq:estEKnew}, we first 
obtain estimates for each $\etaK^i$ as 
\begin{align}\label{eq:estetaK1}
   \norm{\etaaK}_\rho & \leq
   \CNT\cteOm\norm{\EK}_\rho\leq\CetaoneK\norm{\EK}_\rho,\\\label{eq:estetaK2}
   \norm{\etabK}_\rho & \leq
   \CNT\cteOm\norm{\EK}_\rho\leq\CetatwoK\norm{\EK}_\rho,\\\label{eq:estetaK3}
   \norm{\etacK}_\rho & \leq \Max{\sigmaDteKT,\
   \CcXT}\cteOm\norm{\EK}_\rho\leq\CetathreeK\norm{\EK}_\rho,\\\label{eq:estetaK4}
   \norm{\etadK}_\rho & \leq
   \sigmaWT\cteOm\norm{\EK}_\rho\leq\CetafourK\norm{\EK}_\rho,
\end{align}
and for each $\xiK^i$ we obtain the following estimates from
\eqref{eq:xiK2}-\eqref{eq:xiK1} 
\begin{align}\label{eq:estaverxiK3}
   &\abs{\aver{\xicK}} \leq\sigmainvaverS\left(\CetaoneK\ga\de^\tau
   + \ChS\CR\CetathreeK \right)
   \frac{\norm{\EK}_\rho}{\ga\de^\tau}\leq
   \frac{\CaverxithreeK}{\ga\de^{\tau}}\norm{\EK}_\rho,\\\label{eq:estxiK1}
   &\norm{\xiaK}_{\rho-2\de}\leq\CR\left(\CetaoneK\ga\de^\tau +
         \ChS\left(\CR\CetathreeK +
      \CaverxithreeK\right)\right)\frac{\norm{\EK}_\rho}
      {\ga^2\de^{2\tau}}\leq\frac{\CxioneK}{\ga^2\de^{2\tau}}
      \norm{\EK}_\rho, \\\label{eq:estxiK2}
   &\norm{\xibK}_{\rho}\leq
   \frac{1}{1-\sigmala}\CetatwoK\norm{\EK}_\rho\leq
      \CxitwoK\norm{\EK}_\rho,\\\label{eq:estxiK3}
   &   \norm{\xicK}_{\rho-\de}\leq\left(\CR\CetathreeK +
      \CaverxithreeK\right)\frac{\norm{\EK}_\rho}{\ga\de^{\tau}}\leq
      \frac{\CxithreeK}{\ga\de^{\tau}}\norm{\EK}_\rho,
      \\\label{eq:estxiK4}
   & \norm{\xidK}_{\rho}\leq
   \frac{\sigmala}{1-\sigmala}\CetafourK\norm{\EK}_\rho  \leq
         \CxifourK\norm{\EK}_\rho.
\end{align}
Therefore,
\begin{equation}\label{eq:estxiK}
   \norm{\xiK}_{\rho-2\de}\leq\Max{\CxioneK,\
      \CxitwoK\ga^2\de^{2\tau},\
      \CxithreeK\ga\de^\tau,\
   \CxifourK\ga^2\de^{2\tau}}\frac{\norm{\EK}_\rho}{\ga^2\de^{2\tau}}
   \leq\frac{\CxiK}{\ga^2\de^{2\tau}}\norm{\EK}_\rho,
\end{equation}
and, for $\DeK = P\xiK$, we have the estimate
\begin{align}\label{eq:estDeK}
   \norm{\DeK}_{\rho-2\de}&\leq\Big(\CL\Max{\CxioneK,\
      \CxitwoK\ga^2\de^{2\tau}} + \CN\Max{\CxithreeK\ga\de^\tau,\
   \CxifourK\ga^2\de^{2\tau}}\Big)
      \frac{\norm{\EK}_{\rho}}{\ga^2\de^{2\tau}}\\\nonumber
      &\leq \frac{\CDeK}{\ga^2\de^{2\tau}}\norm{\EK}_{\rho}.
\end{align}
In order to obtain estimates for $\abs{\aver{\etacK}}$, 
let 
\begin{equation*}
  \etacK =
     \begin{pmatrix}
        (\Dif_\te \K\comp\Romal)^\ttop \OmbarK \EK
     \\
     \cX^\ttop\comp\Romal \OmbarK \EK
    \end{pmatrix} =:
    \begin{pmatrix}
    \etacaK    \\
    \etacbK
    \end{pmatrix}.
\end{equation*}
We now make use of the following lemma from \cite{FMHM24},
\begin{lemma}
The averages of $\etaK^{31}$ and $\etaK^{32}$ are 
\begin{align*}
\langle\etaK^{31}\rangle &=
\langle(\Dif_\te \EK)^\ttop \Delta^1 a+ (\Dif_\te
\K\comp\Romal)^\ttop
\De^2 a\rangle,\\
\langle\etaK^{32}\rangle &=\langle \Delta^2 H\rangle-
\langle (\Delta^1 a)^\ttop
\Dif_\vp \EK\hatal\rangle - \langle(\De^2 a)^\ttop
\Dif_\vp \K\comp\Romal \hatal\rangle 
 ,
\end{align*}
where the Taylor remainders $\Delta^i$ of order $i$ in $E^K$ are given by 
\begin{align*}
\Delta^1 a:=& 
a\comp\phi_T\comp(\K,\id) -a\comp\K\comp\Romal \\
=&\int_{0}^1 \Dif a\comp\big(\K\comp\Romal
+s \EK\big)\EK\ ds,\\
\Delta^2 a:=& 
a\comp\phi_T\comp(\K,\id) - a\comp\K\comp\Romal -
\Dif a\comp\K\comp\Romal \EK   \\
=&\int_{0}^1 (1-s)\Dif^2 a\comp\big(\K\comp\Romal+s \EK\big)
[\EK,\EK]\ ds, \\
\Delta^2 H:=&
H\comp\big(\phi_T\comp(\K,\id),\Ral\big)
-H\comp(\K\comp\Romal,\Ral\big) 
- \Dif_z H\comp(\K\comp\Romal,\Ral)\EK   \\
=&\int_0^1(1-s)\Dif^2_z H\comp\big(\K\comp\Romal +
s\EK,\Ral\big)[\EK,\EK]\ ds.
\end{align*}
\end{lemma}
Since---by assumption---the integral Taylor remainders are well
defined, we can obtain the estimates 
\begin{align}\label{eq:estavereta31}
   \abs{\aver{\etacaK}} &\leq\left(2n\cteDa +
   \sigmaDteKT\cteDDa\frac{\de}{2}\right)\frac{\norm{\EK}_\rho^2}{\de}
   \leq\frac{\CaveretathreeoneK}{\de}\norm{\EK}_\rho^2, \\
   \label{eq:estavereta32}
   \abs{\aver{\etacbK}}&\leq\left(2n\ell\cteDaT\abs{\hatal}
   + (\cteDDzH + 2n\ell\cteDDa\sigmaDvpK\abs{\hatal})
   \frac{\de}{2}\right)
   \frac{\norm{\EK}_\rho^2}{\de}
   \leq\frac{\CaveretathreetwoK}{\de}\norm{\EK}_\rho^2,
\end{align}
and
\begin{equation}\label{eq:estavereta3}
   \abs{\aver{\etacK}}\leq\Max{\CaveretathreeoneK,\
   \CaveretathreetwoK}\frac{\norm{\EK}_\rho^2}{\de}
   \leq\frac{\CaveretathreeK}{\de}\norm{\EK}_\rho^2.
\end{equation}
To control $\Esym$, we will use Lemma \ref{lemma:ap_sym} which
applies if $\norm{\hEsym}_{\rho-2\de}\leq\hnu<1$; which recall is
one of our assumptions. Therefore, using 
Lemmas \ref{lemma:ap_sym} and \ref{lemma:ap_red}, we obtain
\begin{align}\label{eq:estElinK}
\norm{\ElinK}_{\rho-2\de}&\leq\Big(\left(\CEredK +
   \CEsymK\right)\CxiK +
\CaveretathreeK\ga^3\de^{3\tau}\Big)
\frac{\norm{\EK}_\rho^2}{\ga^3\de^{3\tau+1}}\\\nonumber
&\phantom{\leq}+\left(\CEredW+\CEsymW\right)\CxiK
\frac{\norm{\EK}_\rho\norm{\EW}_\rho}{\ga^2\de^{2\tau}}\\\nonumber
&\leq\frac{\CElinKK}{\ga^3\de^{3\tau+1}}\norm{\EK}_\rho^2
+\frac{\CElinKKW}{\ga^2\de^{2\tau}}\norm{\EK}_\rho\norm{\EW}_\rho.
\end{align}
Observe that using Neumann series,
\begin{equation}\label{eq:neumann}
   \norm{\left(I_{2n}-\OmO \Esym\right)^{-1}}_{\rho-2\de}
   \leq\frac{1}{1-\nu},
\end{equation}
Also, note that $\DeDephiT$ is well-defined by assumption
\eqref{eq:ass_DeK}. Hence, 
\begin{equation*}
   \norm{\DeDephiT}_{\rho-2\de}\leq \frac{1}{2}\cteDDzphiT\norm{\De
   K}_{\rho-2\de}^2\leq\frac{1}{2}\cteDDzphiT(\CDeK)^2\frac{\norm{\EK}^2_\rho}{\ga^4\de^{4\tau}}.
\end{equation*}
We now have all the estimates to obtain 
\begin{align}\label{eq:estEKnew2}
\norm{\EKnew}_{\rho-2\de}&\leq \frac{\CP}{1-\nu}
\left(\frac{\CElinKK}{\ga^3\de^{3\tau+1}}\norm{\EK}_\rho^2 +
\frac{\CElinKKW}{\ga^2\de^{2\tau}}\norm{\EK}_\rho\norm{\EW}_\rho\right) +
\frac{\cteDDzphiT(\CDeK)^2}{2\ga^4\de^{4\tau}}\norm{\EK}_\rho^2\\\nonumber
&\leq\left(\frac{\CP}{1-\nu}\CElinKK\ga\de^{\tau-1}+\frac{1}{2}\cteDDzphiT\CDeK^2\right)\frac{\norm{\EK}_\rho^2}{\ga^4\de^{4\tau}}
+ \frac{\CP}{1-\nu}\CElinKKW\frac{\norm{\EK}_\rho
\norm{\EW}_\rho}{\ga^2\de^{2\tau}}\\\nonumber
&\leq\frac{\CEKKK}{\ga^4\de^{4\tau}}\norm{\EK}_\rho^2
+\frac{\CEKKW}{\ga^2\de^{2\tau}}\norm{\EK}_\rho\norm{\EW}_\rho.
\end{align}
\end{proof}

\subsection{The linearized equation for the bundle}
Let us now inspect the new error in the bundle invariance equation.
First recall definition \eqref{eq:tildeEW}. The new
error in the bundle invariance equation is the following
\begin{align}\label{eq:EWnew1}
   \EWnew:=& \DzphiT\comp(\K+\DeK,\id)(\W+\DeW)
   - \big(\W\comp\Romal+\DeW\comp\Romal\big)(\la+\Dela)\\\nonumber
   =&\tildeEW +
 \DzphiK\DeW-\DeW\comp\Romal\la -
 \W\comp\Romal\Dela\\\nonumber
    \phantom{=}&
    +\DeDzphiT[\DeW] - \DeW\comp\Romal\Dela
    ,
\end{align}
where, for any $\zeta:\TT^{d+\ell}\to\RR^{2n}$,
\begin{equation*}
   \DeDzphiT[\zeta]=
   \int_0^1\DDzphiT\comp(\K+ s\DeK,\id)
   [\DeK,\zeta]ds.
\end{equation*}

\begin{lemma}\label{lemma:EWnew}
The new error in the bundle invariance
equation after correcting $W$ by solving \eqref{eq:coho-bundle}
is given by
\begin{equation}\label{eq:EWnew2}
   \EWnew=
   \P\comp\Romal\big(I_{2n}-\OmO\Esym\comp\Romal\big)^{-1}\ElinW+
   \DeDzphiT[\DeW]- \DeW\comp\Romal\Dela,
\end{equation}
where 
\begin{equation*}
   \ElinW:=
\Ered\xiW +
\OmO\Esym\comp\Romal\big(\xiW\comp\Romal\la + e_n\Dela\big)
\end{equation*}
Additionally, for any $\de\in(0,\rho/3)$ 
\begin{equation}\label{eq:estEWnew}
   \norm{\EWnew}_{\rho-3\de}\leq\frac{\CEWKK}{\ga^5\de^{5\tau}}\norm{\EK}_\rho^2
   + \frac{\CEWWW}{\ga\de^\tau}\norm{\EW}_\rho^2+
    \frac{\CEWKW}{\ga^3\de^{3\tau}}\norm{\EK}_\rho\norm{\EW}_\rho.
\end{equation}

\end{lemma}

\begin{proof}
Let us first inspect the sum
\begin{align*}
\tildeEW + \DzphiK\DeW-\DeW\comp\Romal\la -
\W\comp\Romal\Dela.
\end{align*}
If left-multiply it by $-\OmO\PT\comp\Romal\OmbarK$ and by its inverse,
we obtain
\begin{align*}
\phantom{=}&\left(-\OmO\PT\comp\Romal\OmbarK\right)^{-1}
\Big(-\etaW+\tLa\xiW + \Ered\xiW \\
&-\left(I_{2n}-\OmO\Esym\comp\Romal\right)\left(\xiW\comp\Romal\la
+e_n\Dela\right)\\
=&\left(-\OmO\PT\comp\Romal\OmbarK\right)^{-1}\bigg(\Ered\xiW +
\OmO\Esym\comp\Romal\big(\xiW\comp\Romal\la + e_n\Dela\big)\bigg)\\
=&\left(-\OmO\PT\comp\Romal\OmbarK\right)^{-1}\ElinW,
\end{align*}
where we used definitions \eqref{eq:defEsym}, \eqref{eq:defEred},
and that $\xiW$ satisfies \eqref{eq:coho-bundle}.
Result \eqref{eq:EWnew2} then follows from using this last
expression in \eqref{eq:EWnew1}, the assumption
$\norm{\Esym}_{\rho-2\de}i\leq\nu<1$, and \eqref{eq:inv-property}.
In order to obtain estimate \eqref{eq:estEWnew}, note that 
$\DeDzphiT$ is well-defined by assumption \eqref{eq:ass_DeK} 
and that 
\begin{equation}\label{eq:esttildeEW}
   \norm{\tildeEW}_{\rho-2\de}\leq
   \cteDDzphiT\CDeK\sigmaW\frac{\norm{\EK}_\rho}{\ga^2\de^{2\tau}} +
   \norm{\EW}_\rho \leq
   \frac{\CtildeEK}{\ga^2\de^{2\tau}}\norm{\EK}_\rho +
   \CtildeEW\norm{\EW}_\rho.
\end{equation}
Then, we proceed to obtain estimates for each $\etaW^i$ as
\begin{align}\label{eq:estetaW1}
   \norm{\etaaW}_{\rho-2\de} & \leq
   \CNT\cteOm\left(\frac{\CtildeEK}{\ga^2\de^{2\tau}} \norm{\EK}_\rho +
   \norm{\EW}_\rho\right)\\\nonumber
   &\leq\frac{\CetaoneWK}{\ga^2\de^{2\tau}}\norm{\EK}_\rho
   + \CetaoneWW \norm{\EW}_\rho, \\\label{eq:estetaW2}
   \norm{\etabW}_{\rho-2\de} & \leq
   \CNT\cteOm\left(\frac{\CtildeEK}{\ga^2\de^{2\tau}} \norm{\EK}_\rho +
    \norm{\EW}_\rho\right)\\\nonumber
     &\leq\frac{\CetatwoWK}{\ga^2\de^{2\tau}}\norm{\EK}_\rho
    + \CetatwoWW \norm{\EW}_\rho, \\\label{eq:estetaW3}
   \norm{\etacW}_{\rho-2\de} & \leq \Max{\sigmaDteKT,\ \CcXT}\cteOm
   \left(\frac{\CtildeEK}{\ga^2\de^{2\tau}}\norm{\EK}_\rho +
     \norm{\EW}_\rho\right)\\\nonumber
  &\leq\frac{\CetathreeWK}{\ga^2\de^{2\tau}}\norm{\EK}_\rho
  + \CetathreeWW \norm{\EW}_\rho, \\\label{eq:estetaW4}
   \norm{\etadW}_{\rho-2\de} & \leq
   \sigmaWT\cteOm\left(\frac{\CtildeEK}{\ga^2\de^{2\tau}} \norm{\EK}_\rho +
   \norm{\EW}_\rho\right)\\\nonumber
  &\leq\frac{\CetafourWK}{\ga^2\de^{2\tau}}\norm{\EK}_\rho
   + \CetafourWW \norm{\EW}_\rho, 
\end{align}
and for each $\xiW^i$ we obtain from
\eqref{eq:xiW1}-\eqref{eq:xiW4} the following
\begin{align}\label{eq:estxiW1}
   \norm{\xiaW}_{\rho-2\de}&\leq\frac{1}{1-\sigmala}\left(\left(\CetaoneWK +
      \frac{\ChS \CetathreeWK}{1 -\sigmala } \right)
      \frac{\norm{\EK}_\rho}{\ga^2\de^{2\tau}}+
      \left(\CetaoneWW+\frac{\ChS\CetathreeWW}{1-\sigmala}\right)
\norm{\EW}_\rho\right) \\\nonumber
   &\leq\frac{\CxioneWK}{\ga^2\de^{2\tau}}\norm{\EK}_\rho
   + \CxioneWW\norm{\EW}_\rho,
   \\\label{eq:estxiW2}
   \norm{\xibW}_{\rho-3\de}&\leq\frac{\CR\sigmainvla}{\ga\de^{\tau}}
   \left(\frac{\CetatwoWK}{\ga^2\de^{2\tau}}\norm{\EK}_\rho
   + \CetatwoWW \norm{\EW}_\rho
   \right)\\\nonumber
   &\leq\frac{\CxitwoWK}{\ga^3\de^{3\tau}}\norm{\EK}_\rho +
   \frac{\CxitwoWW}{\ga\de^\tau}\norm{\EW}_\rho,\\\label{eq:estxiW3}
   \norm{\xicW}_{\rho-2\de}&\leq\frac{1}{1-\sigmala}
   \left(\frac{\CetathreeWK}{\ga^2\de^{2\tau}}\norm{\EK}_\rho +
   \CetathreeWW \norm{\EW}_\rho\right)\\\nonumber
   &\leq\frac{\CxithreeWK}{\ga^2\de^{2\tau}}\norm{\EK}_\rho +
   \CxithreeWW\norm{\EW}_\rho,
\\\label{eq:estxiW4}
   \norm{\xidW}_{\rho-2\de}&\leq\frac{\sigmala}{1-(\sigmala)^2}
   \left(\frac{\CetafourWK}{\ga^2\de^{2\tau}}\norm{\EK}_\rho
+ \CetafourWW \norm{\EW}_\rho \right)\\\nonumber
&\leq \frac{\CxifourWK}{\ga^2\de^{2\tau}}\norm{\EK}_\rho +
\CxifourWW\norm{\EW}_\rho.
\end{align}
Then, 
\begin{align}\label{eq:estxiW}
   \norm{\xiW}_{\rho-3\de}\leq\max\Big\{&\frac{\CxioneWK}{\ga^2\de^{2\tau}}\norm{\EK}_\rho
   + \CxioneWW\norm{\EW}_\rho,\
     \frac{\CxitwoWK}{\ga^3\de^{3\tau}}\norm{\EK}_\rho +
\frac{\CxitwoWW}{\ga\de^\tau}\norm{\EW}_\rho,\\\nonumber
&\frac{\CxithreeWK}{\ga^2\de^{2\tau}}\norm{\EK}_\rho +
\CxithreeWW\norm{\EW}_\rho,\
\frac{\CxifourWK}{\ga^2\de^{2\tau}}\norm{\EK}_\rho +
\CxifourWW\norm{\EW}_\rho\Big\}\\\nonumber
&\negphantom{$\leq\max\Big\{$}\leq\max\Big\{\CxioneWK\ga\de^\tau,\
      \CxitwoWK,\ \CxithreeWK\ga\de^\tau,\ \CxifourWK\ga\de^\tau
\Big\}\frac{\norm{\EK}_\rho}{\ga^3\de^{3\tau}}\\\nonumber
&\negphantom{$\max\Big\{$} +\max\Big\{\CxioneWW\ga\de^\tau,\ 
   \CxitwoWW,\ \CxithreeWW\ga\de^\tau,\ \CxifourWW\ga\de^\tau \Big\}
\frac{\norm{\EW}_\rho}{\ga\de^\tau}\\\nonumber
&\negphantom{$\leq\max\Big\{$}\leq\frac{\CxiWK}{\ga^3\de^{3\tau}}\norm{\EK}_\rho +
   \frac{\CxiWW}{\ga\de^\tau}\norm{\EW}_\rho.
\end{align}
We can then control the correction of the bundle parameterization and
the rate of contraction as
\begin{align}\label{eq:estDeW}
   \norm{\DeW}_{\rho-3\de}&\leq\CL\Max{
      \frac{\CxioneWK}{\ga^2\de^{2\tau}}\norm{\EK}_\rho
   + \CxioneWW\norm{\EW}_\rho,\
\frac{\CxitwoWK}{\ga^3\de^{3\tau}}\norm{\EK}_\rho +
\frac{\CxitwoWW}{\ga\de^\tau}\norm{\EW}_\rho}\\\nonumber
&\phantom{\leq}+\CN\Max{\frac{\CxithreeWK}{\ga^2\de^{2\tau}}\norm{\EK}_\rho
+ \CxithreeWW\norm{\EW}_\rho,\
   \frac{\CxifourWK}{\ga^2\de^{2\tau}}\norm{\EK}_\rho +
   \CxifourWW\norm{\EW}_\rho}\\\nonumber
&\leq \left(\CL\Max{\CxioneWK\ga\de^\tau,\ \CxitwoWK} + \CN\Max{\CxithreeWK,\
\CxifourWK}\ga\de^\tau\right)
\frac{\norm{\EK}_\rho}{\ga^3\de^{3\tau}}\nonumber  \\
&\phantom{\leq}+\left(\CL\Max{\CxioneWW\ga\de^\tau,\ \CxitwoWW} 
   + \CN \Max{\CxithreeWW,\ \CxifourWW}\ga\de^\tau\right)\frac{\norm{\EW}_\rho}{\ga\de^\tau}\nonumber\\
&\leq\frac{\CDeWK}{\ga^3\de^{3\tau}}\norm{\EK}_\rho +
   \frac{\CDeWW}{\ga\de^\tau}\norm{\EW}_\rho,\nonumber\\
   \abs{\Dela}&\leq \frac{\CetatwoWK}{\ga^2\de^{2\tau}}\norm{\EK}_\rho +
   \CetatwoWW\norm{\EW}_\rho\leq
   \frac{\CDelaK}{\ga^2\de^{2\tau}}
   \norm{\EK}_\rho +\CDelaW\norm{\EW}_\rho.\label{eq:estDela} 
\end{align}
In order to control $\ElinW$, we will use Lemmas \ref{lemma:ap_sym}
and \ref{lemma:ap_red}. Recall that Lemma \ref{lemma:ap_sym}
applies from the assumption $\norm{\hEsym}\leq\hnu<1$.
Therefore, we can proceed as follows
\begin{align}\label{eq:estElinW}
   \norm{\ElinW}_{\rho-3\de}&\leq
   \norm{\Ered}_{\rho-2\de}\norm{\xiW}_{\rho-3\de} +
   \norm{\Esym}_{\rho-2\de}\norm{\xiW}_{\rho-3\de}\abs{\la} +
   \norm{\Esym}_{\rho-2\de} \abs{\Dela}\\\nonumber
   &\leq\left(\frac{\CEredK}{\ga\de^{\tau+1}}\norm{\EK}_\rho
   +\CEredW\norm{\EW}_\rho\right)\left(\frac{\CxiWK}{\ga^3\de^{3\tau}}
   \norm{\EK}_{\rho} + \frac{\CxiWW}{\ga\de^\tau}
    \norm{\EW}_\rho\right)\\\nonumber
   &\phantom{\leq} + \left(\frac{\CEsymK}{\ga\de^{\tau+1}}
   \norm{\EK}_{\rho} + \CEsymW \norm{\EW}_\rho\right)
   \left(\frac{\CxiWK}{\ga^3\de^{3\tau}}
   \norm{\EK}_{\rho} + \frac{\CxiWW}{\ga\de^\tau}
    \norm{\EW}_\rho\right)\sigmala \\\nonumber
   &\phantom{\leq} + \left(\frac{\CEsymK}{\ga\de^{\tau+1}}
   \norm{\EK}_{\rho} + \CEsymW \norm{\EW}_\rho\right)
   \left(\frac{\CDelaK}{\ga^2\de^{2\tau}}
   \norm{\EK}_{\rho} + \CDelaW \norm{\EW}_\rho\right)\\\nonumber
   &\leq\left(\CEredK\CxiWK + \CEsymK\CxiWK\sigmala + 
\CEsymK\CDelaK\ga\de^\tau\right)\frac{\norm{\EK}_\rho^2}{\ga^4\de^{4\tau+1}}\\\nonumber
   &\phantom{\leq} + \left(\CEredW\CxiWW +\CEsymW\CxiWW\sigmala +
      \CEsymW\CDelaW\ga\de^\tau\right)
      \frac{\norm{\EW}_\rho^2}{\ga\de^\tau} \\\nonumber
   &\phantom{\leq} + \Big(\CEredW\CxiWK + \CEsymW\CxiWK 
   + \left( \CEredK\CxiWW +  \CEsymK\CxiWW \right)\ga\de^{\tau-1}
   \\\nonumber
   &\phantom{\leq} +  \CEsymW\CDelaK\ga\de^\tau +
   \CEsymK\CDelaW\ga^2\de^{2\tau-1}
    \Big)
    \frac{\norm{\EK}_\rho\norm{\EW}_\rho}{\ga^3\de^{3\tau}}\\\nonumber
   &\leq\frac{\CElinWKK}{\ga^4\de^{4\tau+1}}\norm{\EK}_\rho^2  +
    \frac{\CElinWKW}{\ga^3\de^{3\tau}}\norm{\EK}_\rho\norm{\EW}_\rho+
    \frac{\CElinWWW}{\ga\de^\tau}\norm{\EW}_\rho^2.
\end{align}
Lastly, we can control the remaining terms in \eqref{eq:EWnew2} as 
\begin{align}\label{eq:est1}
   \norm{\De\DzphiT[\DeW]}_{\rho-3\de}&\leq\cteDDzphiT\CDeK\left(
   \frac{\CDeWK}{\ga^5\de^{5\tau}}\norm{\EK}_\rho^2 +
\frac{\CDeWW}{\ga^3\de^{3\tau}}\norm{\EK}_\rho\norm{\EW}_\rho\right)\\\nonumber
&\leq\frac{\CDeDzphiTKK}{\ga^5\de^{5\tau}}\norm{\EK}_\rho^2 + 
     \frac{\CDeDzphiTKW}{\ga^3\de^{3\tau}}\norm{\EK}_\rho\norm{\EW}_\rho,\\ 
  \label{eq:est2}
\norm{\DeW\Dela}_{\rho-3\de}&\leq\frac{\CDeWK\CDelaK}
   {\ga^5\de^{5\tau}}\norm{\EK}_\rho^2 +
   \frac{\CDeWW\CDelaW}{\ga\de^\tau}\norm{\EW}_\rho^2\\\nonumber
 &\phantom{\leq} + \frac{\CDeWK\CDelaW +
 \CDeWW\CDelaK}{\ga^3\de^{3\tau}}\norm{\EK}_\rho\norm{\EW}_\rho\\\nonumber
 & \leq\frac{\CDeWDelaKK}{\ga^5\de^{5\tau}}\norm{\EK}_\rho^2  +
\frac{\CDeWDelaKW}{\ga^3\de^{3\tau}}\norm{\EK}_\rho\norm{\EW}_\rho+
 \frac{\CDeWDelaWW}{\ga\de^\tau}\norm{\EW}_\rho^2.
\end{align}
We can now use the assumption
$\norm{\Esym}_{\rho-2\de}\leq\nu<1$
and \eqref{eq:neumann} to finally obtain the estimate
\eqref{eq:estEWnew} as
\begin{align}\label{eq:estEWnew2}
   \norm{\EWnew}_{\rho-3\de}&\leq
   \frac{\CP}{1-\nu}\left(\frac{\CElinWKK}{\ga^4\de^{4\tau+1}}\norm{\EK}_\rho^2
      + \frac{\CElinWKW}{\ga^3\de^{3\tau}}\norm{\EK}_\rho\norm{\EW}_\rho
        + \frac{\CElinWWW}{\ga\de^\tau}\norm{\EW}_\rho^2
\right)\\\nonumber
    &\phantom{\leq} +
    \frac{\CDeDzphiTKK}{\ga^5\de^{5\tau}}\norm{\EK}_\rho^2 +
    \frac{\CDeDzphiTKW}{\ga^3\de^{3\tau}}\norm{\EK}_\rho\norm{\EW}_\rho\\\nonumber
    &\phantom{\leq} + \frac{\CDeWDelaKK}{\ga^5\de^{5\tau}}\norm{\EK}_\rho^2 +
 \frac{\CDeWDelaKW}{\ga^3\de^{3\tau}}\norm{\EK}_\rho\norm{\EW}_\rho
  +\frac{\CDeWDelaWW}{\ga\de^\tau}\norm{\EW}_\rho^2
\\\nonumber
    &\leq\left(\frac{\CP}{1-\nu}\CElinWKK\ga\de^{\tau-1} 
    + \CDeDzphiTKK + \CDeWDelaKK\right)
    \frac{\norm{\EK}_\rho^2}{\ga^5\de^{5\tau}}\\\nonumber
    &\phantom{\leq} +  \left(\frac{\CP}{1-\nu}\CElinWKW + \CDeDzphiTKW +
    \CDeWDelaKW\right)\frac{\norm{\EK}_\rho\norm{\EW}_\rho}{\ga^3\de^{3\tau}}\\\nonumber
    &\phantom{\leq}+\left(\frac{\CP}{1-\nu}\CElinWWW + \CDeWDelaWW\right)
    \frac{\norm{\EW}_\rho^2}{\ga\de^\tau}\\\nonumber
    &\leq\frac{\CEWKK}{\ga^5\de^{5\tau}}\norm{\EK}_\rho^2 +
    \frac{\CEWKW}{\ga^3\de^{3\tau}}\norm{\EK}_\rho\norm{\EW}_\rho
    + \frac{\CEWWW}{\ga\de^\tau}\norm{\EW}_\rho^2.
\end{align}
\end{proof}

\section{Proof of the KAM theorem}\label{sec:proof KAM}
As mentioned previously, the iterative scheme described in
Section \ref{sec:quasi-Newton} allows the construction of a
sequence of objects. The key idea for the proof of Theorem \ref{thm:KAM} 
consists in showing that the sequence for $\K, \W,$ and $\la$
converge---under certain conditions---to solutions of
\eqref{eq:invK} and \eqref{eq:invW}. First, we obtain an
iterative lemma and derive conditions such that, if satisfied, the
iterative lemma can be applied again to the objects obtained
after one step of the Quasi-Newton method. Convergence then
follows from ensuring that the conditions necessary for the
iterative lemma are satisfied at every step of the iterative
process.

\subsection{The iterative lemma}\label{sec:iterative lemma}
After one step of the iterative scheme, we obtain the new 
parameterizations
$\barK\leftarrow\K+\DeK$ and $\barW\leftarrow \W + \DeW$ and 
the new rate of contraction $\barla\leftarrow\la+\Dela$. From
$\barK,\barW$, and $\barla$, 
we obtain the new objects $\barNO, \barB,$ and $\bar S$. 
The following lemma allows the control of the new objects---and
consequently the control of the iterative procedure.

\begin{lemma}[The iterative lemma]\label{lemma:iter}
Assume the hypothesis of Theorem \ref{thm:KAM}. 
For every $\de\in(0,\rho/3)$, there exists a constant $\CDe$ such that
if
\begin{equation}\label{eq:Hypoiter}
   \frac{\CDe}{\ga\de^\tau}\E<1
\end{equation}
holds, where
\begin{equation*}
   \E=\Max{\frac{\norm{\EK}_\rho}{\ga^2\de^{2\tau}},\
   \norm{\EW}_\rho},
\end{equation*}
\begin{align}\label{eq:deffracC}
   \CDe := \max\Bigg\{&
      \CEsym\ga\de^{\tau},\
      \frac{\Max{\ga^2\de^{2\tau},\CDeK}}{R-\norm{\K-\KO}_\rho}\ga\de^\tau,\
      \frac{d\CDeK}{\sigmaDteK -
      \norm{\DteK}_\rho}\ga\de^{\tau-1},\\\nonumber
      &\frac{\ell\CDeK}{\sigmaDvpK -
   \norm{\DvpK}_\rho}\ga\de^{\tau-1},\
   \frac{2n\CDeK}{\sigmaDteKT -
   \norm{\DteKT}_\rho}\ga\de^{\tau-1},\\\nonumber 
      &  \frac{2n\CDeK}{\sigmaDvpKT - \norm{\DvpKT}_\rho}\ga\de^{\tau-1},\ 
\frac{\CDeW}{\sigmaW-\norm{\W}_\rho},\
\frac{2n\CDeW}{\sigmaWT-\norm{\WT}_\rho},\\ \nonumber
   &\frac{\CDeB}{\sigmaB-\norm{\B}_\rho},\
   \frac{\CDeNO}{\sigmaNO-\norm{\NO}_\rho},\
   \frac{\CDeNOT}{\sigmaNOT-\norm{\NOT}_\rho},\ 
\frac{\CDeinvaverS}{\sigmainvaverS -
\big|\aver{\S}^{\!-1}\big|},\\\nonumber
   &\frac{\CDela}{\sigmala- \abs{\la}}\ga\de^\tau,\
\frac{\CDeinvla}{\sigmainvla- \abs{\invla}}\ga\de^\tau
\Bigg\},
\end{align}
then we have new parameterizations
$\barK\in\left(\cA\big(\TT^{d+\ell}_{\rho-2\de}\big)\right)^{2n}$,
$\barW\in\left(\cA\big(\TT^{d+\ell}_{\rho-3\de}\big)\right)^{2n}$
and their associated objects satisfy
\begin{align}\label{eq:estDetebarK}
&\norm{\DtebarK}_{\rho-3\de}<\sigmaDteK,\\
&\norm{\DvpbarK}_{\rho-3\de}<\sigmaDvpK,\\
&\norm{\DtebarKT}_{\rho-3\de}<\sigmaDteKT,\\\label{eq:estDevpbarKT}
&\norm{\DvpbarKT}_{\rho-3\de}<\sigmaDvpKT,\\\label{eq:estbarW}
&\norm{\barW}_{\rho-3\de}<\sigmaW,\\\label{eq:estbarWT}
&\norm{\barWT}_{\rho-3\de}<\sigmaWT,\\\label{eq:estNO}
&\norm{\barNO}_{\rho-3\de}<\sigmaNO,\\\label{eq:estNOT}
&\norm{\barNOT}_{\rho-3\de}<\sigmaNOT,\\\label{eq:estB}
&\norm{\barB}_{\rho-3\de}<\sigmaB,\\\label{eq:estinvaverSone}
&\Abs{\aver{\bar \S}^{-1}}<\sigmainvaverS.
\end{align}
Furthermore, the hyperbolicity is controlled by 
\begin{equation}\label{eq:barla}
   \abs{\barla} < \sigmala,\quad \abs{\barinvla}<\sigmainvla,
\end{equation}
and the new objects are close to the original ones in the
following sense 
\begin{align}\label{eq:barK-K}
&\norm{\barK-\K}_{\rho-2\de}\leq\CDeK\E,\\\label{eq:DtebarK-DteK}
&\norm{\DtebarK-\DteK}_{\rho-3\de}\leq\frac{d\CDeK}{\de}\E,\\
&\norm{\DvpbarK-\DvpK}_{\rho-3\de}\leq\frac{\ell\CDeK}{\de}\E,\\
&\norm{(\DtebarK)^\ttop-(\DteK)^\ttop}_{\rho-3\de}\leq\frac{2n\CDeK}{\de}\E,\\\label{eq:DvpbarKT-DvpKT}
&\norm{(\DvpbarK)^\ttop-(\DvpK)^\ttop}_{\rho-3\de}\leq\frac{2n\CDeK}{\de}\E,\\\label{eq:barW-W}
&\norm{\barW - \W}_{\rho-3\de}\leq
\frac{\CDeW}{\ga\de^{\tau}}\E,\\\label{eq:barWT-WT}
&\norm{\barWT -\WT}_{\rho-3\de}\leq
\frac{2n\CDeW}{\ga\de^{\tau}}\E,\\
\label{eq:barB-B}
&\norm{\barB - \B}_{\rho-3\de} \leq
\frac{\CDeB}{\ga\de^{\tau}}\E,\\\label{eq:barNO-NO}
&\norm{\barNO-\NO}_{\rho-3\de}\leq
\frac{\CDeNO}{\ga\de^{\tau}}\E,\\ \label{eq:barNOT-NOT}
&\norm{\barNOT -\NOT}_{\rho-3\de}\leq
\frac{\CDeNOT}{\ga\de^{\tau}}\E,\\\label{eq:barSone-Sone}
&\Abs{\aver{\bar\S}^{\!-1} - \aver{\S}^{\!-1}}\leq
\frac{\CDeinvaverS}{\ga\de^{\tau}}\E,\\\label{eq:barla-la}
&\abs{\barla-\la}\leq\CDela\E,\\\label{eq:barinvla-invla}
&\abs{\barinvla-\invla}\leq\CDeinvla\E.
\end{align}
Lastly, the new errors in the invariance equations
\begin{align*}
   \EKnew&=\phi_T\comp(\barK,\id)-\barK\comp\Romal,\\
   \EWnew&=\DzphiT\comp(\barK,\id)\barW
   -\barW\comp\Romal\barla,
\end{align*}
satisfy
\begin{equation}\label{eq:estEnew}
   \norm{\EKnew}_{\rho-2\de}\leq\CEK\E^2,\quad
   \norm{\EWnew}_{\rho-3\de}\leq\frac{\CEW}{\ga\de^\tau}\E^2.
\end{equation}

\end{lemma}

\begin{proof}
Recall that in order to use Lemmas \ref{lemma:EKnew} and
\ref{lemma:EWnew}, we need hypothesis \eqref{eq:ass_hEsym},
\eqref{eq:ass_Esym}, \eqref{eq:ass_EK}, and \eqref{eq:ass_DeK}
to hold. Let us take for the values of $\hnu$ and $\nu$ the
following
\begin{align}\label{eq:def_hnu}
   \hnu&=\left(\ChEsymK\ga\de^{\tau-1}+\ChEsymW\right)\E=\ChEsym\E,\\\label{eq:def_nu}
\nu &=\left(\CEsymK\ga\de^{\tau-1} + \CEsymW\right)\E=\CEsym\E,
\end{align}
which clearly satisfy $\norm{\hEsym}_{\rho-2\de}\leq\hnu$ and
$\norm{\Esym}_{\rho-2\de}\leq\nu$. It is not hard to show that
$\hnu\leq\nu$; see Table \ref{tab:sym-red}. Then, 
from hypothesis \eqref{eq:Hypoiter} for the first term in
\eqref{eq:deffracC} we conclude that \eqref{eq:ass_hEsym} and
\eqref{eq:ass_Esym} are satisfied.
For \eqref{eq:ass_EK} and \eqref{eq:ass_DeK}, observe that by hypothesis
\eqref{eq:Hypoiter} for the second term in
\eqref{eq:deffracC} 
\begin{equation*}
   \frac{\Max{\ga^2\de^{2\tau},\CDeK}\ga\de^\tau}
   {R-\norm{\K-\KO}_\rho}\frac{\E}{\ga\de^\tau}<1.
\end{equation*}
Then, from the definition of $\E$ and estimate \eqref{eq:estDeK}, we conclude that
\eqref{eq:ass_EK} and \eqref{eq:ass_DeK} are also satisfied.
Furthermore, estimates \eqref{eq:barK-K}, \eqref{eq:barW-W}, and
\eqref{eq:barWT-WT} are
immediate from \eqref{eq:estDeK} and \eqref{eq:estDeW}, respectively.
For estimates  \eqref{eq:estDetebarK}-\eqref{eq:estDevpbarKT}, we 
proceed as follows
\begin{align*}
   \norm{\DtebarK}_{\rho-3\de}&\leq\norm{\DteK}_\rho
   +\norm{\Dif_\te\DeK}_{\rho-3\de}\leq\norm{\DteK}_\rho  +
   \frac{d\CDeK}{\ga^2\de^{2\tau+1}}\norm{\EK}_\rho,\\
   \norm{\DvpbarK}_{\rho-3\de}&\leq\norm{\DvpK}_\rho + 
   \norm{\Dif_\vp\DeK}_{\rho-3\de}\leq\norm{\DvpK}_\rho  + 
   \frac{\ell\CDeK}{\ga^2\de^{2\tau+1}}\norm{\EK}_\rho,\\
   \norm{(\DtebarK)^\ttop}_{\rho-3\de}&\leq
   \norm{(\DteK)^\ttop}_\rho + 
   \norm{(\Dif_\te\DeK)^\ttop}_{\rho-3\de}\leq
   \norm{(\DteK)^\ttop}_\rho +
   \frac{2n\CDeK}{\ga^2\de^{2\tau+1}}\norm{\EK}_\rho,\\
   \norm{(\DvpbarK)^\ttop}_{\rho-3\de}&\leq 
   \norm{(\DvpK)^\ttop}_\rho +
   \norm{(\Dif_\vp\DeK)^\ttop}_{\rho-3\de}\leq
   \norm{(\DvpK)^\ttop}_\rho +
   \frac{2n\CDeK}{\ga^2\de^{2\tau+1}}\norm{\EK}_\rho,
\end{align*}
where we used Cauchy estimates on $\DeK$. Note that these Cauchy
estimates result directly in
\eqref{eq:DtebarK-DteK}-\eqref{eq:DvpbarKT-DvpKT}. Then, 
\begin{align*}
   \norm{\DtebarK}_{\rho-3\de}&\leq\norm{\DteK}_\rho
    + \frac{d\CDeK}{\de}\E,\\
   \norm{\DvpbarK}_{\rho-3\de}&\leq\norm{\DvpK}_\rho + 
   \frac{\ell\CDeK}{\de}\E,\\
   \norm{(\DtebarK)^\ttop}_{\rho-3\de}&\leq
   \norm{(\DteK)^\ttop}_\rho + 
   \frac{2n\CDeK}{\de}\E,\\
   \norm{(\DvpbarK)^\ttop}_{\rho-3\de}&\leq 
   \norm{(\DvpK)^\ttop}_\rho +
   \frac{2n\CDeK}{\de}\E,
\end{align*}
and by hypothesis \eqref{eq:Hypoiter}---in particular for the third
to sixth terms in \eqref{eq:deffracC}---we conclude
\begin{align*}
   \norm{\DtebarK}_{\rho-3\de} & <\sigmaDteK,  
  &\norm{(\DtebarK)^\ttop}_{\rho-3\de} &< \sigmaDteKT, \\ 
   \norm{\DvpbarK}_{\rho-3\de} &< \sigmaDvpK, 
  & \norm{(\DvpbarK)^\ttop}_{\rho-3\de} &< \sigmaDvpKT.
\end{align*}
Similarly, for the corrected parameterization of the bundle,
from estimate \eqref{eq:estDeW} we have
\begin{align}\label{eq:CDeW}
   \norm{\barW}_{\rho-3\de}&\leq \norm{\W}_\rho + 
   \frac{\CDeWK}{\ga^3\de^{3\tau}}\norm{\EK}_\rho +
   \frac{\CDeWW}{\ga\de^\tau}\norm{\EW}_\rho \\\nonumber
   &\leq \norm{\W}_\rho +
   \frac{\CDeW}{\ga\de^\tau}\E,\\\nonumber
   \norm{\barW^\ttop}_{\rho-3\de}&\leq \norm{\W^\ttop}_\rho +
    \frac{2n\CDeWK}{\ga^3\de^{3\tau}}\norm{\EK}_\rho +
    \frac{2n\CDeWW}{\ga\de^\tau}\norm{\EW}_\rho\\\nonumber
   &\leq \norm{\W^\ttop}_\rho + 
   \frac{2n\CDeW}{\ga\de^\tau}\E.
\end{align}
By hypothesis \eqref{eq:Hypoiter} for the seventh and
eighth terms in \eqref{eq:deffracC}, we obtain estimates
\eqref{eq:estbarW} and \eqref{eq:estbarWT}.
In order to control $\barL$, note that
\begin{equation*}
   \DeL:=\barL-\L =
   \begin{pmatrix}
           \Dif_\te\DeK\ \Big|\ \widetilde{\De X} -
      \Dif_\vp\DeK\hatal\ \Big|\ \DeW
   \end{pmatrix},
\end{equation*}
where
\begin{equation*}
        \widetilde{\De X}:=\int_0^1 \Dif_z
   X\comp\left(\K + s\DeK,\id\right)\DeK ds,
\end{equation*}
which is well-defined in $\TT^d_{\rho-2\de}\times\TT^\ell_{\rho-2\de}$
by hypothesis \eqref{eq:Hypoiter} for the second term in
\eqref{eq:deffracC}. Then, using \eqref{eq:estDeK},
\eqref{eq:estDeW}, and Cauchy estimates, we obtain
\begin{align}\label{eq:CbarL-L}
   \norm{\barL-\L}_{\rho-3\de}&\leq\left(
           \left(d+\ell\abs{\hatal}  + \cteDzX\de\right)\CDeK\ga\de^{\tau-1} + \CDeWK
\right)\frac{\norm{\EK}_\rho}{\ga^3\de^{3\tau}} +
\frac{\CDeWW}{\ga\de^\tau}\norm{\EW}_\rho\\\nonumber
&\leq\frac{\CDeLK}{\ga^3\de^{3\tau}} \norm{\EK}_\rho +
\frac{\CDeLW}{\ga\de^\tau}\norm{\EW}_\rho\\\nonumber
&\leq\frac{\CDeL}{\ga\de^\tau}\E,\\\label{eq:CbarLT-LT}
\norm{\barL^\ttop - \L^\ttop}_{\rho-3\de}&\leq \max
   \Bigg\{\frac{2n}{\de}\norm{\DeK}_{\rho-2\de},\
   2n\left(\cteDzXT+\frac{\abs{\hatal^\ttop}}{\de}\right)\norm{\DeK}_{\rho-2\de},\
      2n\norm{\DeW}_{\rho-3\de} \Bigg\}\\\nonumber
     &\leq \frac{2n}{\de}\max\Big\{1, \cteDzXT\de +
     \abs{\hatal^\ttop}\Big\} \norm{\DeK}_{\rho-2\de}
     + 2n\norm{\DeW}_{\rho-3\de}\\\nonumber
     &\leq  2n\left(\max\Big\{1, \cteDzXT\de + \abs{\hatal^\ttop}
     \Big\}\CDeK\ga\de^{\tau-1} 
     + \CDeWK\right)\frac{\norm{\EK}_\rho}{\ga^3\de^{3\tau}} +
     \frac{2n\CDeWW}{\ga\de^\tau}\norm{\EW}_\rho\\\nonumber
     &\leq\frac{\CDeLTK}{\ga^3\de^{3\tau}}\norm{\EK}_\rho +
        \frac{\CDeLTW}{\ga\de^\tau}\norm{\EW}_\rho\\\nonumber
     &\leq\frac{\CDeLT}{\ga\de^\tau}\E.
\end{align}
Let us now control $\DeB:=\barB-\B$. First, 
for the restricted metric, we have that 
\begin{equation*}
   \GbarL-\GL =\L^\ttop \DeG \barL
   +\L^\ttop\G\comp\K\DeL +
\DeL^\ttop\G\comp\barK\barL,
\end{equation*}
where
\begin{equation*}
\DeG:=\GbarK-\GK=\int_0^1\Dif \G\left(\K +
s\DeK\right)\DeK ds,
\end{equation*}
which is well-defined in
$\TT^d_{\rho-2\de}\times\TT^\ell_{\rho-2\de}$ by hypothesis
\eqref{eq:Hypoiter} for the second term in \eqref{eq:deffracC}.  
Then, 
\begin{align}\label{eq:CbarGL-GL}
   \norm{\GbarL - \GL}_{\rho-3\de}&\leq
   \CLT\cteDG\CL\norm{\DeK}_{\rho-2\de} +
   \CLT\cteG\norm{\DeL}_{\rho-3\de}  +
   \cteG\CL\norm{\DeL^\ttop}_{\rho-3\de}\\\nonumber
   &\leq \left(\CLT\cteDG\CL\CDeK\ga\de^\tau + \CLT\cteG\CDeLK +
   \cteG\CL\CDeLTK\right)\frac{\norm{\EK}_\rho}{\ga^3\de^{3\tau}}\\\nonumber
   &\phantom{=} + \left(\CLT\cteG\CDeLW + \cteG\CL\CDeLTW\right)
   \frac{\norm{\EW}_\rho}{\ga\de^\tau}\\\nonumber
   &\leq\frac{\CDeGLK}{\ga^3\de^{3\tau}}\norm{\EK}_\rho +
   \frac{\CDeGLW}{\ga\de^\tau}\norm{\EW}_\rho\\\nonumber
   & \leq\frac{\CDeGL}{\ga\de^\tau}\E.
\end{align}
We will now use the following result that is presented for matrices
but extends to matrix valued maps \cite{FH24}.
\begin{lemma}\label{lemma:invM}
   Let $M\in\CC^{n\times n}$ be an invertible matrix satisfying
   $|M^{-1}|<\si$ and assume that $\bar M\in\CC^{n\times n}$ satisfies
   \begin{equation*}
      \frac{\si^2|\bar M-M|}{\si-|M^{-1}|}<1.
   \end{equation*} 
   Then, $\bar M$ is invertible and 
   \[
      |\bar M^{-1}|<\si, \quad |\bar M^{-1} -M^{-1}
      |\leq\si^2|\bar M - M|.
   \]
\end{lemma}
Note that by hypothesis \eqref{eq:Hypoiter} for the ninth term in
\eqref{eq:deffracC}, 
\[
   \frac{(\sigmaB)^2\norm{\GbarL - \GL}_{\rho-3\de}}{\sigmaB -
   \norm{B}_\rho}<1.
\]
Therefore, applying Lemma \ref{lemma:invM} to $\GbarL(\angles)$ and
$\GL(\angles)$, we obtain \eqref{eq:barB-B} and \eqref{eq:estB}
as
\begin{align}\label{eq:CbarB-B}
   \norm{\barB - \B}_{\rho-3\de}&\leq(\sigmaB)^2
   \norm{\GbarL - \GL}_{\rho-3\de}\\\nonumber
  &\leq\frac{(\sigmaB)^2\CDeGLK}{\ga^3\de^{3\tau}}\norm{\EK}_\rho
   +
   \frac{(\sigmaB)^2\CDeGLW}{\ga\de^\tau}\norm{\EW}_\rho\\\nonumber
  &\leq\frac{\CDeBK}{\ga^3\de^{3\tau}}\norm{\EK}_\rho +
  \frac{\CDeBW}{\ga\de^\tau} \norm{\EW}_\rho\\\nonumber
  &\leq\frac{\CDeB}{\ga\de^\tau}\E,\\\nonumber
   \norm{\barB}_{\rho-3\de}&<\sigmaB.
\end{align}
Let us proceed and control $\DeNO:=\barNO-\NO$ and
$\DeNOT:=\barNOT-\NOT$. Note that
\begin{equation*}
   \barNO-\NO = \J\comp \K \L \DeB + \J\comp\K\DeL\barB
+ \DeJ \barL\barB,
\end{equation*}
where
\begin{equation*}
   \DeJ:=\J\comp\barK
   - \J\comp\K
                 = \int_0^1\DJ\big(\K +
                 s\DeK\big)\DeK ds
\end{equation*}
is well-defined in $\TT^d_{\rho-2\de}\times\TT^\ell_{\rho-2\de}$ 
by hypothesis \eqref{eq:Hypoiter} for the second term in 
\eqref{eq:deffracC}. Then, we obtain
\eqref{eq:barNO-NO} and \eqref{eq:barNOT-NOT} as
\begin{align}\label{eq:CbarNO-NO}
   \norm{\barNO-\NO}_{\rho-3\de}&\leq\cteJ\CL\norm{\DeB}_{\rho-3\de}
+ \cteJ\norm{\DeL}_{\rho-3\de}\sigmaB
+\cteDJ\norm{\DeK}_{\rho-2\de}\CL\sigmaB\\\nonumber
&\leq \left(\cteJ\left(\CL\CDeBK + \CDeLK\sigmaB\right)
 + \cteDJ\CDeK\CL\sigmaB\ga\de^\tau\right)
 \frac{\norm{\EK}_\rho}{\ga^3\de^{3\tau}}\\\nonumber
&\phantom{\leq} + \cteJ\left(\CL\CDeBW + 
\CDeLW\sigmaB\right)\frac{\norm{\EW}_\rho}{\ga\de^\tau}\\\nonumber
&\leq\frac{\CDeNOK}{\ga^3\de^{3\tau}}\norm{\EK}_\rho + 
\frac{\CDeNOW}{\ga\de^\tau}\norm{\EW}_\rho\\\nonumber
   &\leq\frac{\CDeNO}{\ga\de^\tau}\E,\\\label{eq:CbarNOT-NOT}
\norm{\barNOT -\NOT}_{\rho-3\de}&\leq
   \CLT\cteJT\norm{\DeB}_{\rho-3\de} +
   \sigmaB\norm{\DeLT}_{\rho-3\de}\cteJT + 
   \sigmaB\CLT\cteDJT2n\norm{\DeK}_{\rho-2\de}\\\nonumber
&\leq\left(\cteJT\left(\CLT\CDeBK + \sigmaB\CDeLTK\right) +
\sigmaB\CLT\cteDJT2n\CDeK\ga\de^\tau\right)
\frac{\norm{\EK}_\rho}{\ga^3\de^{3\tau}} \\\nonumber
&\phantom{\leq} + \cteJT\left(\CLT\CDeBW + \sigmaB\CDeLTW
\right)\frac{\norm{\EW}_\rho}{\ga\de^\tau}\\\nonumber
&\leq\frac{\CDeNOTK}{\ga^3\de^{3\tau}}\norm{\EK}_\rho
+\frac{\CDeNOTW}{\ga\de^\tau}\norm{\EW}_\rho\\\nonumber
&\leq\frac{\CDeNOT}{\ga\de^\tau}\E.
\end{align}
Also, by hypothesis \eqref{eq:Hypoiter} for the tenth and
eleventh terms in \eqref{eq:deffracC}, 
\begin{align*}
   \norm{\barNO}_{\rho-3\de}&\leq\norm{\NO}_{\rho}+
   \frac{\CDeNO}{\ga\de^{\tau}}\E <\sigmaNO,\\
   \norm{\barNOT}_{\rho-3\de}&\leq\norm{\NOT}_\rho + 
      \frac{\CDeNOT}{\ga\de^{\tau}}\E<\sigmaNOT,
\end{align*}
so we have estimates \eqref{eq:estNO} and \eqref{eq:estNOT}.
Note that, according to \eqref{eq:esthatS}, the new torsion
satisfies
\[
   \norm{\barhS}_{\rho-3\de}<\ChS.
\]
Similarly as for $\DeNO$, we obtain
\begin{align}\label{eq:CbarhS-hS}
   \norm{\barhS-\hS}_{\rho-3\de}&\leq\sigmaNOT\cteOm\cteDzphiT
   \norm{\DeNO}_{\rho-3\de}+\sigmaNOT\sigmaNO\left(\cteDOm\cteDzphiT +
   \cteOm\cteDDzphiT\right)\norm{\DeK}_{\rho-2\de}\\\nonumber
    &\phantom{\leq} +
    \norm{\DeNOT}_{\rho-3\de}\cteOm\cteDzphiT\sigmaNO\\\nonumber
    & \leq \Big(\sigmaNOT\cteOm\cteDzphiT\CDeNOK +
    \sigmaNOT\sigmaNO\left(\cteDOm\cteDzphiT +
 \cteOm\cteDDzphiT\right)\CDeK\ga\de^\tau \\\nonumber
    &\phantom{\leq} +\CDeNOTK\cteOm\cteDzphiT\sigmaNO\Big)
    \frac{\norm{\EK}_\rho}{\ga^3\de^{3\tau}}
    + \Big(\sigmaNOT\cteOm\cteDzphiT\CDeNOW +
    \CDeNOTW\cteOm\cteDzphiT\sigmaNO\Big)\frac{\norm{\EW}_\rho}
    {\ga\de^\tau}\\\nonumber
    &\leq\frac{\CDehSK}{\ga^3\de^{3\tau}}\norm{\EK}_{\rho} +
    \frac{\CDehSW}{\ga\de^\tau}\norm{\EW}_\rho\\\nonumber
    &\leq\frac{\CDehS}{\ga\de^\tau}\E.
\end{align}
In order to obtain estimates \eqref{eq:estinvaverSone} and
\eqref{eq:barSone-Sone}, we need to control the reduced torsion $\S$. 
We will use Lemma \ref{lemma:invM} which
holds if
\begin{equation}\label{eq:condSone}
   \frac{\left(\sigmainvaverS\right)^2\Abs{\aver{\bar\S} -
   \aver{\S}}}{\sigmainvaverS - \Abs{\aver{\S}^{\!-1}}}<1.
\end{equation}
Observe that 
\begin{equation*}
   \Abs{\aver{\bar\S} - \aver{\S}}\leq 
   \Abs{\aver{\barhS} - \aver{\hS}}\leq
   \norm{\barhS - \hS}_{\rho-3\de},
\end{equation*}
and by hypothesis \eqref{eq:Hypoiter} for the twelfth term in
\eqref{eq:deffracC}, 
\begin{equation*}
   \frac{\left(\sigmainvaverS\right)^2\norm{\barhS - \hS}_{\rho-3\de}
   }{\sigmainvaverS - \Abs{\aver{\S}^{\!-1}}}<1.
\end{equation*}
Therefore, \eqref{eq:condSone} holds, $\aver{\bar\S}$ is invertible, and
\begin{align}\label{eq:CbarinvS-invS}
\Abs{\aver{\bar\S}^{\!-1} - \aver{\S}^{\!-1}}
&\leq\left(\sigmainvaverS\right)^2\norm{\barhS - \hS}_{\rho-3\de}\\\nonumber
&\leq\frac{\left(\sigmainvaverS\right)^2\CDehSK}{\ga^3\de^{3\tau}}
\norm{\EK}_{\rho} + \frac{\left(\sigmainvaverS\right)^2\CDehSW}{\ga\de^\tau}
\norm{\EW}_\rho\\\nonumber
&\leq\frac{\CDeinvaverSK}{\ga^3\de^{3\tau}}\norm{\EK}_{\rho}
+ \frac{\CDeinvaverSW}{\ga\de^\tau}\norm{\EW}_\rho\\\nonumber
&\leq\frac{\CDeinvaverS}{\ga\de^\tau}\E,\\\nonumber
\Abs{\aver{\bar\S}^{\!-1}}&<\sigmainvaverS.
\end{align}

For the new rate of contraction, observe that 
\begin{equation*}
   \abs{\barla}\leq \abs{\la} + \abs{\Dela}, \quad
   \abs{\barinvla}\leq \abs{\invla} + \abs{\Deinvla}.
\end{equation*}
Therefore, the hyperbolicity controls \eqref{eq:barla} hold if
\begin{equation*}
   \frac{\abs{\Dela}}{\sigmala - \abs{\la}}<1, \quad
   \frac{\abs{\Deinvla}}{\sigmainvla - \abs{\invla}}<1.
\end{equation*}
Note that from bound \eqref{eq:estDela} and hypothesis
\eqref{eq:Hypoiter} for the thirteenth term in \eqref{eq:deffracC}
\begin{equation}\label{eq:CDela}
        \frac{\abs{\Dela}}{\sigmala-\abs{\la}}\leq
        \frac{1}{\sigmala-\abs{\la}}
        \left(\frac{\CDelaK}{\ga^2\de^{2\tau}}\norm{\EK}_\rho
        + \CDelaW\norm{\EW}_\rho\right)\leq
        \frac{\CDela}{\sigmala-\abs{\la}}\E<1.
\end{equation}
Also, hypothesis \eqref{eq:Hypoiter} for the last term in
\eqref{eq:deffracC} allow us to use Lemma \ref{lemma:invM} and
obtain
\begin{align}\label{eq:Deinvla}
        \abs{\Deinvla}&\leq(\sigmainvla)^2\CDela\E,\\\nonumber
        \abs{\barinvla}&\leq\sigmainvla.
\end{align}

For estimates \eqref{eq:estEnew} we will use Lemmas
\ref{lemma:EKnew} and \ref{lemma:EWnew}. From
\eqref{eq:estEKnew} we have that
\begin{equation}\label{eq:CEKnew}
   \norm{\EKnew}_{\rho-2\de}\leq\left(\CEKKK +
   \CEKKW\right)\E^2\leq\CEK\E^2.
\end{equation}
Also, from \eqref{eq:estEWnew}, we obtain
\begin{align}\label{eq:CEWnew}
   \norm{\EWnew}_{\rho-3\de}&\leq\left(\CEWKK + \CEWWW + \CEWKW
   \right)\frac{\E^2}{\ga\de^\tau}\\\nonumber
 &\leq\frac{\CEW}{\ga\de^\tau}\E^2.
\end{align}

\end{proof}

\subsection{Convergence of the iterative procedure and proof of the
KAM theorem}\label{sec:convergence}
We show that under the assumptions of Theorem \ref{thm:KAM},
condition \eqref{eq:Hypoiter} in Lemma \ref{lemma:iter} holds
at every step of the iterative scheme. Therefore, Lemma \ref{lemma:iter}
can be applied infinitely many times. In the
limit, the sequences of objects converge and define invariant
tori and its invariant bundles.

Consider the objects at the iterative step $j$ denoted by $\K_j,
~\W_j,~\la_j,~\NO_j,~\B_j,$ and $\S_j$. Consider also the
errors $E_{\K_j}$ and $E_{\W_j}$ in the invariance 
equations \eqref{eq:defEK} \eqref{eq:defEW}, respectively, 
the analyticity domain $\rho_j$, and analyticity bite $\de_j$. For
the step $j=0$, consider the objects declared in theorem
\ref{thm:KAM}. That is, take $\K_0=\K,~\W_0=\W,~\la_0=\la,
~\NO_0=\NO,~\B_0=\B,~\S_0=\S,~\rho_0=\rho,~\de_0=\de$, and
$\E_0=\E$. 
Note that the domain of analyticity of the objects decreases at
every step according to
\[
        \rho_j=\rho_{j-1} - 3\de_{j-1},
\]
where $\de_j$ is the domain bite chosen for step $j$. Let us then
choose the following geometric sequence of bites 
\[
        \de_j=\frac{\de}{a^j},
\]
for some $a>1$---which results in the limiting width
\[
        \rho_\infty = \rho-3\de\frac{a}{a-1}
\]
and 
\[
   a=\frac{\rho-\rho_\infty}{\rho_1-\rho_\infty}.
\]
Let us now define 
\begin{align}\label{eq:CEj}
   \E_j&:=\Max{\frac{\norm{\EKj}_\rhoj}{\ga^2\de_j^{2\tau}},  
   \norm{\EWj}_\rhoj},& 
   \CEj&:=\Max{E_{\bar\K_j} a^{2\tau},
   E_{\bar\W_j} \ga\de_j},\\\nonumber
   \vka_j&:=\frac{\E_j}{\ga^2\de^{2\tau}_j} , &
   \ka_j&:=a^{2\tau}\CEj\vka_j
\end{align}
and assume \eqref{eq:Hypoiter} in Lemma \ref{lemma:iter} holds
for the objects up to step $j-1$. Then, 
\[
   \vka_j\leq
   a^{2\tau}\CEjmo\vka_{j-1}^2\leq\ka_{j-1}\vka_{j-1}
\]
and
\[
   \ka_{j}\leq a^{2\tau}\CEj\ka_{j-1}\vka_{j-1}\leq\ka_{j-1}^2,
\]
where the last inequality holds as long as $\CE_j\le\CEjmo$.
In order to apply Lemma
\ref{lemma:iter} again we need \eqref{eq:Hypoiter} to hold for
the objects in the step $j$. 
Notice that $\CDe$ in \eqref{eq:deffracC}, as well as the
constants therein, depend on $\de_j$---polynomially by
construction---and on $\hnu_j$ and $\nu_j$.
See Appendix \ref{ap:constants} and definitions 
\eqref{eq:def_hnu} and \eqref{eq:def_nu}. Let us now
assume $\ka:=\ka_0<1$ and $\nu_0<1$,
which are included in assumption
\ref{eq:HKAM}. Recall that $\hnu\leq\nu$. 
Additionally, the assumption $\ka<1$ 
implies decreasing sequences for
$\hnu_j,~\nu_j,~\CEj,~\ka_j,~\vka_j,$ and $\E_j$. In particular, 
\begin{gather}\nonumber
   \ka_j\leq\ka_{j-1}^2\leq\ka^{2^j}\\\label{eq:Ekapa}
   \E_j\leq\ga^2\de_j^{2\tau}\vka_j\leq
   \ga^2\de_j^{2\tau}\ka_0^{2^{j-1}}\vka_0
   \leq\left(\frac{\ka}{a^{2\tau}}\right)^j\E.
\end{gather}
In order to prove convergence, it will be necessary to control the
constants within $\CDe$ uniformly with respect to $j$. Hence,
the constants within $\CDe$ are taken for the largest values
of $\de,~\hnu,$ and $\nu$, i.e., $\de=\de_0,~\hnu=\hnu_0$, and
$\nu=\nu_0$.  Nonetheless, the explicit $\de$ in $\CDe$ can
be taken to be the corresponding ones for the step $j$, i.e., they
can be taken to be $\de=\de_j$.

We will now use estimate \eqref{eq:Ekapa} to check condition
\eqref{eq:Hypoiter} for step $j$.
For the first term in \eqref{eq:deffracC}, condition
\eqref{eq:Hypoiter} reads 
\begin{equation*}
\CEsym\E_j<1
\end{equation*}
and, using \eqref{eq:Ekapa}, we have
\begin{equation*}
   \CEsym\E_j\leq\CEsym\left(\frac{\ka}{a^{2\tau}}\right)^j\E
   \leq\CEsym\E<1
\end{equation*}
where the last inequality holds from the assumption $\nu_0<1$.
For the second term in \eqref{eq:deffracC}, 
let us first compute
\begin{equation}\label{eq:CDeK}
   \norm{\K_j-\K}_\rhoj\leq\sum_{i=0}^{j-1}\norm{\K_{i+1}-\K_{i}}_{\rho_{i+1}}
   \leq\CDeK\E\sum_{i=0}^{j-1}\left(\frac{\ka}{a^{2\tau}}\right)^i
   \leq\frac{a^{2\tau}}{a^{2\tau}-\ka}\CDeK \E:=\fracCDeK\E,
\end{equation}
where we used \eqref{eq:barK-K} and \eqref{eq:Ekapa}. Notice that 
the previous bound holds for all $j$---hence, we can obtain
estimate \eqref{eq:KAMK}. Similarly, using Lemma
\ref{lemma:iter}, we obtain the estimates
\eqref{eq:KAMDteK}-\eqref{eq:KAMinvla}.
Let us now consider the inequality
\[
   \frac{\ga^2\de^{2\tau}_j\E_j}{R-\norm{\K_j-\K}_{\rho_j}}<1
\]
which is satisfied if
\begin{equation*}
   \ga^2\de_j^{2\tau}\E_j + \fracCDeK\E<R
\end{equation*}
or, alternatively,
\begin{equation*}
\frac{\ga^2\de^{2\tau} +\fracCDeK}{R}\E<1,
\end{equation*}
where we used that $\de_j<\de$ and $\E_j<\E$.  
It is not hard to show that the previous condition implies
\eqref{eq:Hypoiter} for the second term in \eqref{eq:deffracC} at
every step; which is included in assumption \eqref{eq:HKAM}.
In order to express condition \eqref{eq:Hypoiter} for the third
term in \eqref{eq:deffracC}, we need to control
$\norm{\Dif_\te\K_j}_\rhoj$. Notice that we can recurrently
obtain the bound
\begin{align*}
   \norm{\Dif_\te\K_j}_\rhoj&\leq\norm{\Dif_\te\K}_{\rho}
   +\sum_{i=0}^{j-1}\frac{d\CDeK}{\de_i}\E_i\\
    &\leq \norm{\Dif_\te\K}_{\rho} +
   \frac{d\CDeK}{\de}\sum_{i=0}^{\infty}
   \left(\frac{\ka}{a^{2\tau-1}}\right)^i\E\\
    &\leq \norm{\Dif_\te\K}_{\rho} +
    \frac{d\CDeK a^{2\tau-1}}{\de(a^{2\tau-1} -
    \ka)}\E<\sigmaDteK,
\end{align*}
where the last inequality holds if 
\begin{equation*}
   \frac{d\CDeK}{\de(\sigmaDteK - \norm{\Dif_\te
   \K}_{\rho})}\frac{a^{2\tau-1}}{a^{2\tau-1}-\ka}\E<1,
\end{equation*}
which is included in assumption \eqref{eq:HKAM}. Notice that this
condition guarantees \eqref{eq:Hypoiter} for the third
term in \eqref{eq:deffracC}. 
We can proceed similarly and obtain conditions such that
\eqref{eq:Hypoiter} holds at every step for the remaining
terms in \eqref{eq:deffracC}. Such conditions are then included
in \eqref{eq:HKAM} for which we have the explicit constants
\begin{equation}\label{eq:fracC}
   \fracC:=\max\Bigg\{a^{2\tau}C_{\E} ,\
   \fracCDe\ga\de^\tau
   \Bigg\},
\end{equation}
\begin{align*}
   \fracCDe:=\max\Bigg\{&
   \CEsym\ga\de^\tau,\ 
   \frac{\fracCtU}{R}\ga\de^\tau,\
   \frac{\fracCDteK}{\sigmaDteK-\norm{\DteK}_{\rho}}\ga\de^{\tau-1},\
   \frac{\fracCDvpK}{\sigmaDvpK-\norm{\DvpK}_{\rho}}\ga\de^{\tau-1},\\ 
   &\frac{\fracCDteKT}{\sigmaDteKT-\norm{(\DteK)^\ttop}_{\rho}}\ga\de^{\tau-1},\
   \frac{\fracCDvpKT}{\sigmaDvpKT-\norm{(\DvpK)^\ttop}_{\rho}}\ga\de^{\tau-1},\
   \frac{\fracCW}{\sigmaW-\norm{\W}_{\rho}},\\
   &\frac{\fracCWT}{\sigmaWT-\norm{\W^\ttop}_{\rho}},\
   \frac{\fracCDeB}{\sigmaB-\norm{\B}_{\rho}},\
   \frac{\fracCNO}{\sigmaNO-\norm{\NO}_{\rho}},\
   \frac{\fracCNOT}{\sigmaNOT-\norm{\NOT}_{\rho}},\
   \frac{\fracCDeinvaverS}{\sigmainvaverS-\abs{\aver{\S}^{\!-1}}},\\
   &\frac{\fracCDela}{\sigmala-\abs{\la}}\ga\de^\tau,\
   \frac{\fracCDeinvla}{\sigmainvla-\abs{\invla}}\ga\de^\tau
   \Bigg\},
\end{align*}
where 
\begin{align*}
&\fracCtU:=\ga^2\de^{2\tau}+ \fracCDeK,
& &\fracCDteK:=\frac{a^{2\tau-1}}{a^{2\tau-1} - \ka}d\CDeK,
& &\fracCDvpK:=\frac{a^{2\tau-1}}{a^{2\tau-1} - \ka}\ell\CDeK,\\
&\fracCDteKT:=\frac{a^{2\tau-1}}{a^{2\tau-1} - \ka}2n\CDeK,
& &\fracCDvpKT:=\fracCDteKT,
& &\fracCW:=\frac{a^\tau}{a^\tau-\ka}\CDeW,\\
&\fracCWT:=\frac{a^\tau}{a^\tau-\ka}2n\CDeW,
& &\fracCDeB:=\frac{a^\tau}{a^\tau-\ka}\CDeB,
& &\fracCNO:=\frac{a^\tau}{a^\tau-\ka}\CDeNO,\\
& \fracCNOT:=\frac{a^\tau}{a^\tau-\ka}\CDeNOT,
& &\fracCDeinvaverS:=\frac{a^\tau}{a^\tau-\ka}\CDeinvaverS,
& &\fracCDela:=\frac{a^{2\tau}}{a^{2\tau}-\ka}\CDela\\
& & &\fracCDeinvla:=\frac{a^{2\tau}}{a^{2\tau}-\ka}\CDeinvla. &
&  
\end{align*}
Observe that the first term in \eqref{eq:fracC} corresponds to
the condition of having a decreasing sequence of errors whereas
the second term corresponds to the conditions necessary for
applying Lemma \ref{lemma:iter} infinitely many times.

\section{Acknowledgments}
This work has been supported by the Spanish grants
PID2021-125535NB-I00 and PID2020-118281GB-C31,
the Catalan grant 2021 SGR 00113, the Severo Ochoa and Mar\'{\i}a
de Maeztu Program for Centers and Units of Excellence in R\&D
(CEX2020-001084-M) of the Spanish Research Agency, the
Secretariat for Universities and Research of the Ministry of
Business and Knowledge of the Government of Catalonia, and by the
European Social Fund. 

\bibliographystyle{alpha}
\bibliography{references}

\newpage
\appendix
\newgeometry{left=15mm, right=15mm, top=25mm, bottom=25mm, bindingoffset=6mm}
\section{Compendium of constants involved in the KAM theorem}\label{ap:constants}
\color{black}
\bgroup
{ \def\arraystretch{1.5} \begin{longtable}{|l l l|}
\caption{Constants in Section \ref{sec:frames control}}
\label{tab:}\\
\hline
Object & Constant & Label  \\
\hline
\hline
$\cX$ & 
$\CcX=\cteX + \sigmaDvpK\abs{\hat\al}$ &
\eqref{eq:defCcX}\\
\hline
$\cXT$ & 
$\CcXT=\cteXT + \abs{\hat\al^\ttop}\sigmaDvpKT$ &
\eqref{eq:defCcX}\\
\hline
$\L$ & 
$\CL=\sigmaDteK + \CcX + \sigmaW$ &
\eqref{eq:defCL}\\
\hline
$\LT$ & 
$\CLT=\Max{\sigmaDteKT,\CcXT,\sigmaWT}$ &
\eqref{eq:defCL}\Bspace{1.8}\\
\hline
$\hP$ & 
$\ChP=\CL+\sigmaNO$ &
\eqref{eq:esthP}\\
\hline
$\hPT$ & 
$\ChPT=\Max{\CLT,\ \sigmaNOT}$ &
\eqref{eq:esthP}\\
\hline
$\hS$ & 
$\ChS=\sigmaNOT \cteOm\cteDzphiT\sigmaNO$ &
\eqref{eq:esthatS}\\
\hline
$\hS^\ttop$ & 
$\ChST=\sigmaNOT \cteOm\cteDzphiTT\sigmaNO$ &
\eqref{eq:esthatS}\\
\hline
$\A$ & 
$\CA=\frac{1}{1-(\sigmala)^2}\left(\ChS\sigmala
+ \ChST\right)$ &
\eqref{eq:estA}\\ [+0.5ex]
\hline
$\AT$ & 
$\CAT= \frac{1}{1-(\sigmala)^2}
\left(\ChST\sigmala +\ChS \right)$&
\eqref{eq:estAT}\\ [+0.5ex]
\hline
$\N$ & 
$\CN=\sigmaNO +\CL\CA$ &
\eqref{eq:estN}\\
\hline
$\NT$ & 
$\CNT=\sigmaNOT + \CAT\CLT$ &
\eqref{eq:estNT}\\
\hline
$\P$ & 
$\CP=\CL+\CN$ &
\eqref{eq:estP}\\
\hline
$\PT$ & 
$\CPT=\Max{\CLT,\ \CNT}$ &
\eqref{eq:estPT}\\
\hline
\end{longtable}
}
\egroup

\newpage

\bgroup
{
\def\arraystretch{1.5}
\begin{longtable}{| l l l|}
\caption{Constants in Lemmas~\ref{lemma:invL},
 \ref{lemma:ap_iso}, and \ref{lemma:ap_Lag}.
\label{tab:ap_lemas}}\\
\hline
Object & Constant & Label \\ 
\hline
\hline
$\EcX$ &
$\CEcX=\cteDzX \de + \ell\abs{\hat\al}$ &
\eqref{eq:estEcX}\\
\hline
$\EcX^\ttop$ & 
$\CEcXT=2n \left(\cteDzXT \de +\Abs{\hat\al^\ttop}\right)$ &
\eqref{eq:estEcXT} \\
\hline
$\EL$ &
$\CELK=d+\CEcX,\quad \CELW=1 $ &
\eqref{eq:estEL} \\
\hline
$\EL^\ttop$ &
$\CELTK=\Max{2n,\ \CEcXT},\quad \CELTW=2n$ &
\eqref{eq:estELT}\Bspace{1.8} \\
\hline
$\Lop\OmDK$ &
$\CLieOmDK=\sigmaDteKT\cteDOm\sigmaDteK\de+\sigmaDteKT\cteOm d + 2 n
\cteOm\cteDzphiT\sigmaDteK$ &
\eqref{eq:estLieOmK} \\
\hline
$\Lop\OmDKcX$ &
$\CLieaone=\sigmaDteKT\cteDOm\CcX\de
    + \sigmaDteKT\cteOm\CEcX + 2n\cteOm\cteDzphiT\CcX$ &
\eqref{eq:estLieOmDKcX} \\
\hline
$\Lop^{1\la}\OmDKW$ &
$\CLieatwoK=\sigmaDteKT\cteDOm\sigmaW\sigmala
      \de +2n\cteOm\cteDzphiT\sigmaW, \quad 
\CLieatwoW=\sigmaDteKT \cteOm$ &
\eqref{eq:estLieOmDKW} \\[+0.3ex]
\hline
$\Lop\OmcXDK$ &
$\CLieathree=\CcXT\cteDOm\sigmaDteK\de + \CcXT\cteOm d + 
\CEcXT\cteOm\cteDzphiT\sigmaDteK$ &
\eqref{eq:estLieOmcXDK} \\[+0.2ex]
\hline
$\Lop^{1\la}\OmcXW$ &
$\CLieafiveK=\CcXT\cteDOm\sigmaW\sigmala\de
    + \CEcXT\cteOm\cteDzphiT\sigmaW , \quad 
\CLieafiveW = \CcXT\cteOm$ &
\eqref{eq:estLieOmcXW} \\[+0.3ex]
\hline
$\Lop^{1\la}\OmWDK$ &
$\CLieasixK=\sigmala\sigmaWT\cteDOm\sigmaDteK\de
     + \sigmala\sigmaWT\cteOm d ,\quad 
\CLieasixW=2n\cteOm\cteDzphiT\sigmaDteK$ &
\eqref{eq:estLieOmWDK} \\[+0.3ex]
\hline
$\Lop^{1\la}\OmWcX$ &
$\CLieasevenK=\sigmala\sigmaWT\cteDOm\CcX\de
     + \sigmala\sigmaWT\cteOm\CEcX ,\quad
\CLieasevenW=2n\cteOm\cteDzphiT\CcX$&
\eqref{eq:estLieOmWcX} \\[+0.3ex]
\hline
$\OmDK$ &
$\COmDK=\CR\CLieOmDK$ &
\eqref{eq:estOmDK} \\
\hline
$\OmDKcX$ &
$\Caone=\CR \CLieaone$ &
\eqref{eq:estOmDKcX} \\
\hline
$\OmDKW$ &
$\CatwoK=\frac{1}{1-\sigmala} \CLieatwoK, \quad 
\CatwoW=\frac{1}{1-\sigmala}
\CLieatwoW$ &
\eqref{eq:estOmDKW} \\ [+0.3ex]
\hline
$\OmcXDK$ &
$\Cathree=\CR \CLieathree$ &
\eqref{eq:estOmcXDK} \\
\hline
$\OmcXW$ &
$\CafiveK= \frac{1}{1-\sigmala} \CLieafiveK, \quad 
\CafiveW= \frac{1}{1-\sigmala} \CLieafiveW$ &
\eqref{eq:estOmcXW} \\ [+0.5ex]
\hline
$\OmWDK$ &
$\CasixK=\frac{1}{1-\sigmala} \CLieasixK,\quad 
\CasixW= \frac{1}{1-\sigmala} \CLieasixW$ &
\eqref{eq:estOmWDK} \\[+0.5ex]
\hline
$\OmWcX$ &
$\CasevenK= \frac{1}{1-\sigmala} \CLieasevenK,\quad 
\CasevenW=\frac{1}{1-\sigmala} \CLieasevenW$ &
\eqref{eq:estOmWcX} \\ [+0.5ex]
\hline
$\OmL$ &
$\COmLK=\max\Big\{\COmDK + \Caone + \CatwoK\ga\de^\tau,\
   \Cathree+\CafiveK\ga\de^\tau,\ \Big(\CasixK+
\CasevenK\Big)\ga\de^\tau \Big\}$ & \\[+0.5ex]
& $\COmLW=\max\Big\{\CatwoW,\ \CafiveW,\ \CasixW
+ \CasevenW \Big\} $ &
\eqref{eq:estOmL} \\ [+0.5ex]
\hline
\end{longtable}
}
\egroup

\newpage

\bgroup
{
\def\arraystretch{1.5}
\begin{longtable}{|l l l|}
   \caption{Constants in Lemmas~\ref{lemma:ap_sym_hP},
      \ref{lemma:ap_red_hP}, \ref{lemma:ap_inv_hP}, \ref{lemma:esthSsym}, \ref{lemma:estA-AT}, 
   \ref{lemma:ap_sym}, and \ref{lemma:ap_red}}\label{tab:sym-red}\\
\hline
Object & Constant & Label  \\
\hline
\hline
$\hEsym$ & 
$\ChEsymK=\max\left\{1,\ (\sigmaB)^2\right\}\COmLK,\quad 
\ChEsymW=\max\left\{1,\ (\sigmaB)^2\right\}\COmLW$ &
\eqref{eq:esthEsym}\\
\hline
$\hEred^{11}$ &
$\ChEredaaK= \sigmaNOT\cteOm\CELK,\quad 
\ChEredaaW= \sigmaNOT\cteOm\CELW$ &
\eqref{eq:esthEred11}\Bspace{2.2} \\
\hline
$\hEred^{21}$ & 
$\ChEredbaK=\COmLK + \CLT\cteOm\CELK\ga\de^\tau, \quad 
\ChEredbaW=\COmLW + \CLT\cteOm\CELW $ & 
\eqref{eq:esthEred21}\Bspace{2.2} \\
\hline
$\hEred^{22}$ &
$\ChEredbbK=\CLT\cteDOm\cteDzphiT\sigmaNO\de +
   \sigmainvla\CELTK\cteOm\cteDzphiT\sigmaNO, \quad
\CEredbbW =\sigmainvla\CELTW\cteOm\cteDzphiT\sigmaNO$ &
   \eqref{eq:esthEred22} \\ [+0.75ex]
\hline
$\hEred$ &
$\ChEredK=\max\Big\{\ChEredaaK\ga\de^\tau,\ \ChEredbaK
+\ChEredbbK\ga\de^\tau\Big\}
,\quad
\ChEredW=\max\Big\{\ChEredaaW,\ \ChEredbaW+\ChEredbbW\Big\} $ &
\eqref{eq:esthEred} \\[+1ex]
\hline
$\EinvhP$ & 
$\CEinvhPK=\frac{1}{1-\hnu}\ChP\ChPT\cteOm\CEsymK,\quad
\CEinvhPW=\frac{1}{1-\hnu}\ChP\ChPT\cteOm\CEsymW$ & 
\eqref{eq:EinvhP} \\
\hline
$\EsyminvLahS$ & 
$\CinvLahSK=(\sigmaB)^2\COmLK+\sigmaNOT\cteDzphiTT\Big(
   \ChS\left(\cteDOm\CL\de + \sigmainvla\cteOm\CELK\right)\ga\de^\tau
   +\cteDOm\cteDzphiT\sigmaNO\ga\de^{\tau+1}$  & \\
 &
$\phantom{\CinvLahSK=}+
\cteOm\cteDzphiT\sigmaNO\CEinvhPK +
   \cteOm\ChP\CEredbbK\ga\de^\tau\Big)$ & \\
 & $\CinvLahSW = (\sigmaB)^2\COmLW + \sigmaNOT\cteDzphiTT\cteOm\Big(
\ChS\sigmainvla\CELW + \cteDzphiT\sigmaNO\CEinvhPW
    + \ChP\CEredbbW\Big)$ 
 & \eqref{eq:esthSsym} \\[.5ex]
\hline
$\EsymA$ & 
$\CEsymAK=\frac{1}{1-\sigmala}\CinvLahSK,\quad
\CEsymAW=\frac{1}{1-\sigmala}\CinvLahSW$ & 
\eqref{eq:estA-AT}\\ [+.3ex]
\hline
$\Esym$ &
$\CEsymK=\left(1+\CAT\right)\left(1+\CA\right)\ChEsymK
 + \CEsymAK$ &  \\
& $\CEsymW=\left(1+\CAT\right)\left(1+\CA\right)\ChEsymW 
 + \CEsymAW$ & \eqref{eq:estEsym} \\
\hline
$\Ered$ &
$\CEredK= \left(1+\CAT\right)\left(1+\CA\right)\ChEredK,\quad  
\CEredW=\left(1+\CAT\right)\left(1+\CA\right)\ChEredW$
& \eqref{eq:estEred}\Bspace{2} \\
\hline
\end{longtable}
}
\egroup

\newpage

\bgroup
{
\def\arraystretch{1.5}
\begin{longtable}{|l l l|}
   \caption{Constants in Lemma \ref{lemma:EKnew}}
   \label{tab:} \\
\hline
Object & Constant & Label \\
\hline
\hline
$\etaaK$ &
$\CetaoneK=\CNT\cteOm$ &
\eqref{eq:estetaK1}\\
\hline 
$\etabK$ &
$\CetatwoK=\CNT\cteOm$ &
\eqref{eq:estetaK2}\\
\hline 
$\etacK$ &
$\CetathreeK=\Max{\sigmaDteKT,\ \CcXT} $ &
\eqref{eq:estetaK3}\\ [+0.5ex]
\hline
$\aver{\etacK}$ &
$\CaveretathreeoneK=2n\cteDa +
\sigmaDteKT\cteDDa\tfrac{\de}{2},\quad 
\CaveretathreetwoK=2n\ell\cteDaT\abs{\hat\al}
   + (\cteDDzH + 2n\ell\cteDDa\sigmaDvpK\abs{\hat\al})\tfrac{\de}{2}$ & 
\\
 &
$\CaveretathreeK=\Max{\CaveretathreeoneK,\ \CaveretathreetwoK}$ &
\eqref{eq:estavereta31}-\eqref{eq:estavereta3}\\
\hline
$\etadK$ &
$\CetafourK=\sigmaWT\cteOm$ &
\eqref{eq:estetaK4}\\
\hline 
$\aver{\xicK}$ &
$\CaverxithreeK=\sigmainvaverS\left(\CetaoneK\ga\de^\tau
   + \ChS\CR\CetathreeK \right)$ &
   \eqref{eq:estaverxiK3}\Bspace{1}\\
\hline 
$\xiaK$ &
$\CxioneK=\CR\left(\CetaoneK\ga\de^\tau +
         \ChS\left(\CR\CetathreeK +
      \CaverxithreeK\right)\right)$ &
\eqref{eq:estxiK1}\\[+0.5ex]
\hline 
$\xibK$ &
$\CxitwoK=\frac{1}{1-\sigmala}\CetatwoK$ &
\eqref{eq:estxiK2}\\[+0.5ex]
\hline 
$\xicK$ &
$\CxithreeK=\CR\CetathreeK +
   \CaverxithreeK$ &
   \eqref{eq:estxiK3}\Bspace{1.9}\\
\hline 
$\xidK$ &
$\CxifourK=\frac{\sigmala}{1-\sigmala}\CetafourK$ &
\eqref{eq:estxiK4}\\
\hline 
$\xiK$ &
$\CxiK=\Max{\CxioneK,\ \CxitwoK\ga^2\de^{2\tau},\
\CxithreeK\ga\de^\tau,\ \CxifourK\ga^2\de^{2\tau}}$ &
\eqref{eq:estxiK}\Bspace{1.8}\\
\hline 
$\DeK$ &
$\CDeK=
\CL\Max{\CxioneK,\ \CxitwoK\ga^2\de^{2\tau}} + \CN\Max{\CxithreeK\ga\de^\tau,\
      \CxifourK\ga^2\de^{2\tau}}$ &
      \eqref{eq:estDeK}\Bspace{1.8}\\
\hline 
$\ElinK$ &
$\CElinKK= \left(\CEredK + \CEsymK\right)\CxiK +
\CaveretathreeK\ga^3\de^{3\tau}, \quad 
\CElinKKW=\left(\CEredW+\CEsymW\right)\CxiK$ &
\eqref{eq:estElinK}\\[+0.4ex]
\hline 
$\EKnew$ &
$\CEKKK=\frac{\CP}{1-\nu}\CElinKK\ga\de^{\tau-1}
+\frac{1}{2}\cteDDzphiT(\CDeK)^2, \quad
\CEKKW=\frac{\CP}{1-\nu}\CElinKKW$ &
\eqref{eq:estEKnew2}\\[+0.4ex]
\hline 
\end{longtable}
}
\egroup

\newpage

\bgroup
{
\def\arraystretch{1.5}
\begin{longtable}{|l l l|}
\caption{Constants in Lemma  \ref{lemma:EWnew}} \label{tab:} \\
\hline
Object & Constant & Label \\
\hline
\hline
$\tildeEW$ &
$\CtildeEK=\cteDDzphiT\CDeK\sigmaW,\quad
\CtildeEW=1$ &
\eqref{eq:esttildeEW}\\[+0.4ex]
\hline 
$\etaaW$ &
$\CetaoneWK=\CNT\cteOm\CtildeEK, \quad 
\CetaoneWW=\CNT\cteOm$ &
\eqref{eq:estetaW1}\\ [+0.3ex]
\hline 
$\etabW$ &
$\CetatwoWK= \CNT\cteOm\CtildeEK,\quad 
\CetatwoWW = \CNT\cteOm $&
\eqref{eq:estetaW2}\\
\hline 
$\etacW$ &
$\CetathreeWK =\Max{\sigmaDteKT,\ \CcXT}\cteOm\CtildeEK , \quad
\CetathreeWW = \Max{\sigmaDteKT,\ \CcXT}\cteOm$ &
\eqref{eq:estetaW3}\\[+0.5ex]
\hline 
$\etadW$ &
$\CetafourWK =\sigmaWT\cteOm\CtildeEK  ,\quad
\CetafourWW = \sigmaWT\cteOm$ &
\eqref{eq:estetaW4}\\[+0.5ex]
\hline 
$\xiaW$ &
$\CxioneWK=\frac{1}{1-\sigmala}\left(\CetaoneWK +
\frac{\ChS\CetathreeWK}{1-\sigmala}\right) ,\quad 
\CxioneWW = \frac{1}{1-\sigmala}\left(\CetaoneWW+
\frac{\ChS\CetathreeWW}{1-\sigmala}\right)$ &
\eqref{eq:estxiW1}\\[+2.ex]
\hline 
$\xibW$ &
$\CxitwoWK = \CR\sigmainvla\CetatwoWK,\quad 
\CxitwoWW = \CR\sigmainvla\CetatwoWW$ &
\eqref{eq:estxiW2}\\[+0.5ex]
\hline 
$\xicW$ &
$\CxithreeWK = \frac{1}{1-\sigmala}\CetathreeWK,\quad 
\CxithreeWW = \frac{1}{1-\sigmala}\CetathreeWW$ &
\eqref{eq:estxiW3}\\[+0.5ex]
\hline 
$\xidW$ &
$\CxifourWK = \frac{\sigmala}{1-(\sigmala)^2}\CetafourWK,\quad 
\CxifourWW = \frac{\sigmala}{1-(\sigmala)^2}\CetafourWW$ &
\eqref{eq:estxiW4}\\[+0.7ex]
\hline 
$\xiW$ &
$\CxiWK =
\max\Big\{\CxioneWK\ga\de^\tau,\
\CxitwoWK,\ \CxithreeWK\ga\de^\tau,\ \CxifourWK\ga\de^\tau
\Big\}$ & \\
& $\CxiWW =  \max\Big\{\CxioneWW\ga\de^\tau,\
\CxitwoWW,\ \CxithreeWW\ga\de^\tau,\ \CxifourWW\ga\de^\tau \Big\}
$ &
\eqref{eq:estxiW}\\ [+0.5ex]
\hline 
$\DeW$ &
$\CDeWK =
\CL\Max{\CxioneWK\ga\de^\tau,\ \CxitwoWK} + \CN\Max{\CxithreeWK,\
  \CxifourWK}\ga\de^\tau$ &
\\
& $\CDeWW=\CL\Max{\CxioneWW\ga\de^\tau,\ \CxitwoWW} 
 + \CN \Max{\CxithreeWW,\ \CxifourWW}\ga\de^\tau$ &
\eqref{eq:estDeW}\Bspace{1.9}\\
\hline 
$\Dela$ &
$\CDelaK =\CetatwoWK,\quad
\CDelaW = \CetatwoWW$ &
\eqref{eq:estDela}\\[+0.5ex]
\hline 
$\ElinW$ &
$\CElinWKK=\CEredK\CxiWK + \CEsymK\CxiWK\sigmala + 
\CEsymK\CDelaK\ga\de^\tau$ &
\\
& $\CElinWKW=\CEredW\CxiWK + \CEsymW\CxiWK 
   + \left( \CEredK\CxiWW +  \CEsymK\CxiWW \right)\ga\de^{\tau-1}
   $ &
\\
 &  $\phantom{\CElinWKW=}+  \CEsymW\CDelaK\ga\de^\tau +
   \CEsymK\CDelaW\ga^2\de^{2\tau-1} $ &
\\
& $\CElinWWW=\CEredW\CxiWW +\CEsymW\CxiWW\sigmala +
      \CEsymW\CDelaW\ga\de^\tau$ &
\eqref{eq:estElinW}\\[+0.5ex]
\hline
$\DeDzphiT[\DeW]$ &
$\CDeDzphiTKK=\cteDDzphiT\CDeK\CDeWK,\quad
\CDeDzphiTKW=\cteDDzphiT\CDeK\CDeWW$ &
\eqref{eq:est1}\\
\hline 
$\DeW\Dela$ & 
$\CDeWDelaKK = \CDeWK\CDelaK,\quad 
\CDeWDelaKW = \CDeWK\CDelaW + \CDeWW\CDelaK,\quad
\CDeWDelaWW = \CDeWW\CDelaW$ &
\eqref{eq:est2}\\
\hline
$\EWnew$ & 
$\CEWKK =\frac{\CP}{1-\nu}\CElinWKK\ga\de^{\tau-1} +
     \CDeDzphiTKK + \CDeWDelaKK $ & 
\\
 & $
   \CEWKW =\frac{\CP}{1-\nu}\CElinWKW + \CDeDzphiTKW +
    \CDeWDelaKW, \quad
     \CEWWW=\frac{\CP}{1-\nu}\CElinWWW + \CDeWDelaWW$  &
 \eqref{eq:estEWnew2}\\[+0.9ex]
\hline
\end{longtable}
}
\egroup

\newpage

\bgroup
{
\def\arraystretch{1.5}
\begin{longtable}{|l l l|}
   \caption{Constants in Lemma \ref{lemma:iter}} \label{tab:} \\
\hline
Object & Constant & Label \\
\hline
\hline
$\hEsym$ & 
$\ChEsym=\ChEsymK\ga\de^{\tau-1} + \ChEsymW$ & 
\eqref{eq:def_hnu}\\ [0.5ex]
\hline
$\Esym$ & 
$\CEsym=\CEsymK\ga\de^{\tau-1} + \CEsymW$ & 
\eqref{eq:def_nu}\\[0.2ex]
\hline
$\barW-\W$ & 
$\CDeW = \CDeWK + \CDeWW$ & 
\eqref{eq:CDeW}\\
\hline
$\barL-\L$ & 
$\CDeLK =\left(d+\ell\abs{\hatal} +
\cteDzX\de\right)\CDeK\ga\de^{\tau-1} + \CDeWK
 ,\quad 
\CDeLW = \CDeWW,\quad 
\CDeL = \CDeLK + \CDeLW$ & 
\eqref{eq:CbarL-L}\\
\hline
$\barL^\ttop - \L^\ttop$ & 
$\CDeLTK =2n\left(\max\Big\{1, \cteDzXT\de +
\abs{\hatal^\ttop} \Big\}\CDeK\ga\de^{\tau-1}
     + \CDeWK\right), \quad 
\CDeLTW =2n\CDeWW$ &\\
 & $\CDeLT=\CDeLTK+\CDeLTW$ & 
\eqref{eq:CbarLT-LT}\\
\hline
$\GbarL - \GL$ & 
$\CDeGLK =\CLT\cteDG\CL\CDeK\ga\de^\tau + \CLT\cteG\CDeLK +
\cteG\CL\CDeLTK $ & \\
& $\CDeGLW =\CLT\cteG\CDeLW + \cteG\CL\CDeLTW ,\quad 
\CDeGL= \CDeGLK + \CDeGLW  $ & 
\eqref{eq:CbarGL-GL}\\
\hline
$\barB - \B$ & 
$\CDeBK =(\sigmaB)^2\CDeGLK ,\quad
\CDeBW = (\sigmaB)^2\CDeGLW,\quad
\CDeB = \CDeBK + \CDeBW$ & 
\eqref{eq:CbarB-B}\\
\hline
$\barNO-\NO$ & 
$\CDeNOK =\cteJ\left(\CL\CDeBK + \CDeLK\sigmaB\right)
 + \cteDJ\CDeK\CL\sigmaB\ga\de^\tau $ & 
\\
 & $\CDeNOW=\cteJ\left(\CL\CDeBW + 
\CDeLW\sigmaB\right) ,\quad
 \CDeNO=\CDeNOK + \CDeNOW$ & 
\eqref{eq:CbarNO-NO}\\
\hline
$\barNOT -\NOT$ & 
$\CDeNOTK=\cteJT\left(\CLT\CDeBK + \sigmaB\CDeLTK\right) +
\sigmaB\CLT\cteDJT2n\CDeK\ga\de^\tau$ & 
\\
 & $\CDeNOTW =\cteJT\left(\CLT\CDeBW + \sigmaB\CDeLTW
\right), \quad
\CDeNOT = \CDeNOTK+\CDeNOTW$ & 
\eqref{eq:CbarNOT-NOT}\\
\hline
$\barhS-\hS$ & 
$\CDehSK=\sigmaNOT\cteOm\cteDzphiT\CDeNOK +
    \sigmaNOT\sigmaNO\left(\cteDOm\cteDzphiT +
 \cteOm\cteDDzphiT\right)\CDeK\ga\de^\tau +
\CDeNOTK\cteOm\cteDzphiT\sigmaNO$ & 
\\
 & $\CDehSW=\sigmaNOT\cteOm\cteDzphiT\CDeNOW +
    \CDeNOTW\cteOm\cteDzphiT\sigmaNO,\quad
\CDehS=\CDehSK + \CDehSW$ & 
\eqref{eq:CbarhS-hS}\\
\hline
$\aver{\bar\S}^{\!-1} - \aver{\S}^{\!-1}$ & 
$\CDeinvaverSK=\left(\sigmainvaverS\right)^2\CDehSK ,\quad
\CDeinvaverSW =\left(\sigmainvaverS\right)^2\CDehSW ,\quad
\CDeinvaverS = \CDeinvaverSK + \CDeinvaverSW$ & 
\eqref{eq:CbarinvS-invS}\Bspace{1.5}\\
\hline
$\Dela$ & 
$\CDela = \CDelaK + \CDelaW$ & 
\eqref{eq:CDela}\\
\hline
$\Deinvla$ & 
$\CDeinvla = (\sigmainvla)^2\CDela$ & 
\eqref{eq:Deinvla}\\
\hline
$\EKnew$ & 
$\CEK=\CEKKK + \CEKKW$ & 
\eqref{eq:CEKnew}\Bspace{1.5}\\
\hline
$\EWnew$ & 
$\CEW=\CEWKK + \CEWWW + \CEWKW$ & 
   \eqref{eq:CEWnew}\\
\hline  
$\E$ & 
$\CDe=\max\Bigg\{
      {\scriptstyle\CEsym\ga\de^{\tau}},\
      \tfrac{\Max{\ga^2\de^{2\tau},\CDeK}}{R-\norm{\K-\KO}_\rho}\ga\de^\tau,\
      \tfrac{d\CDeK}{\sigmaDteK - \norm{\DteK}_\rho}\ga\de^{\tau-1},\
   \tfrac{\ell\CDeK}{\sigmaDvpK -
\norm{\DvpK}_\rho}\ga\de^{\tau-1},$ & \Tspace{4.5} 
\\
& $\phantom{\CDe=\max\Big(}\tfrac{2n\CDeK}{\sigmaDteKT -
   \norm{\DteKT}_\rho}\ga\de^{\tau-1},\ 
   \tfrac{2n\CDeK}{\sigmaDvpKT - \norm{\DvpKT}_\rho}\ga\de^{\tau-1},\ 
\tfrac{\CDeW}{\sigmaW-\norm{\W}_\rho},\
\tfrac{2n\CDeW}{\sigmaWT-\norm{\WT}_\rho},$ & \\
& $\phantom{\CDe=\max\Big(}\tfrac{\CDeB}{\sigmaB-\norm{\B}_\rho},\
   \tfrac{\CDeNO}{\sigmaNO-\norm{\NO}_\rho},\
   \tfrac{\CDeNOT}{\sigmaNOT-\norm{\NOT}_\rho},\ 
\tfrac{\CDeinvaverS}{\sigmainvaverS -
\Abs{\aver{\S}^{\!-1}}},\
\tfrac{\CDela}{\sigmala- \abs{\la}}\ga\de^\tau,\
\tfrac{\CDeinvla}{\sigmainvla- \abs{\invla}}\ga\de^\tau
\Bigg\}$ & 
\eqref{eq:deffracC}\\[+2ex]
\hline

\end{longtable}
}
\egroup

\newpage

\bgroup
{
\def\arraystretch{1.5}
\begin{longtable}{|l l|}
   \caption{Constants in Theorem \ref{thm:KAM} from Section
   \ref{sec:convergence} } \label{tab:} \\
\hline
Constant & Label \\
\hline
\hline
$\CE=\Max{\CEK a^{2\tau},\CEW\ga\de}$ & 
\eqref{eq:CEj}\\
\hline
$\fracCDeK=\frac{a^{2\tau}}{a^{2\tau}-\ka}\CDeK$ &
\eqref{eq:CDeK}\\
\hline
$\fracCtU=\ga^2\de^{2\tau} + \fracCDeK$ & 
\eqref{eq:fracC}\\
\hline
$\fracCDteK=\frac{a^{2\tau-1}}{a^{2\tau-1} - \ka}d\CDeK$ & 
\eqref{eq:fracC}\\
\hline
$\fracCDvpK=\frac{a^{2\tau-1}}{a^{2\tau-1} - \ka}\ell\CDeK$ & 
\eqref{eq:fracC}\\
\hline
$\fracCDteKT=\frac{a^{2\tau-1}}{a^{2\tau-1} - \ka}2n\CDeK$ & 
\eqref{eq:fracC}\\
\hline
$\fracCDvpKT=\frac{a^{2\tau-1}}{a^{2\tau-1} - \ka}2n\CDeK$ & 
\eqref{eq:fracC}\\
\hline
$\fracCW=\frac{a^\tau}{a^\tau-\ka}\CDeW$ & 
\eqref{eq:fracC}\\
\hline
$\fracCWT=\frac{a^\tau}{a^\tau-\ka}2n\CDeW$ & 
\eqref{eq:fracC}\\
\hline
$\fracCDeB:=\frac{a^\tau}{a^\tau-\ka}\CDeB$ & 
\eqref{eq:fracC}\\
\hline
$\fracCNO=\frac{a^\tau}{a^\tau-\ka}\CDeNO$ & 
\eqref{eq:fracC}\\
\hline
$\fracCNOT=\frac{a^\tau}{a^\tau-\ka}\CDeNOT$ & 
\eqref{eq:fracC}\\
\hline
$\fracCDeinvaverS=\frac{a^\tau}{a^\tau-\ka}\CDeinvaverS$ & 
\eqref{eq:fracC}\\
\hline
$\fracCDela=\Cla\frac{a^{2\tau}}{a^{2\tau}-\ka}\CDela$ & 
\eqref{eq:fracC}\\
\hline
$\fracC=\max\Big\{a^{2\tau}\CE,\ \fracCDe\ga\de^\tau
\Big\}$  &
\eqref{eq:fracC}\\[+0.5ex]
\hline
$\fracCDe=\max\Bigg\{
   {\scriptstyle\CEsym\ga\de^{\tau}},\
   \tfrac{\fracCtU}{R}\ga\de^{\tau},\
\tfrac{\fracCDteK}{\sigmaDteK -
\norm{\DteK}_{\rho}}\ga\de^{\tau-1}, \
  \tfrac{\fracCDvpK}{\sigmaDvpK -
\norm{\DvpK}_{\rho}}\ga\de^{\tau-1},$ & \Tspace{4.5}
\\
$\phantom{\fracCDe=\max\Bigg(}\tfrac{\fracCDteKT}{\sigmaDteKT -
   \norm{(\DteK)^\ttop}_{\rho}}\ga\de^{\tau-1},\ 
   \tfrac{\fracCDvpKT}{\sigmaDvpKT -
   \norm{(\DvpK)^\ttop}_{\rho}}\ga\de^{\tau-1},\ 
   \tfrac{\fracCW}{\sigmaW - \norm{\W}_{\rho}},\ 
   \tfrac{\fracCWT}{\sigmaWT -
   \norm{\W^\ttop}_{\rho}},$ & 
\\
$\phantom{\fracC=\max\Bigg(}\tfrac{\fracCDeB}{\sigmaB -
   \norm{\B}_{\rho}},\
\frac{\fracCNO}{\sigmaNO -
   \norm{\NO}_{\rho}},\
   \tfrac{\fracCNOT}{\sigmaNOT -
   \norm{\NOT}_{\rho}},\
      \tfrac{\fracCDeinvaverS}{\sigmainvaverS -
   \Abs{\aver{\S}^{\!-1}}},\
\frac{\fracCDela}{\sigmala-\abs{\la}}\ga\de^\tau,\
\frac{\fracCDeinvla}{\sigmainvla-\abs{\invla}}\ga\de^\tau,\
\Bigg\}$ & 
\eqref{eq:fracC}\\
\hline
\end{longtable}
}
\egroup

\restoregeometry

\end{document}